\documentclass[12pt,twoside]{amsart}

\usepackage{geometry}
\usepackage{amssymb, amsmath, amsthm, amscd}
\usepackage{enumerate}
\usepackage{mathrsfs}
\usepackage{bm}

\usepackage{graphicx}
\usepackage{hhline}
\usepackage[all]{xy}

\usepackage{xcolor}

\usepackage{url}

\usepackage{iftex}

\ifPDFTeX

  \usepackage[pdfencoding=auto, unicode]{hyperref}
\else
  \usepackage[dvipdfmx, unicode]{hyperref}
\fi

\hypersetup{
    colorlinks=false,%true
    linkcolor=blue,
    citecolor=blue,
    urlcolor=blue,
    bookmarksnumbered=true,
    pdfstartview=FitH
}

\title{Notes on acceptable bundles I} 
\author{Osamu Fujino, Taro Fujisawa, and Takashi Ono}
\date{2026/4/8, version 0.59}
\address{Department of 
Mathematics, Graduate School of Science, 
Kyoto University, Kyoto 606-8502, Japan}
\email{fujino@math.kyoto-u.ac.jp}
\address{Department of Mathematics and Data Science,
Center for Liberal Arts and Sciences,
Tokyo Denki University, Tokyo, Japan}
\email{fujisawa@mail.dendai.ac.jp}
\address{Research Institute for Mathematical Sciences, Kyoto University, Kyoto 606-8502, Japan}
\email{takashio@kurims.kyoto-u.ac.jp}
\DeclareMathOperator{\re}{Re}
\DeclareMathOperator{\ord}{ord}
\DeclareMathOperator{\rank}{rank}
\DeclareMathOperator{\supp}{supp}
\DeclareMathOperator{\Area}{Area}
\DeclareMathOperator{\Hom}{Hom}
\DeclareMathOperator{\Par}{\mathcal{P}\!\it{ar}}
\DeclareMathOperator{\Parr}{\mathcal{P}{\overset{\mathrm{red}}{\!\it{ar}}}}
\newtheorem{thm}{Theorem}[section]
\newtheorem{lem}[thm]{Lemma}
\newtheorem{cor}[thm]{Corollary}
\newtheorem{prop}[thm]{Proposition}

\newtheorem*{claim}{Claim}

\theoremstyle{definition}
\newtheorem{defn}[thm]{Definition}
\newtheorem{rem}[thm]{Remark}
\newtheorem{ex}[thm]{Example}
\newtheorem*{ack}{Acknowledgments}  
\newtheorem{step}{Step}
\newtheorem{say}[thm]{}

\makeatletter

\@addtoreset{equation}{section}
\makeatother

\begin{document}

\begin{abstract} 
The notion of acceptable bundles plays a fundamental role in the Simpson--Mochizuki theory. This paper presents a detailed study of acceptable bundles on a punctured disk. In addition to its expository aspects, we introduce a new invariant and provide arguments that differ from those of Simpson and Mochizuki.
\end{abstract}

\maketitle 

\tableofcontents 

\section{Introduction}\label{p-sec1}

The notion of acceptable bundles plays a fundamental role in the Simpson--Mochizuki theory; see, for example, \cite{simpson1}, 
\cite{simpson2}, \cite{mochizuki1}, \cite{mochizuki2}, 
\cite{mochizuki3}, \cite{mochizuki4}, \cite{mochizuki5}, and 
\cite{mochizuki6}. The present paper and the subsequent paper 
\cite{ffo} provide detailed studies of acceptable bundles 
on a punctured disk and on a partially punctured polydisk, respectively. 
While these papers are primarily expository in nature, they also contain new arguments that differ from those of Simpson and Mochizuki. 

Takuro Mochizuki gives a general account 
of acceptable bundles in a broad setting 
in \cite[Chapter 21, Acceptable Bundles]{mochizuki4}. 
However, the primary focus there is on higher-dimensional 
generalizations of the results of Simpson (\cite{simpson1} 
and \cite{simpson2}), and the treatment of the most 
basic case, namely, acceptable bundles on a punctured disk, 
is relatively brief. 
The present paper is intended to fill this gap. 
In particular, we give a detailed study of acceptable bundles on a punctured disk and introduce new tools that will be used in subsequent work. 
In the subsequent paper \cite{ffo}, we study acceptable bundles on a partially punctured polydisk, building on the results of the present paper.

\medskip

Let $E$ be a holomorphic vector bundle over 
$\Delta^* := \{z \in \mathbb{C} \mid 0 < |z| < 1\}$, 
and let $h$ be a smooth Hermitian metric on $E$.  
We denote the curvature form of the Chern 
connection associated with $(E, h)$ by $\sqrt{-1}\Theta_h(E)$, 
which is a smooth $\Hom(E, E)$-valued $(1,1)$-form on $\Delta^*$.

We consider the Poincar\'e metric on $\Delta^*$ given by
\[
\omega_P := \frac{\sqrt{-1}\,dz \wedge d\overline{z}}{|z|^2 (-\log|z|^2)^2}.
\]
The induced metric on $\Hom(E, E)$ by 
$h$ is also denoted by $h$, whenever there is no risk of confusion.

Let us recall the definition of acceptable 
vector bundles on $\Delta^*$ in the sense of Mochizuki (see 
\cite{mochizuki1}, \cite{mochizuki2}, 
\cite{mochizuki3}, and \cite[Chapter 21]{mochizuki4}).

\begin{defn}[Acceptable bundles, see Definition \ref{p-def2.1}]\label{p-def1.1}
Let $(E, h)$ be a Hermitian holomorphic vector bundle on $\Delta^*$.  
We say that $(E, h)$ is an \emph{acceptable vector bundle} (in 
the sense of Mochizuki) if there exists a constant $C > 0$ such that
\[
|\Theta_h(E)|_{h, \omega_P} \leq C \quad \text{on } \Delta^*,
\]
where $|\bullet|_{h, \omega_P}$ denotes the pointwise 
norm of $\bullet$ with respect to the Hermitian metric induced by $h$ and $\omega_P$. 
\end{defn}

Although Simpson treats a more general setting (see \cite[Section 10]{simpson1} 
and \cite[Section 3]{simpson2}), 
in this paper we adopt the above definition of acceptable vector bundles.

\begin{defn}[Prolongation by increasing orders, 
see Definition \ref{p-def2.3}]\label{p-def1.2}
Let $(E, h)$ be an acceptable vector bundle on $\Delta^*$, 
and let $a$ be any real number.
For any open subset $U \subset \Delta$, we define
\[
{}_a E(U) := \left\{ f \in 
E(U \setminus \{0\}) \,\middle|\, |f|_h = O\left(\frac{1}{|z|^{a+\varepsilon}}\right) 
\text{ for every } \varepsilon > 0 \right\}, 
\] 
where $|f|_h$ denotes the norm of $f$ with respect 
to the Hermitian metric $h$. 
Then we obtain a sheaf of \( \mathcal{O}_\Delta \)-modules, 
denoted by \( {}_a E \). 
When $a = 0$, we usually write ${}^\diamond\! E := {}_0 E$.
\end{defn}

The following foundational result is due to Simpson (see \cite{simpson1} and 
\cite{simpson2}):

\begin{thm}[Simpson, 
see \cite{simpson1} and \cite{simpson2}]\label{p-thm1.3}
Let $(E, h)$ be an acceptable vector 
bundle on $\Delta^*$. Then ${}_a E$ is a 
holomorphic vector bundle for every $a \in \mathbb{R}$.
\end{thm}

More precisely, Simpson asserts the 
coherence of ${}_a E$ in a slightly more general setting.
Furthermore, he states that the desired coherence 
follows from the theory 
of Cornalba--Griffiths \cite{cornalba-griffiths}, with a minor modification. 
For details, see the discussion on pages 909--910 of \cite{simpson1}.

The next corollary follows easily from the definition of ${}_a E$ and Theorem \ref{p-thm1.3}:

\begin{cor}[see Section \ref{p-sec7}]\label{p-cor1.4}
In the setting of Theorem \ref{p-thm1.3}, for $a, b \in \mathbb{R}$, we have:
\begin{itemize}
\item[(i)] ${}_a E$ is locally free;
\item[(ii)] ${}_a E \subset j_*E$ and 
${}_a E|_{\Delta^*} = E$, where 
$j\colon \Delta^* \hookrightarrow \Delta := \{z \in \mathbb{C} \mid |z| < 1\}$;
\item[(iii)] ${}_a E \subset {}_b E$ if $a \leq b$;
\item[(iv)] ${}_{a+1}E = {}_a E \otimes \mathcal{O}_\Delta([0])$;
\item[(v)] ${}_{a+\varepsilon}E = {}_a E$ for 
all sufficiently small $\varepsilon > 0$;
\item[(vi)] The set $\{a \in \mathbb{R} \mid {}_a E / {}_{<a} E \neq 0\}$ 
is discrete in $\mathbb{R}$, 
where ${}_{<a} E := \bigcup_{b < a} {}_b E \subset j_*E$.
\end{itemize}
\end{cor}

Thus, we can regard ${}_\ast E 
:= \left( {}_a E \mid a \in \mathbb{R} \right)$ 
as a filtered bundle over $E$ in the sense of Mochizuki 
(see Section \ref{p-sec8}).

To prove Theorem \ref{p-thm1.3}, we first 
establish the following special case, which plays a crucial role in the overall proof.

\begin{prop}[Proposition \ref{p-prop4.1}]\label{p-prop1.5} 
Let $L$ be a holomorphic line bundle 
on $\Delta^*$ and let $h$ be a Hermitian metric on $L$ such that
\[
-C \cdot \omega_P \leq \sqrt{-1} \Theta_h(L) \leq C \cdot \omega_P
\]
holds on $\Delta^*$ for some constant $C > 0$.
That is, writing
\[
\sqrt{-1} \Theta_h(L) = f(z) \cdot \omega_P,
\]
we have $|f(z)| \leq C$ on $\Delta^*$. 
Then ${}_a L$ is a holomorphic line bundle for every $a \in \mathbb{R}$.
\end{prop}

Note that a more precise description of ${}_a L$ is provided in Theorem \ref{p-thm4.4}.  
The authors believe that the explicit formulation given in this paper is new. 

\begin{thm}[Theorem \ref{p-thm4.4}]\label{p-thm1.6}
Let $(L, h)$ be an acceptable line bundle on $\Delta^*$. 
By taking a suitable trivialization 
\[
(L, h)\simeq \left(\mathcal O_{\Delta^*}, |\cdot|^2e^{-2\varphi}\right), 
\] 
we have the following properties. 
\begin{itemize}
\item[(i)] The limit 
\[
\gamma :=\lim_{z\to 0} \frac{\varphi(z)}{\log |z|}\in \mathbb R
\] 
exists. 
\item[(ii)] Let $f$ be a holomorphic function on $\Delta(0, r)^*$ for some $0<r<1$, 
where $\Delta(0, r)^*:=\{z\in \mathbb C\mid 0<|z|<r\}$. 
Then $f\in ({}_a L)_0$ holds for some $a\in \mathbb R$ if and only if 
$f$ is meromorphic at $0$, where $({}_a L)_0$ 
denotes the stalk of ${}_a L$ at $0 \in \Delta$. 
\item[(iii)] Let $f$ be a meromorphic function on some open neighborhood 
of $0$ and let $a$ be any real number. 
Then $f\in ({}_a L)_0$ holds if and only if 
\begin{equation}\label{p-eq1.1} 
\lim _{z\to 0} \frac{\log \left(|f|e^{-\varphi}\right)}{\log |z|} \geq -a. 
\end{equation} 
Note that 
\[
\lim _{z\to 0} \frac{\log \left(|f|e^{-\varphi}\right)}{\log |z|}=\ord _0 f-\gamma 
\] holds. 
Therefore, \eqref{p-eq1.1} is equivalent to 
\[
\ord _0 f\geq -\lfloor a -\gamma\rfloor. 
\]
\item[(iv)] Let $f$ be a meromorphic function on some open 
neighborhood of $0$ and let $a$ be any real number. 
Then $f\not\in ({}_a L)_0$ holds if and only if 
\begin{equation}\label{p-eq1.2}
\lim _{z\to 0} \frac{\log \left(|f|e^{-\varphi}\right)}{\log |z|}<-a. 
\end{equation} 
Note that \eqref{p-eq1.2} implies that 
\[
|f|e^{-\varphi} >\frac{1}{|z|^a}
\] 
holds on some small open neighborhood of $0$. 
\end{itemize}
\end{thm}  

The following corollaries 
follow directly from the description of ${}_a L$ in the proof of Proposition \ref{p-prop1.5}:

\begin{cor}[Duality for line bundles, see Corollary \ref{p-cor4.3}]\label{p-cor1.7}
In Proposition \ref{p-prop1.5}, we have
\[
({}_a L)^\vee = {}_{-a + 1 - \varepsilon}(L^\vee)
\]
for all sufficiently small $\varepsilon > 0$.
\end{cor}

\begin{cor}\label{p-cor1.8}
Let $L$ be a holomorphic line bundle on 
$\Delta^*$, and let $h$ be a flat Hermitian metric on $L$. Then
\[
(L, h) \simeq \left( \mathcal{O}_{\Delta^*}, 
\frac{|\cdot|^2}{|z|^{2c}} \right)
\]
for some $c \in \mathbb{R}$.
\end{cor}

In this paper, we introduce a new invariant 
$\gamma({}_a E)$ for studying the structure of 
${}_a E$, and establish the following result:

\begin{thm}[see Definition \ref{p-def7.4}, Corollary \ref{p-cor7.6}, 
Theorem \ref{p-thm7.13}, and Theorem \ref{p-thm12.3}]\label{p-thm1.9}
Let $(E, h)$ be an acceptable vector 
bundle on $\Delta^*$ with $\rank E 
= r$, and let $\{v_1, \ldots, v_r\}$ be a local frame of ${}_a E$ near the origin. Define
\[
\gamma({}_a E) := -\frac{1}{2} \liminf_{z \to 0} \frac{\log \det H(h, \bm v)}{\log |z|},
\]
where $H(h, \bm v)$ is the $r \times r$ matrix 
$\left( h(v_i, v_j) \right)$. Then 
$\gamma({}_a E)$ is a well-defined real-valued invariant of ${}_a E$.

Furthermore, if we let
\[
\Par_a(E, h) =: \{b_1, \ldots, b_r\},
\]
then we have
\[
\gamma({}_a E) = -\frac{1}{2} 
\lim_{z \to 0} \frac{\log \det H(h, \bm v)}{\log |z|} = \sum_{i=1}^r b_i.
\] 
For the precise definition of the parabolic 
weights $\Par_a(E, h)$, see \ref{p-say7.11} below.

Note that if we define 
\[
\{\lambda_1, \ldots, \lambda_k\} 
:= \{ \lambda \in (a-1, a] \mid {}_\lambda E / {}_{<\lambda} E \ne 0 \}
\]
with $\lambda_i \ne \lambda_j$ for $i \ne j$, then
\[
\sum_{i=1}^r b_i = \sum_{i=1}^k \lambda_i 
\dim_{\mathbb{C}} \left( {}_{\lambda_i} E / {}_{<\lambda_i} E \right).
\]
\end{thm}

This theorem plays a central role in our analysis.
We emphasize that the most 
technically challenging part of this paper is the proof of the identity 
\[\gamma ({}_a E) =
\sum 
_{i=1}^rb_i.
\] 

\begin{thm}[Determinant bundles, see Theorem \ref{p-thm7.5}]\label{p-thm1.10}
Let $(E, h)$ be an acceptable vector bundle on $\Delta^*$. Then 
the determinant bundle $(\det E, \det h)$ is an acceptable 
line bundle on $\Delta^*$, and
\[
\det({}_a E) = {}_{\gamma({}_a E)} \det E
\]
holds for every $a \in \mathbb{R}$.
\end{thm}

The proof of Theorem \ref{p-thm1.10} closely follows 
that of Proposition \ref{p-prop1.5} (see Proposition \ref{p-prop4.1}), 
once the well-definedness of $\gamma({}_a E)$ is established. 
By using $\gamma ({}_a L)$, we can reformulate Corollary 
\ref{p-cor1.7} as follows. 

\begin{lem}[Duality for line bundles, see Lemma \ref{p-lem13.1}]\label{p-lem1.11}
Let $(L, h)$ be an acceptable line bundle on $\Delta^*$.  
Let $a\in \mathbb{R}$ be any real number.  
Then we have ${}_a L = {}_{\gamma({}_a L)} L$, 
and $\Par_a(L, h) = \{ \gamma({}_a L) \}$.   

Moreover, if $0 < \varepsilon \ll 1$, then  
\[
\gamma({}_{-a + 1 - \varepsilon}(L^\vee)) = -\gamma({}_a L).
\]  

In particular, the following equality holds:
\[
\left({}_{\gamma({}_a L)} L\right)^\vee = {}_{-\gamma({}_a L)}(L^\vee).
\]
\end{lem}

In contrast, the proofs of the following 
theorems, namely Theorem \ref{p-thm1.12}, 
Theorem \ref{p-thm1.13}, and 
Theorem \ref{p-thm1.14}, rely on the 
equality $\gamma({}_a E) = \sum_{i=1}^r b_i$ 
in Theorem \ref{p-thm1.9}, and are therefore considerably more involved.

\begin{thm}[Dual bundles, see Theorem \ref{p-thm13.2}]\label{p-thm1.12} 
Let $(E, h)$ be an acceptable vector bundle on $\Delta^*$, 
and let $a$ be any real number. Then,
\[
\left({}_a E\right)^\vee = {}_{-a + 1 - \varepsilon}\left(E^\vee\right)
\]
holds for any sufficiently small $\varepsilon > 0$. 

Moreover, let $\{v_1, \ldots, v_r\}$ be a local frame of ${}_a E$ near the origin, 
compatible with the parabolic filtration, such that 
$v_i \in {}_{b_i} E \setminus {}_{<b_i} E$ for each $i$. 
For each $i$, define
\[ 
v_i^\vee := (-1)^{i-1} \, v_1 \wedge 
\cdots \wedge v_{i-1} \wedge v_{i+1} \wedge 
\cdots \wedge v_r \otimes (v_1 \wedge \cdots \wedge v_r)^{\otimes -1}.
\] 
Then $\{v_1^\vee, \ldots, v_r^\vee\}$ forms a local frame of ${}_{-a + 1 - \varepsilon}(E^\vee)$ 
near the origin, compatible with the parabolic filtration, such that 
\[
v_i^\vee \in {}_{-b_i}(E^\vee) \setminus {}_{<-b_i}(E^\vee)
\]
for each $i$. In particular, we have 
\[
\Par _a(E, h)=\{b_1, \ldots, b_r\} \quad 
\text{ and }\quad \Par _{-a+1-\varepsilon} (E^\vee, h^\vee)=\{ -b_1, \ldots, -b_r\}. 
\]
\end{thm}

As an immediate consequence of 
Theorem \ref{p-thm1.12}, we have: 

\begin{thm}[Weak norm estimate, see Theorem \ref{p-thm13.3}]\label{p-thm1.13}
Let $\{v_1, \ldots, v_r\}$ be a local frame 
of ${}_a E$ around the origin, compatible with the parabolic filtration, such that 
\[
v_i \in {}_{b_i} E \setminus {}_{<b_i} E \quad \text{for every } i.
\] 
We define  
\[
H(h, \bm v') :=\left( h(v_i\cdot |z|^{b_i}, v_j\cdot |z|^{b_j})\right)_{i, j}. 
\]
Then there exist positive constants $C$ and $M$ such that
\[
C^{-1}(-\log |z|)^{-M} I_r \leq H(h, \bm{v}')(z) \leq C(-\log |z|)^M I_r
\]
holds in a neighborhood of the origin, where $I_r$ is the identity matrix of size $r$.
\end{thm}

\begin{thm}[Tensor products, see Theorem \ref{p-thm16.2}]\label{p-thm1.14}
Let $(E_1, h_1)$ and $(E_2, h_2)$ 
be acceptable vector bundles on $\Delta^*$. 
Then the tensor product bundle $(E_1 \otimes E_2, h_1 \otimes h_2)$ is also acceptable, and
\[
{}_a(E_1 \otimes E_2) = \sum_{a_1 + a_2 \leq a} {}_{a_1} E_1 \otimes {}_{a_2} E_2
\]
holds for any $a \in \mathbb{R}$.
\end{thm}

Finally, we remark that significant 
effort has been made to ensure that this paper is as self-contained as possible.

\medskip 

This paper focuses solely 
on acceptable bundles over the punctured 
disk and does not address any applications. 
There is already extensive 
literature on related topics; see, 
for example, \cite{biquard1}, \cite{biquard2}, \cite{bb}, 
\cite{sabbah-schnell1}, and \cite{sabbah-schnell2}. 
Our selection of references 
reflects the authors' preferences 
and perspective. We apologize for 
omitting many important works and 
refer interested readers to the broader literature.

\medskip

We now outline the organization of the present paper. 
In Section \ref{p-sec2}, we collect some basic definitions 
and state a few elementary properties that follow directly from them. 
In Section \ref{p-sec3}, we prove some preliminary 
lemmas concerning harmonic and holomorphic functions on a punctured disk. 
Section \ref{p-sec4} is 
devoted to the proof of Proposition \ref{p-prop1.5} (see Proposition \ref{p-prop4.1}), 
where we describe the prolongation of acceptable line bundles by 
increasing orders. To the best of the authors' knowledge, 
this treatment is new. 
In Section \ref{p-sec5}, we briefly 
discuss $\overline{\partial}$-equations and derive a growth estimate via the $L^2$-method. 
In Section \ref{p-sec6}, we prove 
Theorem \ref{p-thm1.3}, establishing 
the prolongation of acceptable vector bundles by increasing orders. 
In Section \ref{p-sec7}, we 
introduce a new invariant and prove some 
fundamental properties of prolongations of acceptable bundles. 
In Section \ref{p-sec8}, we briefly 
review the framework of filtered bundles for later use. 
Section \ref{p-sec9} collects 
several elementary inequalities, which will play a crucial 
role in the subsequent section. 
In Section \ref{q-sec10}, we establish Simpson's key lemma, 
which is one of the main ingredients in the proof of 
Theorem \ref{p-thm12.3} given in Section \ref{p-sec12}. 
In Section \ref{p-sec11}, we study the behavior of acceptable bundles 
via cyclic covers. 
Section \ref{p-sec12} is devoted to the proof of Theorem \ref{p-thm12.3}, 
which is one of the most technically involved results in this paper. 
In Section \ref{p-sec13}, we investigate the prolongation of dual vector bundles. 
In Section \ref{p-sec14}, we present some examples of filtered bundles introduced 
in Section \ref{p-sec8}. 
In Section \ref{p-sec15}, we return to the study of 
dual bundles, now within the framework of filtered bundles. 
In Section \ref{p-sec16}, we examine 
the prolongation of tensor products of acceptable bundles, 
again in the context of filtered bundles. 
Finally, in Section \ref{p-sec17}, we 
study Hom bundles from the perspective of filtered bundles.

\medskip 

While certain parts of the exposition 
may be new, and others have been simplified 
or clarified, we believe that all essential 
results are already contained, perhaps implicitly, 
within the substantial works of Simpson and Mochizuki 
(see \cite{simpson1}, \cite{simpson2}, 
\cite{mochizuki1}, \cite{mochizuki2}, 
\cite{mochizuki3}, \cite{mochizuki4}, 
\cite{mochizuki5}, \cite{mochizuki6}, and so on). 
We have cited the most relevant 
references to their works, though 
we do not aim to exhaustively list all related material. 
We nevertheless hope that the present 
paper contributes to making their profound 
and extensive theories more accessible.

\begin{say}[Convention]\label{p-say1.15}
Let $\mathcal{F}$ be a sheaf on a topological space $X$.  
Unless explicitly stated otherwise, we write  
$f \in \mathcal{F}$ to indicate that $f$ is a local section  
$f \in \mathcal{F}(U)$ over some open subset $U \subset X$.  
%The specific domain $U$ will either be clear from the context  
%or explicitly stated when necessary.

In this paper, we do not distinguish between holomorphic vector bundles  
on a complex manifold $X$ and the corresponding locally free  
$\mathcal{O}_X$-modules. These are treated as equivalent  
unless stated otherwise.
\end{say}

\begin{ack}\label{p-ack}
The first author was partially 
supported by JSPS KAKENHI Grant Numbers 
JP20H00111, JP21H00974, JP21H04994, JP23K20787. 
The third author was supported by 
JSPS KAKENHI Grant Number JP24KJ1611. 
The authors are deeply grateful 
to Professors Carlos Simpson and 
Takuro Mochizuki for kindly 
answering their questions and for generously 
sharing their private notes (\cite{simpson3} and \cite{mochizuki7}). 
They also wish to thank Hitoshi Fujioka 
and Natsuo Miyatake for helpful discussions. 
They are very thankful to Professors Philip Boalch, Ya Deng, and 
Takahiro Saito for their comments and for sharing valuable information 
on related topics.  
Finally, they are very grateful to Professors 
Hiromichi Takagi, Shin-ichi Matsumura, and 
Takeo Ohsawa for their valuable comments and support.
\end{ack}

\section{Preliminaries}\label{p-sec2}

In this paper, we will almost always 
work over either the punctured disk 
$\Delta^* := \{ z \in \mathbb{C} \mid 0 < |z| < 1 \}$ 
or the unit disk $\Delta := \{ z \in \mathbb{C} \mid |z| < 1 \}$.
Let 
\[
\omega_P := \frac{\sqrt{-1}\,dz \wedge d\overline{z}}{|z|^2 (-\log|z|^2)^2}
\] 
denote the Poincar\'e metric on $\Delta^*$.
Then the pair $(\Delta^*, \omega_P)$ defines a K\"ahler manifold.

Let us recall the definition of {\em{acceptable bundles}} 
on a punctured disk $\Delta^*$ in the sense of Mochizuki. 
As already mentioned in Section \ref{p-sec1}, Simpson treats a more general setting 
in \cite{simpson1} and \cite{simpson2}. 

\begin{defn}[Acceptable bundles]\label{p-def2.1}
Let $E$ be a holomorphic vector bundle on the punctured disk $\Delta^*$, and let $h$ be a Hermitian metric on $E$.  
Then $(E, h)$ admits a {\em Chern connection} $D = D' + \overline{\partial}$, whose curvature form is given by  
\[
\sqrt{-1}\Theta_h(E) := \sqrt{-1} D^2.  
\]
This is a smooth $(1,1)$-form on $\Delta^*$ with values in $\Hom(E, E)$. 

We use the same notation $h$ 
to denote the induced Hermitian metric on 
$\Hom(E, E)$, whenever there is no risk of confusion.

We say that $(E, h)$ is an \emph{acceptable bundle} 
on $\Delta^*$ if the norm of $\sqrt{-1} \Theta_h(E)$ is bounded on $\Delta^*$, 
that is, there exists a constant $C > 0$ such that
\[
|\Theta_h(E)|_{h, \omega_P} \leq C \quad \text{on } \Delta^*,
\]
where $|\bullet|_{h, \omega_P}$ denotes the pointwise 
norm of $\bullet$ with respect to the Hermitian metric 
$h$ and the Poincar\'e metric $\omega_P$.
\end{defn}

Lemma \ref{p-lem2.2} easily follows from the definition. 

\begin{lem}\label{p-lem2.2}
Let $(E, h)$ be an acceptable vector bundle on $\Delta^*$. 
Then the dual bundle $(E^\vee, h^\vee)$ and the 
determinant line bundle $(\det E, \det h)$ are also acceptable. 

Let $(E_1, h_1)$ and $(E_2, h_2)$ be acceptable 
vector bundles on $\Delta^*$. 
Then the tensor product $(E_1 \otimes E_2, h_1 \otimes h_2)$ 
and the Hom bundle $(\Hom(E_1, E_2), h_1^\vee \otimes h_2)$ are acceptable. 
\end{lem}

\begin{proof}[Proof of Lemma \ref{p-lem2.2}]
Since $\Theta_{h^\vee}(E^\vee) = -\Theta_h(E)$ and 
\[
\Theta_{h_1 \otimes h_2}(E_1 \otimes E_2) 
= \Theta_{h_1}(E_1) \otimes \mathrm{Id}_{E_2} + 
\mathrm{Id}_{E_1} \otimes \Theta_{h_2}(E_2),
\]
it follows that both $(E^\vee, h^\vee)$ and 
$(E_1 \otimes E_2, h_1 \otimes h_2)$ are acceptable. 
Using the natural identification $\Hom(E_1, E_2) = E_1^\vee \otimes E_2$, 
we see that the Hom bundle $\Hom(E_1, E_2)$ 
is also acceptable. 
Note that $\det E$ is a direct summand of $E^{\otimes \rank E}$. 
Hence, $(\det E, \det h)$ is acceptable. 
This completes the proof of Lemma \ref{p-lem2.2}.
\end{proof}

The main object of this paper is the prolongation by 
increasing orders. 

\begin{defn}[Prolongation by increasing orders]\label{p-def2.3}
Let $(E, h)$ be an acceptable vector bundle on $\Delta^*$ and let 
$a$ be any real number. 
For any open subset $U$ of $\Delta$, 
we put 
\[
{}_aE(U):=\left\{f\in 
E(U\setminus \{0\})\, 
\middle|\,  |f|_h=O\left(\frac{1}{|z|^{a+\varepsilon}}\right) \text{for 
every $\varepsilon$}\right\},  
\] 
where $|f|_h$ denotes the norm of $f$ with 
respect to the Hermitian metric $h$. 
Then we obtain a sheaf of \( \mathcal{O}_\Delta \)-modules, 
denoted by \( {}_a E \). 
When $a=0$, we usually use ${}^\diamond\! E$ to denote 
${}_0E$. 
\end{defn}

Let us briefly recall the positivity of vector bundles. 
For details, see, for example, \cite[Chapter 10]{demailly1} 
and \cite[Chapter VII, \S 6 Positivity Concepts for Vector Bundles]{demailly}. 

\begin{defn}[Positivity of vector bundles]\label{p-def2.4}
Let $X$ be a complex manifold of dimension one, that is, $\dim X = 1$.  
Let $E$ be a holomorphic vector bundle on $X$, and let $h$ be a Hermitian metric on $E$.  
Let $D$ denote the Chern connection of $(E, h)$, and define the curvature form by
\[
\sqrt{-1} \Theta_h(E) := \sqrt{-1} D^2
\] 
as before. 
Then the curvature form $\sqrt{-1}\Theta_h(E)$ and the metric $h$ induce 
a Hermitian form $\theta_E$ on $T_X \otimes E$.

If $\theta_E$ is positive definite, positive semi-definite, negative, or negative semi-definite,  
then we say that $(E, h)$ (or equivalently, 
$\sqrt{-1}\Theta_h(E)$) is  
\emph{Nakano positive}, \emph{Nakano semipositive}, 
\emph{Nakano negative}, or \emph{Nakano seminegative}, respectively.

Since $\dim X = 1$, Nakano (semi)positivity and (semi)negativity are equivalent to  
\emph{Griffiths} (semi)positivity and (semi)negativity, respectively.

In this paper, we sometimes omit the 
terms \lq\lq Nakano\rq\rq \ and \lq\lq Griffiths\rq\rq \ since we are working 
in dimension one.
\end{defn}

The property established in 
Lemma \ref{p-lem2.5} below is a fundamental feature of 
acceptable bundles.
In fact, it may be said that this is the only 
property of acceptable bundles needed in this paper.

\begin{lem}\label{p-lem2.5}
Let $(E, h)$ be an acceptable vector bundle on $\Delta^*$ such that 
\[
|\Theta_h(E)|_{h, \omega_P} \leq C
\]
holds on $\Delta^*$. Then we have
\[
-C \omega_P \otimes \mathrm{Id}_E  \leq_{\mathrm{Nak}}  
\sqrt{-1} \Theta_h(E) \leq_{\mathrm{Nak}} 
C \omega_P \otimes \mathrm{Id}_E.
\]
Here, $A \leq_{\mathrm{Nak}} B$ means that the Hermitian 
form on $T_{\Delta^*} \otimes E$ induced by $B - A$ and $h$ is Nakano semipositive.
\end{lem}

\begin{proof}[Proof of Lemma \ref{p-lem2.5}]
For any $x\in \Delta^*$, we take 
a local coordinate $w$ centered at $x$ such that 
\[
\omega_P=\sqrt{-1} dw \wedge d\overline {w} 
\] 
around $x$. Let $\{e_1, \ldots, e_r\}$ be a local 
holomorphic frame of $E$, which is orthonormal at $x$. 
Let $\{e^1, \ldots, e^r\}$ be its 
dual in $E^\vee$. We write 
\[
\sqrt{-1} \Theta_h(E)=R^\beta_\alpha dw\wedge d\overline{w} \otimes 
e^\alpha \otimes e_\beta
\] 
around $x$. 
We put $R_{\alpha \overline \beta}=\sum _\gamma h_{\gamma 
\overline \beta} R^\gamma _\alpha$, where 
$h_{\gamma \overline\beta}:=h(e_\gamma, e_\beta)$. 
Since $(h_{\gamma \overline \beta})$ is the identity matrix at $x$, 
$R_{\alpha \overline \beta}(x)=R^\beta_\alpha (x)$ holds. 
By assumption, 
\[
\sum_{\alpha, \beta} |R_{\alpha \overline \beta}(x)|^2 
=|\Theta_h(E)(x)|^2_{h, \omega_P} \leq C^2. 
\] 
For any $u=\sum _{\alpha} u^\alpha \frac{\partial}{\partial w} 
\otimes e_\alpha$, by using the Cauchy--Schwarz inequality twice, 
\[
\begin{split}
\left|\sum _{\alpha, \beta} R_{\alpha \overline \beta}(x) u^\alpha \overline
{u^\beta}\right|^2&\leq 
\left( \sum _\beta \left|\sum _\alpha  R_{\alpha 
\overline \beta} (x) u^\alpha\right|^2\right) 
\left(\sum _\beta |\overline {u^\beta}|^2\right)\\ 
&\leq \left(\sum _\beta \left(\sum _\alpha |R_{\alpha 
\overline\beta}(x)|^2\right)\left(\sum _\alpha |u^\alpha|^2\right)\right)
\left(\sum _\beta |\overline{u^\beta}|^2\right)
\\ 
&= |u|^4_{h, \omega_P}\cdot \sum _{\alpha, \beta} |R_{\alpha \overline \beta} (x)|^2
\\ 
&\leq |u|^4_{h, \omega_P} \cdot C^2.  
\end{split}
\] 
This implies that 
\[
-C |u|^2_{h, \omega_P} \leq \sum _{\alpha, \beta} 
R_{\alpha\overline \beta} (x) u^\alpha \overline {u^\beta}
\leq C|u|^2_{h, \omega_P}. 
\] 
This is what we wanted. We finish the proof of Lemma \ref{p-lem2.5}. 
\end{proof}

\begin{rem}\label{p-rem2.6}
Although in Lemma \ref{p-lem2.5} 
we considered only the case over the punctured disk, 
the same statement holds over K\"ahler manifolds of arbitrary dimension.
For details, see \cite[Lemma 2.10]{deng-hao}.
\end{rem}

We need the following well-known result in this paper. 

\begin{lem}\label{p-lem2.7}
Let $E$ be a holomorphic vector bundle on a complex manifold $X$ with 
$\dim X=1$ and let $h$ be a smooth Hermitian metric on $E$ such that 
$\sqrt{-1} \Theta_h(E)$ is seminegative. 
Let $s$ be any holomorphic section of $E$ on $X$. 
Then $\log |s|^2_h$ is subharmonic. 
\end{lem}
We give a proof of Lemma \ref{p-lem2.7} for the sake of completeness although 
it is well known. 

\begin{proof}[Proof of Lemma \ref{p-lem2.7}]

Let $\{\bullet, \bullet\}$ denote the sesquilinear pairing
\[
C^\infty(X, \wedge^p T^\vee_X \otimes E) \times 
C^\infty(X, \wedge^q T^\vee_X \otimes E) \to 
C^\infty(X, \wedge^{p+q} T^\vee_X \otimes \mathbb{C})
\]
induced by the Hermitian metric $h$.

Let $\Omega$ be an open subset of $X$, 
and assume that $E|_\Omega$ is trivialized 
as $\Omega\times \mathbb C^r$ by a $C^\infty$ frame $\{e_\lambda\}$. 
Then for any sections
\[
u = \sum_\lambda u_\lambda \otimes e_\lambda,
\quad 
v = \sum_\mu v_\mu \otimes e_\mu,
\]
we have
\[
\{u, v\} = \sum_{\lambda, \mu} u_\lambda \wedge 
\overline v_\mu \cdot h(e_\lambda, e_\mu).
\]

Let $D = D' + \overline{\partial}$ denote the Chern 
connection associated with $(E, h)$. 
Outside the zero set of $s$, we have
\[
\begin{split}
\sqrt{-1} \partial \overline{\partial} \log |s|^2_h 
&= \sqrt{-1} \frac{\{D' s, D' s\}}{|s|^2_h} 
  - \sqrt{-1} \frac{\{D' s, s\} \wedge \{s, D' s\}}{|s|^4_h} 
  - \frac{\{\sqrt{-1} \Theta_h(E) s, s\}}{|s|^2_h} \\
&\geq - \frac{\{\sqrt{-1} \Theta_h(E) s, s\}}{|s|^2_h}
\end{split}
\]
by the Cauchy--Schwarz inequality. 

Since $\sqrt{-1} \Theta_h(E)$ is assumed 
to be seminegative, it follows that
\[
\sqrt{-1} \partial \overline{\partial} \log |s|^2_h \geq 0
\]
outside the zero set of $s$. That is, 
$\log |s|^2_h$ is subharmonic on $X \setminus \{s = 0\}$. 

Moreover, since $\log |s|^2_h$ is 
locally bounded from above, it extends to a subharmonic function on all of $X$ (see \cite[(3.3.25) Theorem]{noguchi-ochiai}). 

This completes the proof of Lemma \ref{p-lem2.7}.
\end{proof}

\section{Lemmas for functions on a punctured disk}\label{p-sec3} 

In this section, we present several elementary 
lemmas used in the proof of 
Proposition \ref{p-prop1.5} (see also Proposition \ref{p-prop4.1}).  
We begin with a result concerning the Lelong number.  
The following lemma is well 
known; for details, see, for example, 
\cite[2.B.~Lelong Numbers]{demailly1} and 
\cite[Chapter~III, (6.9) Example]{demailly}.

\begin{lem}[Lelong number]\label{p-lem3.1}
Let $u$ be a subharmonic function on $\Delta$. Then we have
\begin{equation}\label{p-eq3.1} 
\lim_{r \to 0} \int_{\Delta(0, r)} 
\frac{\sqrt{-1}}{\pi} \partial \overline{\partial} u 
= \liminf_{z \to 0} \frac{u(z)}{\log |z|}.
\end{equation}
We define
\[
\nu(u, 0) := \liminf_{z \to 0} \frac{u(z)}{\log |z|}
\]
and call it the {\em Lelong number} of $u$ at $0$. 
Note that the expression $\partial \overline{\partial} u$ is 
understood in the sense of currents.

Let $u_1$ and $u_2$ be subharmonic 
functions on $\Delta$. Then $u_1 + u_2$ is also subharmonic on $\Delta$.  
By \eqref{p-eq3.1}, we have the identity
\[
\nu(u_1 + u_2, 0) = \nu(u_1, 0) + \nu(u_2, 0).
\]
\end{lem}

We recall the following elementary lemma.

\begin{lem}[Harmonic functions on $\Delta^*$]\label{p-lem3.2}
Let $f$ be a harmonic function on $\Delta^*$. Then there 
exist a holomorphic function $g$ on $\Delta^*$ and a 
real constant $c \in \mathbb{R}$ such that
\[
f(z) = \re g(z) + c \log |z|.
\]
\end{lem}

We include a detailed proof of Lemma \ref{p-lem3.2} for completeness.

\begin{proof}[Proof of Lemma \ref{p-lem3.2}]
Consider the universal covering
\[
\pi \colon H := \{ w \in \mathbb{C} \mid \re w < 0 \} \to \Delta^*
\]
given by $\pi(w) = e^w$. Then $f \circ \pi$ 
is a harmonic function on the simply connected 
domain $H$, so there exists a holomorphic function $p(w)$ on $H$ such that
\[
\re p(w) = f \circ \pi(w).
\]

Define
\[
q(w) := p(w + 2\pi \sqrt{-1}) - p(w).
\]
Then $q(w)$ is holomorphic on $H$, and since $\pi(w + 2\pi \sqrt{-1}) = \pi(w)$, we have
\[
\re q(w) = \re p(w + 2\pi \sqrt{-1}) - \re p(w) = f \circ \pi(w+2\pi \sqrt{-1}) 
- f \circ \pi(w) = 0.
\]
Hence, $q(w)$ is a purely imaginary constant, i.e.,
\[
q(w) = 2\pi \sqrt{-1} c
\]
for some real constant $c \in \mathbb{R}$.

Set
\[
r(w) := p(w) - c w.
\]
Then $r(w)$ is holomorphic and satisfies
\begin{equation*}
\begin{split}
r(w + 2\pi \sqrt{-1})& = p(w + 2\pi \sqrt{-1}) 
- c(w + 2\pi \sqrt{-1}) \\&= p(w) + 2\pi \sqrt{-1} c - c w - 2\pi \sqrt{-1} c \\ &=r(w).
\end{split}
\end{equation*}
Thus, $r$ is $2\pi \sqrt{-1}$-periodic and descends to a holomorphic 
function $g(z)$ on $\Delta^*$ such that
\[
g \circ \pi(w) = r(w).
\]
Therefore,
\[
f(z) = \re p(w) = 
\re (r(w) + c w) = \re g(z) + c \log |z|,
\]
where we used that $w = \log z$ for $z \in \Delta^*$. This completes the proof.
\end{proof}

We next state another elementary lemma.

\begin{lem}\label{p-lem3.3}
Let $g$ be a holomorphic function on $\Delta^*$. Assume that
\[
\re g(z) \leq C (-\log |z|)
\]
holds on $\Delta^*$ for some constant 
$C > 0$. Then $g$ extends holomorphically to 
the origin; that is, the origin is a removable singularity of $g$.
\end{lem}

We also provide a proof of Lemma \ref{p-lem3.3} for the reader's convenience.

\begin{proof}[Proof of Lemma \ref{p-lem3.3}]
By the Casorati--Weierstrass theorem or Picard's big 
theorem, $g$ is meromorphic at $0$. So we may write 
\[
g(z) = \frac{p(z)}{z^m},
\]
where $p(z)$ is holomorphic on $\Delta$ with $p(0) \neq 0$ and $m$ is an integer.

Let $z = r e^{\sqrt{-1}\theta}$. Suppose, for contradiction, 
that $m > 0$. Then we can choose $\theta_0 \in [0, 2\pi)$ such that
\[
\frac{p(0)}{e^{\sqrt{-1} m \theta_0}} \in \mathbb{R}_{> 0}.
\]
Since $p$ is continuous 
and $p(0) \neq 0$, there exists $0 < r_0 \ll 1$ 
such that for all $0 < r < r_0$, the real part of
\[
\frac{p(r e^{\sqrt{-1}\theta_0})}{e^{\sqrt{-1} m \theta_0}}
\]
is greater than some constant $a > 0$. It follows that
\[
\re g(r e^{\sqrt{-1}\theta_0}) = \re 
\left( \frac{p(r e^{\sqrt{-1}\theta_0})}
{r^m e^{\sqrt{-1} m \theta_0}} \right) > \frac{a}{r^m}.
\]
But the assumption gives
\[
\re g(r e^{\sqrt{-1}\theta_0}) \leq C (-\log r).
\]
This is a contradiction for sufficiently small $r$, 
since $r^{-m}$ grows much faster than $-\log r$ as $r \to 0$. Hence, $m \leq 0$. 
This implies that $g$ is holomorphic at $0$. 
\end{proof}

\section{Prolongation of acceptable line bundles}\label{p-sec4}

In this section, we prove 
Proposition \ref{p-prop1.5}, along with Corollaries \ref{p-cor1.7} and 
\ref{p-cor1.8}.  
We recall that for any \( a \in \mathbb{R} \),
\[
\lceil a \rceil := \min \{ n \in \mathbb{Z} \mid n \geq a \} \quad \text{and} \quad 
\lfloor a \rfloor := \max \{ n \in \mathbb{Z} \mid n \leq a \}.
\]

\begin{prop}[Proposition \ref{p-prop1.5}]\label{p-prop4.1} 
Let $(L, h)$ be an acceptable line bundle 
on $\Delta^*$. 
Then ${}_\alpha L$ 
is a holomorphic line bundle 
on $\Delta$ for every $\alpha \in \mathbb R$. 
\end{prop}

A more precise description of ${}_\alpha L$ is given in Theorem \ref{p-thm4.4} below. 

\begin{proof}[Proof of Proposition \ref{p-prop4.1}]
We will see the behavior of the metric $h$ around 
the origin by taking a suitable trivialization of $L$ on $\Delta^*$ concretely. 
\begin{step}\label{p-step4.1-1}
We put 
\[
\omega_P:=\frac{\sqrt{-1}dz\wedge d\overline z}{|z|^2(-\log|z|^2)^2}
\] 
and 
\[
\chi (N):=-N\log \left(-\log |z|^2\right). 
\] 
We can check that 
\[ 
\sqrt{-1}\partial \overline \partial \chi (N)=N\omega_P. 
\] 
Since $(L, h)$ is an acceptable line bundle on $\Delta^*$, 
there exists $C>0$ such that 
\begin{equation}\label{p-eq4.1}
-C\cdot \omega_P\leq \sqrt{-1}\Theta_h(L)\leq C\cdot \omega_P 
\end{equation} 
holds on $\Delta^*$. 
We fix some positive number $N$ with $N\geq C$. 
We consider Hermitian metrics $he^{-\chi(N)}$ and 
$he^{-\chi(-N)}$ on $L$. 
Then we obtain 
\[ 
\begin{split}
\sqrt{-1}\Theta_{he^{-\chi(N)}}(L)&=
\sqrt{-1}\Theta_h(L)+\sqrt{-1}\partial \overline\partial \chi (N)
\\ &=\sqrt{-1}\Theta_h(L)+N\omega_P\geq 0 
\end{split} 
\] and 
\[ 
\begin{split}
\sqrt{-1}\Theta_{he^{-\chi(-N)}}(L)&=
\sqrt{-1}\Theta_h(L)+\sqrt{-1}\partial \overline\partial \chi (-N)
\\ &=\sqrt{-1}\Theta_h(L)-N\omega_P\leq 0  
\end{split} 
\] 
by \eqref{p-eq4.1}. 
\end{step}
\begin{step}\label{p-step4.1-2}
Since $L$ is a holomorphic line bundle on $\Delta^*$, 
we can trivialize $L$ on $\Delta^*$ (see, for example, \cite[30.3.~Theorem]{forster}).  
Hence, from now,  
we assume $L=\mathcal O_{\Delta^*}$. 
Then we can write 
\[
h=|\cdot|^2e^{-2\varphi}
\]
with some smooth function $\varphi$ on $\Delta^*$. 
We note that 
\[
\sqrt{-1}\Theta_h(L)=\sqrt{-1}\partial \overline \partial 2\varphi 
\] 
on $\Delta^*$. 
\end{step}
\begin{step}\label{p-step4.1-3}
Since 
\[
\int _{\Delta(0, r_0)}\omega_P<\infty
\] 
for every $0<r_0<1$, 
we can see $\omega_P$ a closed positive $(1, 1)$-current on $\Delta$. 
By \eqref{p-eq4.1}, $\sqrt{-1}\Theta_h(L)$ can be seen as a 
$(1, 1)$-current on $\Delta$. Since $\dim_{\mathbb C} \Delta=1$, 
$\sqrt{-1}\Theta_h(L)$ is obviously 
$d$-closed. 
Hence $\sqrt{-1}\Theta_h(L)$ defines a closed $(1, 1)$-current on $\Delta$. 
By Step \ref{p-step4.1-1}, 
\[ 
\sqrt{-1}\Theta_{he^{-\chi(N)}}(L)\quad \text{and}
\quad  
-\sqrt{-1}\Theta_{he^{-\chi(-N)}}(L)
\] are closed 
positive $(1, 1)$-current on $\Delta$. 
This is a very special case of the Skoda--El Mir extension theorem 
(see \cite[(1.18) Theorem]{demailly1} and \cite[Chapter III, \S 2.A]{demailly}). 
\end{step}
\begin{step}\label{p-step4.1-4}
Since we are working on $\Delta$, 
we can find subharmonic functions $\psi_1$ and $\psi_2$ 
on $\Delta$ such that 
\[
\sqrt{-1} \Theta _{he^{-\chi(N)}}(L)=\sqrt{-1} \partial 
\overline \partial 2\psi_1
\] 
and 
\[
-\sqrt{-1} \Theta _{he^{-\chi(-N)}}(L)=\sqrt{-1} \partial 
\overline \partial 2\psi_2. 
\] 
Since $2\varphi+\chi (N)-2\psi_1$ is harmonic 
on $\Delta^*$, by Lemma \ref{p-lem3.2}, 
we can write 
\begin{equation}\label{p-eq4.2}
2\varphi +\chi (N)=2\psi_1+c_1\log |z|^2+2\re g_1(z)
\end{equation}
for some holomorphic 
function $g_1$ on $\Delta^*$ and 
some $c_1\in \mathbb R$. 
Similarly, we can write 
\begin{equation}\label{p-eq4.3}
-2\varphi +\chi (N)=2\psi_2+c_2\log |z|^2+2\re g_2(z)
\end{equation}
for some holomorphic function $g_2$ on $\Delta^*$ and 
some $c_2\in \mathbb R$. 
For the details, see, for example, \cite[Chapter III, \S1.C]{demailly}. 
\end{step}
\begin{step}\label{p-step4.1-5}
By multiplying $e^{g_1(z)}$, 
we take a different trivialization of $L$. 
Then $h$ becomes $|\cdot |^2e^{-2\varphi+2\re g_1}$. 
Hence, by considering this new trivialization of $L$ on 
$\Delta^*$, that is, by replacing 
$-2\varphi+2\re g_1$ with $-2\varphi$, we may assume that 
\begin{equation}\label{p-eq4.4}
2\varphi +\chi(N)=2\psi_1+c_1\log |z|^2
\end{equation} 
holds. In this case, 
\[
2\chi (N)=2\psi_1+2\psi_2+(c_1+c_2)\log |z|^2+2\re g_2(z) 
\] 
holds by \eqref{p-eq4.3} and \eqref{p-eq4.4}. 
Note that $\chi(N)$, $\psi_1$, $\psi_2$, and $\log |z|^2$ are subharmonic 
functions on $\Delta$. 
We have 
\[
\frac{-\re g_2(z)}{\log |z|}=\frac{\psi_1(z)}{\log |z|}
+\frac{\psi_2(z)}{\log |z|}+c_1+c_2
+\frac{-\chi(N)}{\log|z|}. 
\] 
Therefore, we obtain 
\[
\begin{split}
\liminf_{z\to 0}\frac{-\re g_2(z)}{\log |z|}&\geq 
\liminf_{z\to 0}\frac{\psi_1(z)}{\log |z|}
+\liminf_{z\to 0}\frac{\psi_2(z)}{\log |z|}
+c_1+c_2
+\liminf_{z\to 0}\frac{-\chi(N)}{\log|z|}
\\ &=\nu(\psi_1, 0)+\nu(\psi_2, 0)+c_1+c_2. 
\end{split}
\] 
Thus there exists some $C>0$ such that 
\[
\frac{-\re g_2(z)}{\log |z|}\geq -C
\] 
holds over some open neighborhood of $0$. 
This implies that 
\[
\re\left(-g_2(z)\right)\leq C\left(-\log |z|\right)
\] 
holds around $0$. 
By Lemma \ref{p-lem3.3}, we see that 
$g_2$ is holomorphic on $\Delta$. 
Therefore, $\re g_2(z)$ is a harmonic function on $\Delta$. 
Hence, by replacing $\psi_2$ with $\psi_2-\re g_2(z)$, 
we may further assume that  
\begin{equation}\label{p-eq4.5}
-2\varphi +\chi (N)=2\psi_2+c_2\log |z|^2
\end{equation} 
holds. By \eqref{p-eq4.4} and \eqref{p-eq4.5}, 
we have 
\[
2\chi (N)=2\psi_1+2\psi_2+(c_1+c_2)\log |z|^2. 
\] 
Thus, we obtain 
\begin{equation}\label{p-eq4.6}
\nu_1+\nu _2+(c_1+c_2)=0 
\end{equation} 
by Lemma \ref{p-lem3.1}, 
where $\nu_1:=\nu (\psi_1, 0)$ and $\nu_2:=\nu(\psi_2, 0)$. 
We put $\gamma:=\nu_1+c_1$. 
Then $\nu_2+c_2 =-\gamma$ by \eqref{p-eq4.6}. 

By \eqref{p-eq4.4}, we have 
\[
\liminf_{z\to 0} \frac{\varphi(z)}{\log |z|} =\nu_1+c_1=\gamma. 
\]
By 
\eqref{p-eq4.5}, we have 
\[
\liminf_{z\to 0} \frac{-\varphi(z)}{\log |z|} =\nu_2+c_2=-\gamma. 
\]
Therefore, we obtain 
\[
\gamma =\liminf_{z\to 0} \frac{\varphi(z)}{\log |z|} 
\leq \limsup_{z\to 0} \frac{\varphi(z)}{\log |z|} =\gamma. 
\] 
Hence, we finally obtain 
\begin{equation}\label{p-eq4.7}
\gamma =\lim_{z\to 0} \frac{\varphi(z)}{\log |z|}. 
\end{equation} 
\end{step} 

\begin{step}\label{p-step4.1-6} 
In this final step, we will prove the following claim. 

\begin{claim}\label{p-claim} 
The following equality 
\[
{}_\alpha L=\mathcal O_\Delta\cdot z^{-\lfloor \alpha -\gamma\rfloor}
\] 
holds, that is, ${}_\alpha L$ is generated by 
$z^{-\lfloor \alpha -\gamma \rfloor}$. 
\end{claim}
We give a detailed proof of Claim for the sake of completeness. 
\begin{proof}[Proof of Claim] 
We put 
$m_\alpha :=\lfloor \alpha -\gamma\rfloor$. 
Then we have 
\[
m_\alpha \leq \alpha -\gamma <m_\alpha +1. 
\] 
Throughout this proof, we will freely shrink $\Delta$ around $0$. 

First, we will prove the inclusion $\mathcal O_\Delta\cdot z^{-m_\alpha} \subset 
{}_\alpha L$. 
Let $f$ be any local section of $\mathcal O_\Delta\cdot z^{-m_\alpha}$, 
By \eqref{p-eq4.7}, for any $\varepsilon >0$, 
we have 
\[
\gamma -\varepsilon \leq \frac{\varphi(z)}{\log |z|} =
\frac{-\varphi(z)}{-\log |z|} \leq \gamma +\varepsilon 
\] 
around $0$. 
Therefore, we obtain 
\[
(-\gamma +\varepsilon) \log |z| \leq -\varphi(z) 
\leq (-\gamma -\varepsilon) \log |z|
\] 
on some open neighborhood of $0$. 
Thus, we have 
\[
|z|^{-\gamma +\varepsilon} \leq e^{-\varphi(z)} \leq |z|^{-\gamma -\varepsilon}
\] 
around $0$. 
Hence we have 
\[
\begin{split}
|f|_h |z|^{\alpha +\varepsilon} &=|f| e^{-\varphi(z)} |z|^{\alpha +\varepsilon} 
\\ & \leq |f| |z|^{-\gamma -\varepsilon} |z| ^{\alpha +\varepsilon} \\ 
&= |f| |z|^{\alpha -\gamma} \\ 
&\leq |f| |z|^{m_\alpha} \\&\leq C
\end{split} 
\] 
around $0$ since $f$ is a local section of $\mathcal O_\Delta\cdot 
z^{-m_\alpha}$. Here we used 
$|z|^{m_\alpha} \geq |z|^{\alpha -\gamma}$ since $m_\alpha 
\leq \alpha -\gamma$ and $|z|<1$. 
Thus, we see that $f$ is in ${}_\alpha L$. This is what we wanted. 

From now, we will prove the opposite inclusion 
${}_\alpha L\subset \mathcal O_\Delta\cdot z^{-m_\alpha}$. 
Let $f$ be any local section of ${}_\alpha L$. 
Since $\alpha -\gamma <m_\alpha +1$, 
we have 
$m_\alpha +1-(\alpha -\gamma) >0$. We put 
\[
\varepsilon :=\frac{m_\alpha +1-(\alpha -\gamma)}{3}>0. 
\] 
Then $\alpha -\gamma +2\varepsilon =m_\alpha +1-\varepsilon$. 
By shrinking $\Delta$ suitably, 
there exists some constant $C>0$ such that 
\[
|f|_h |z|^{\alpha +\varepsilon} =|f|e^{-\varphi(z)} |z|^{\alpha +\varepsilon} 
\leq C 
\] 
holds for $z\in \Delta^*$. As before, 
we may assume that 
\[
|z|^{-\gamma +\varepsilon} \leq e^{-\varphi(z)} 
\leq |z|^{-\gamma -\varepsilon}
\] 
holds around $0$. 
Therefore, we obtain 
\[
\begin{split}
|f| |z|^{m_\alpha +1-\varepsilon} &= |f| |z|^{\alpha -\gamma +2\varepsilon} \\ 
&=|f| |z|^{-\gamma +\varepsilon} |z| ^{\alpha +\varepsilon} 
\\ 
&\leq |f| e^{-\varphi(z)} |z|^{\alpha +\varepsilon} \\ 
&\leq C 
\end{split}
\] 
around $0$. 
This means that $z^{m_\alpha}f$ is holomorphic at $0$, that is, 
$f$ is in $\mathcal O_\Delta \cdot z^{-m_\alpha}$. 
This is what we wanted. 

Hence we have ${}_\alpha L=\mathcal O_\Delta\cdot z^{-m_\alpha}$. 
We finish the proof of Claim. 
\end{proof}
\end{step}
In particular, ${}_\alpha L$ is a holomorphic line bundle 
on $\Delta$. 
This completes the proof of Proposition \ref{p-prop4.1}. 
\end{proof}

Although we do not use the following observation in this paper, 
we record it here for possible future use.

\begin{rem}\label{p-rem4.2}
By \eqref{p-eq4.1} and Lemma \ref{p-lem3.1}, 
we can easily verify that $\nu_1 = \nu_2 = 0$ in the 
proof of Proposition \ref{p-prop4.1}, since 
the Lelong number $\nu(\chi (N), 0)$ of $\chi (N)$ at $0$ is zero. 
Therefore, we have $\gamma = c_1 = -c_2$. 
Hence,
\[
2\varphi + \chi(N) = 2\psi_1 + \gamma \log |z|^2
\quad \text{and} \quad
-2\varphi + \chi(N) = 2\psi_2 - \gamma \log |z|^2.
\]
In particular, we obtain $\psi_1 + \psi_2 = \chi(N)$. 
Thus,
\[
h = |\cdot|^2 e^{-2\varphi}
= \frac{|\cdot|^2}{|z|^{2\gamma}} e^{-2\psi_1 + \chi(N)}
= \frac{|\cdot|^2}{|z|^{2\gamma}} e^{-\psi_1 + \psi_2},
\]
and
\[
h^\vee = |\cdot|^2 e^{2\varphi}
= |\cdot|^2 |z|^{2\gamma} e^{-2\psi_2 + \chi(N)}
= |\cdot|^2 |z|^{2\gamma} e^{-\psi_2 + \psi_1}.
\]
\end{rem}

Let us prove Corollary \ref{p-cor1.7}. 

\begin{cor}[Corollary \ref{p-cor1.7}]\label{p-cor4.3}
Let $(L, h)$ be an acceptable line bundle on $\Delta^*$. 
Then, for every $\alpha \in \mathbb R$, 
\[
\left({}_\alpha L\right)^\vee ={}_{-\alpha +1-\varepsilon}\left(L^\vee \right)
\] 
holds for all sufficiently small $\varepsilon>0$. 
\end{cor}

\begin{proof}[Proof of Corollary \ref{p-cor4.3}] 
In the proof of Proposition \ref{p-prop4.1}, 
the metric of $L^\vee$ is $|\cdot|^2e^{2\varphi}$. 
We replace $\varphi$ with $-\varphi$ and use the same 
argument as in the proof of Proposition \ref{p-prop4.1}. 
More precisely, for $L$, we used \eqref{p-eq4.5} in the proof of 
Proposition \ref{p-prop4.1}. For $L^\vee$, it is 
sufficient to use \eqref{p-eq4.4}. 
Then ${}_\beta\left(L^\vee\right)$ is generated by 
$z^{-\lfloor \beta+\gamma\rfloor}$ for every $\beta\in \mathbb R$. 
We put $\beta=-\alpha +1-\varepsilon$. 
If $0<\varepsilon \ll 1$, then 
\[
\begin{split}
-\lfloor \beta+\gamma\rfloor 
&=-\lfloor -\alpha+1-\varepsilon +\gamma\rfloor \\ 
& =\lceil \alpha -1+\varepsilon -\gamma\rceil \\ 
& =\lfloor \alpha-\gamma\rfloor. 
\end{split} 
\] 
Hence we obtain the desired equality ${}_{-\alpha +1-\varepsilon} \left(L^\vee\right) 
=\left({}_{\alpha} L\right)^{\vee}$ for $0<\varepsilon \ll 1$. 
\end{proof}

By the proof of Proposition \ref{p-prop4.1}, 
Corollary \ref{p-cor1.8} is almost obvious. 

\begin{proof}[Proof of Corollary \ref{p-cor1.8}]
In \eqref{p-eq4.4} in the proof of Proposition \ref{p-prop4.1}, 
we can make $N=0$, $\psi_1=0$, and $c_1=c$. 
Then $e^{-2\varphi}=\frac{1}{|z|^{2c}}$. 
This is what we wanted. 
\end{proof}

For the reader's convenience, we 
summarize Proposition \ref{p-prop4.1} 
along with its proof.
To the best of the authors' knowledge, 
the following explicit description appears to be new. 

\begin{thm}[Theorem \ref{p-thm1.6}]\label{p-thm4.4} 
Let $(L, h)$ be an acceptable line bundle on $\Delta^*$. 
By taking a suitable trivialization 
\[
(L, h)\simeq \left(\mathcal O_{\Delta^*}, |\cdot|^2e^{-2\varphi}\right), 
\] 
we have the following properties. 
\begin{itemize}
\item[(i)] The limit 
\[
\gamma :=\lim_{z\to 0} \frac{\varphi(z)}{\log |z|}\in \mathbb R
\] 
exists. 
\item[(ii)] Let $f$ be a holomorphic function on $\Delta(0, r)^*$ for some $0<r<1$, 
where $\Delta(0, r)^*:=\{z\in \mathbb C\mid 0<|z|<r\}$. 
Then $f\in ({}_\alpha L)_0$ holds for some $\alpha \in \mathbb R$ if and only if 
$f$ is meromorphic at $0$, where $({}_\alpha L)_0$ 
denotes the stalk of ${}_\alpha L$ at $0 \in \Delta$. 
\item[(iii)] Let $f$ be a meromorphic function on some open neighborhood 
of $0$ and let $\alpha$ be any real number. 
Then $f\in ({}_\alpha L)_0$ holds if and only if 
\begin{equation}\label{p-eq4.8}
\lim _{z\to 0} \frac{\log \left(|f|e^{-\varphi}\right)}{\log |z|} \geq -\alpha. 
\end{equation} 
Note that 
\[
\lim _{z\to 0} \frac{\log \left(|f|e^{-\varphi}\right)}{\log |z|}=\ord _0 f-\gamma 
\] holds. 
Therefore, \eqref{p-eq4.8} is equivalent to 
\[
\ord _0 f\geq -\lfloor \alpha -\gamma\rfloor. 
\]
\item[(iv)] Let $f$ be a meromorphic function on some open 
neighborhood of $0$ and let $\alpha$ be any real number. 
Then $f\not\in ({}_\alpha L)_0$ holds if and only if 
\begin{equation}\label{p-eq4.9}
\lim _{z\to 0} \frac{\log \left(|f|e^{-\varphi}\right)}{\log |z|}<-\alpha. 
\end{equation} 
Note that \eqref{p-eq4.9} implies that 
\[
|f|e^{-\varphi} >\frac{1}{|z|^\alpha}
\] 
holds on some small open neighborhood of $0$. 
\end{itemize}
\end{thm}

From the above description 
of ${}_\alpha L$, the following result is immediate.
We state it explicitly for later use.

\begin{cor}\label{p-cor4.5}
Let $(L, h)$ be an acceptable line bundle on $\Delta^*$.
If $f \in ({}_\alpha L)_0$, then there exists $\beta < \alpha$ such that
$f \notin ({}_\beta L)_0$.
\end{cor}

We close this section with an important remark. 

\begin{rem}\label{p-rem4.6}
We consider 
\[
(L, h)=\left(\mathcal O_{\Delta^*}, \frac{|\cdot |^2}{|z|^{2c}}\right), 
\] 
that is, $h=|\cdot |^2e^{-2\varphi}$ with 
$\varphi=c\log |z|$. 
In this case, we have 
\[
\sqrt{-1}\Theta_h(L)=\sqrt{-1}\partial 
\overline \partial (2\varphi)=0
\] 
on $\Delta^*$. Note that we can see $\log |z|$ as a 
subharmonic function on $\Delta$ and 
that $\sqrt{-1} \partial \overline \partial (2\varphi)$ is not zero 
as a current on $\Delta$. 
\end{rem}

\section{On growth estimates}\label{p-sec5}

In this section, we present the minimal 
analytic results needed in later sections, 
for the reader's convenience.  
We begin with a discussion of the $\overline{\partial}$-equation, 
from which we derive a growth estimate via $L^2$-methods.

\begin{say}[Setting]\label{p-say5.1}
Let $g$ be the K\"ahler metric on $\Delta^*$ defined 
by $\omega_P$.  
Note that $\Delta^*$ is a complete K\"ahler manifold, 
even though $g$ itself is not complete.  
Moreover, $\Delta^*$ is a Stein manifold.

Let $(E, h)$ be an acceptable vector bundle on $\Delta^*$.  
Then, by Lemma \ref{p-lem2.5}, there 
exists a positive real number 
$N_0$ such that for every $N \geq N_0$, 
$\sqrt{-1}\Theta_{he^{-\chi(N)}}(E)$ is Nakano semipositive, and 
$\sqrt{-1}\Theta_{he^{-\chi(-N)}}(E)$ is Griffiths seminegative. 

For simplicity, we also denote by $g$ 
the metric on $K_{\Delta^*}^{\otimes -1}$ 
induced by $\omega_P$, whenever no confusion arises.  
Note that the line bundle $(K_{\Delta^*}^{\otimes -1}, g)$ is acceptable.  
Therefore, by Lemma \ref{p-lem2.2}, 
the vector bundle $E \otimes K_{\Delta^*}^{\otimes -1}$ is also acceptable.

Hence, we can choose a 
sufficiently large positive integer $N \geq N_0$ such that
\[
\sqrt{-1} \Theta_{hge^{-\chi (N)}}(E \otimes K_{\Delta^*}^{\otimes -1}) 
- \omega_P \otimes \mathrm{Id}_E
\]
is Nakano semipositive, again by Lemma \ref{p-lem2.5}.
\end{say}

Lemma \ref{p-lem5.2} is a straightforward application of the $\overline{\partial}$-equation.

\begin{lem}\label{p-lem5.2}
Let $(E, h)$ be an acceptable vector bundle on $\Delta^*$.  
Let $e$ be any element of $E_{z_0}$ for some point $z_0 \in \Delta^*$.  
Assume that
\[
\sqrt{-1} \Theta_{h g e^{-\chi(N)}}(E \otimes K_{\Delta^*}^{\otimes -1})
- \omega_P \otimes \mathrm{Id}_E
\]
is Nakano semipositive.  
Then there exists a holomorphic section 
$v(z)$ of $E$ on $\Delta^*$ such that $v(z_0) = e$ and
\[
\|v\|^2_{h e^{-\chi(N)}} := \int_{\Delta^*} |v|^2_{h e^{-\chi(N)}} \, \omega_P < \infty.
\]
\end{lem}

\begin{proof}[Proof of Lemma \ref{p-lem5.2}]
Take a local holomorphic section $u(z)$ of $E$ defined near $z_0$ 
such that $u(z_0) = e$. 
More precisely, 
$u(z) \in \Gamma(U, E)$ for some open neighborhood $U$ of $z_0$.  
Choose a smooth function $\rho$ 
on $\Delta^*$ such that $\rho \geq 0$, 
$\supp \rho \Subset U$, and $\rho = 1$ on some open neighborhood of $z_0$.

Consider the smooth $E$-valued $(0,1)$-form with compact support:
\[
\frac{\overline{\partial}(\rho(z) u(z))}{z - z_0}.
\] 
It is clearly $\overline{\partial}$-closed and can be regarded as a smooth $(0,1)$-form on $\Delta^*$ with values in  
\[
K_{\Delta^*} \otimes \bigl(E \otimes K_{\Delta^*}^{\otimes -1}\bigr).
\] 

Since 
\[
\sqrt{-1} \Theta_{h g e^{-\chi(N)}}(E \otimes K_{\Delta^*}^{\otimes -1}) 
- \omega_P \otimes \mathrm{Id}_E
\]
is Nakano semipositive by assumption, 
the $\overline{\partial}$-equation can be solved in the $L^2$ sense.

Thus, we can find a measurable $E$-valued function $w(z)$ such that
\[
\int_{\Delta^*} |w(z)|^2_{h e^{-\chi(N)}}  \omega_P < \infty
\] 
and that 
\[
\overline{\partial} w(z) = \frac{\overline{\partial}(\rho(z) u(z))}{z - z_0}.
\]
For details, see for example \cite[Chapter VIII, (6.1) Theorem]{demailly}.

Define
\[
v(z) := \rho(z) u(z) - (z - z_0) w(z).
\]
Then $v(z)$ is holomorphic on $\Delta^*$, satisfies $v(z_0) = e$, and
\[
\|v\|^2_{h e^{-\chi(N)}} = \int_{\Delta^*} 
|v(z)|^2_{h e^{-\chi(N)}} \ \omega_P < \infty.
\]
This completes the proof of Lemma \ref{p-lem5.2}.
\end{proof}

Lemma \ref{p-lem5.3} is a straightforward consequence of the mean value inequality 
for subharmonic functions.

\begin{lem}\label{p-lem5.3}
Let $(E, h)$ be an acceptable vector bundle on $\Delta^*$.  
Let $N$ be a positive integer such that  
$\sqrt{-1} \Theta_{he^{-\chi(N)}}(E)$ is Nakano semipositive and  
$\sqrt{-1} \Theta_{he^{-\chi(-N)}}(E)$ is Griffiths seminegative.  

Suppose that a holomorphic section $v$ of $E$ satisfies  
\[
\|v\|^2_{he^{-\chi(N)}} := \int_{\Delta^*} |v|^2_{he^{-\chi(N)}} \, \omega_P < \infty.
\]  
Then, for every $\varepsilon > 0$, there exists a constant $C_\varepsilon > 0$ such that  
\[
|v(z)|_h \leq C_\varepsilon \cdot \frac{1}{|z|^\varepsilon}
\]  
holds on $\Delta(0, r)^* := \{ z \in \mathbb{C} 
\mid 0 < |z| < r \}$ for some sufficiently small $r > 0$.
\end{lem}

We include a proof of Lemma \ref{p-lem5.3} for completeness.

\begin{proof}[Proof of Lemma \ref{p-lem5.3}]
\setcounter{step}{0}
In this proof, each $C_i$ denotes a positive constant 
for every $i$. 

\begin{step}\label{p-step5.3-1}
By assumption, the bundle $(E, he^{-\chi(-N)})$ is Griffiths seminegative.  
Hence, by Lemma \ref{p-lem2.7}, the function  
\[
\log |s|_{he^{-\chi(-N)}}
\]  
is subharmonic for any holomorphic section $s$ of $E$ on $\Delta^*$.  
In particular, this applies to $v$, so we may use the mean value inequality for $\log |v|^2_{he^{-\chi(-N)}}$.
\end{step}

\begin{step}\label{p-step5.3-2}
Fix $z \in \Delta^*$ with $0 < |z| < \tfrac{1}{4}$. Then:
\begin{equation}\label{p-eq5.1}
\begin{split}
\log |v(z)|^2_{he^{-\chi(-N)}} 
&\leq \frac{4}{\pi |z|^2} \int_{|w - z| \leq \frac{|z|}{2}} 
\log |v(w)|^2_{he^{-\chi(-N)}} \, dw \wedge d\overline{w} \\
&\leq \log \left( \frac{4}{\pi |z|^2} \int_{|w - z| \leq \frac{|z|}{2}} 
|v(w)|^2_{he^{-\chi(-N)}} \, dw \wedge d\overline{w} \right) \\
&\leq \log \left( C_1 \int_{|w - z| \leq \frac{|z|}{2}} 
\frac{1}{|w|^2} |v(w)|^2_{he^{-\chi(-N)}} \, dw \wedge d\overline{w} \right) \\
&\leq \log C_2 + \log \int_{|w - z| \leq \frac{|z|}{2}} 
|v(w)|^2_{he^{-\chi(N)}} \, \omega_P \\
&\leq C_3 + \log \|v\|^2_{he^{-\chi(N)}}.
\end{split}
\end{equation}

Here, the first inequality is the 
mean value inequality for subharmonic functions,  
and the second follows from Jensen's inequality.  
\end{step}

\begin{step}\label{p-step5.3-3}
Using the estimate \eqref{p-eq5.1}, we obtain:
\[
\begin{split}
|v(z)|_h 
&= |v(z)|_{he^{-\chi(-N)}} \cdot (-\log |z|^2)^{N/2} \\
&\leq C_4 \|v\|_{he^{-\chi(N)}} \cdot \frac{1}{|z|^\varepsilon}
\end{split}
\]
for some constant $C_4 > 0$ and any 
given $\varepsilon > 0$. 
This completes the desired estimate.
\end{step}

The proof of Lemma \ref{p-lem5.3} is now complete.
\end{proof}

The following lemma is also a consequence of 
subharmonicity. We will repeatedly use it in subsequent sections. 

\begin{lem}[{cf.~\cite[Lemma 21.2.7]{mochizuki4}}]\label{p-lem5.4}
Let $(E, h)$ be an acceptable vector bundle on 
$\Delta^*$. 
Let $f$ be a holomorphic section of $E$ on $\Delta^*$ such that 
\[
|f|_h=O\left( \frac{1}{|z|^{a+\varepsilon}}\right)
\] 
for any $\varepsilon >0$. 
We assume that $(E, he^{-\chi(-N)})$ is Griffiths seminegative. We put 
\[
H(z):=
|f|^2_{he^{-\chi(-N)}}|z|^{2a}=
|f|^2_h|z|^{2a}\left(-\log |z|^2\right)^{-N}. 
\] 
Then $H(z)$ is bounded on $\Delta(0, r_0)$ for any $0<r_0<1$. 
\end{lem}
\begin{proof}[Proof of Lemma \ref{p-lem5.4}]
We put 
$H_{\varepsilon} (z):=H(z)|z|^{2\varepsilon}$ for any $\varepsilon >0$. 
Note that $\log H_{\varepsilon}(z)$ is subharmonic on $\Delta^*$ 
by Lemma \ref{p-lem2.7}. 
By assumption, we have 
$\lim_{z\to 0} \log H_{\varepsilon}(z)=-\infty$. 
Hence $\log H_{\varepsilon}(z)$ is subharmonic on $\Delta$ 
(see \cite[(3.3.25) Theorem]{noguchi-ochiai}).
Therefore, we have 
\begin{equation}\label{p-eq5.2}
\max_{|z|\leq r_0} |H_{\varepsilon}(z)|
=\max_{|z|=r_0} H_{\varepsilon}(z). 
\end{equation} 
Note that $H(z)$ is a continuous function on $|z|=r_0$ and 
that $H_{\varepsilon _1}(z)\leq H_{\varepsilon_2}(z)$ holds on $|z|=r_0$ for 
$0\leq \varepsilon _2\leq \varepsilon _1\leq 1$. 
By taking the limit for $\varepsilon \to 0$, 
we obtain that $H(z)$ is bounded on $\Delta(0, r_0)$ by 
\eqref{p-eq5.2}. 
We finish the proof of Lemma \ref{p-lem5.4}.  
\end{proof}

\section{Prolongation of acceptable vector bundles}\label{p-sec6}

In this section, we establish Theorem \ref{p-thm1.3}. 

\begin{proof}[Proof of Theorem \ref{p-thm1.3}]

In the following proof, we will use Proposition \ref{p-prop1.5} 
(see Proposition \ref{p-prop4.1}), 
which is a special case of Theorem \ref{p-thm1.3}. 

\setcounter{step}{0}
\begin{step}\label{p-step1.3-1}
Let $(E, h)$ be the given acceptable vector bundle on $\Delta^*$ and 
let $\alpha$ be any real number. 
We put $E':=E$ and $h':=h\cdot |z|^{2\alpha}$ and 
consider $(E', h')$. Then we have 
\[
\sqrt{-1}\Theta_h(E)=\sqrt{-1}\Theta_{h'}(E')
\] 
on $\Delta^*$. Hence $(E', h')$ is also an acceptable 
vector bundle on $\Delta^*$. 
By the definition of $(E', h')$, 
${}_\alpha E={}_0 E'$ obviously holds. Therefore, it is sufficient to 
prove that ${}^\diamond\! E={}_0 E$ is a holomorphic 
vector bundle on $\Delta$. By definition, 
${}^\diamond\!E$ is a torsion-free sheaf on $\Delta$. 
Thus, it is sufficient to prove that ${}^\diamond\! E$ is coherent 
since ${}^\diamond\! E$ 
is a sheaf on $\Delta$. 
\end{step}

\begin{step}\label{p-step1.3-2}
Let $z_0 \in \Delta^*$ be any point. 
Let $\{e_1, \ldots, e_k\}$ be a basis of 
the fiber $E_{z_0}$, where $k = \dim_{\mathbb{C}} E_{z_0} = \rank E$. 

From now on, we allow ourselves to shrink 
the unit disk $\Delta$ and replace it with a smaller disk of the form
\[
\Delta(0, r) := \{ z \in \mathbb{C} \mid |z| < r \}
\]
for some $0 < r < 1$, without explicitly mentioning it.

By Lemmas \ref{p-lem5.2} and \ref{p-lem5.3}, 
for each $i$, we can find a holomorphic section 
$v_i(z)$ of $E$ on $\Delta^*$ such 
that $v_i(z_0) = e_i$ and
\[
|v_i(z)|_h = O\left(\frac{1}{|z|^\varepsilon}\right)
\]
for every $\varepsilon > 0$. 
In other words, $v_i \in \Gamma(\Delta, {}^\diamond\! E)$ for all $i$.
\end{step}

\begin{step}\label{p-step1.3-3}
We put $L:=\det (E)$. 
Then $L$ is an acceptable line bundle on $\Delta^*$ by Lemma \ref{p-lem2.2}. 
Since 
\[
\left(v_1\wedge \cdots \wedge v_k\right)(z_0)\ne 0, 
\] 
$v_1\wedge \cdots \wedge v_k$ is a nontrivial holomorphic 
section of ${}^\diamond\! L$. We fix a trivialization 
\[
{}^\diamond\! L=\mathcal O_{\Delta}\cdot \mathbf e. 
\] 
Then we can write 
\[
v_1\wedge \cdots \wedge v_k =a(z)\mathbf e
\] 
for some holomorphic function $a(z)$ on $\Delta$. 
We put $l:=\ord _0 a(z)\geq 0$. 
\end{step}
\begin{step}\label{p-step1.3-4} 
Since $a(z)$ is a holomorphic function on $\Delta$,
we may assume that $a(z)\not=0$ for all $z \in \Delta^*$
by shrinking $\Delta$ around $0$.
Then
$\left(v_1\wedge \cdots \wedge v_k\right)(z)\ne 0$
for all $z \in \Delta^*$.
Therefore the morphism
\begin{equation}
\label{p-eq6.1}
\mathcal{O}_{\Delta}^{\oplus k} \rightarrow {}^\diamond\! E
\end{equation}
defined by $v_1, \dots, v_k$ is isomorphic over $\Delta^*$. 
\end{step}
\begin{step}\label{p-step1.3-5}
Let $s$ be any local section of ${}^\diamond\!E$ around $0$. 
Since the morphism \eqref{p-eq6.1} is isomorphic over $\Delta^*$, 
we can write 
\[
s(z)=\sum _{i=1}^k s_i(z) v_i(z)
\] 
such that $s_i(z)$ is a holomorphic function on $\Delta^*$ for 
every $i$. 
Since 
\[
\begin{split}
s\wedge v_2\wedge \cdots \wedge v_k&= s_1 (z) v_1\wedge 
\cdots \wedge v_k \\ 
&=s_1(z) a(z) \mathbf e
\end{split} 
\] 
is a holomorphic section of ${}^\diamond\!L$, 
we have $\ord _0s_1(z)\geq -l$. 
Similarly, we obtain $\ord _0s_i(z)\geq -l$ for 
every $i$. 
This implies that 
\begin{equation}\label{p-eq6.2} 
{}^\diamond\!E\subset \bigoplus _{i=1}^k \mathcal O_{\Delta} 
\cdot \frac{v_i}{z^l_i}\subset j_*E, 
\end{equation} 
where $j\colon \Delta^*\hookrightarrow \Delta$. 
By definition, $\left({}^\diamond\!E\right)|_{\Delta^*}=E$ holds. 
Since we have 
\[
{}^\diamond\!E\subset \bigoplus ^k \mathcal O_\Delta (l [0])
\] 
by \eqref{p-eq6.2}, 
the stalk $\left({}^\diamond\! E\right)_0$ is a finitely 
generated $\mathcal O_{\Delta, 0}$-module. 
Then, by shrinking $\Delta$ around $0$ if necessary,  
we obtain a morphism 
\[
\mathcal{O}_{\Delta}^{\oplus n} \rightarrow {}^\diamond\!E
\]
for some positive integer $n$,  
which induces a surjection on the stalk at $0$. 
The direct sum of this morphism with the morphism~\eqref{p-eq6.1} 
is surjective over the entire disk $\Delta$. 
Hence, ${}^\diamond\!E$ is locally finitely generated over $\Delta$.  
This implies that ${}^\diamond\!E$ is a coherent $\mathcal{O}_\Delta$-module. 
\end{step}
We finish the proof of Theorem \ref{p-thm1.3}. 
\end{proof}

\section{Basic properties}\label{p-sec7}

In this section, we introduce a new invariant $\gamma ({}_\alpha E)$ and 
discuss basic properties of ${}_\alpha E$ and $\gamma ({}_\alpha E)$. 

\begin{say}[Setting]\label{p-say7.1}

Let $(E, h)$ be an acceptable vector bundle over $\Delta^*$ with 
$\rank E = r$. 
Let $\bm{v} := \{v_1, \ldots, v_r\}$ be a local 
frame of ${}_\alpha E$ defined on some open 
neighborhood of $0$. We consider the $r \times r$ matrix 
\[
H(h, \bm{v}) := \left(h(v_i, v_j)\right)_{i,j}.
\] 
More precisely, $H(h, \bm{v})$ is 
an $r \times r$ Hermitian matrix-valued 
function on $\Delta^*$. Hence, we sometimes 
write $H(h, \bm{v})(z)$ to denote the value of $H(h, \bm{v})$ at $z \in \Delta^*$. 
If there is no risk of confusion, we may simplify the 
notation by writing $H(\bm{v})$ and $H(\bm{v})(z)$ 
in place of $H(h, \bm{v})$ and $H(h, \bm{v})(z)$, respectively.

We have 
\[
|h(v_i, v_j)|\leq |v_i|_h|v_j|_h=O\left(\frac{1}{|z|^{2\alpha +\varepsilon}}\right)
\] 
for any $\varepsilon >0$. 
This means that for any $\varepsilon >0$ there exists some $C_\varepsilon>0$ 
such that 
\[
\det H(\bm v)=|\det H(\bm v)|\leq C_{\varepsilon} 
{|z|{^{-2\alpha r-\varepsilon}}}. 
\]
Thus we obtain 
\[
\log \det H(\bm v)\leq \log C_{\varepsilon} 
-(2\alpha r+\varepsilon)\log |z|. 
\] 
Hence we have 
\[
\frac{\log \det H(\bm v)}{\log |z|} \geq 
\frac{\log C_\varepsilon}{\log |z|} -
(2\alpha r+\varepsilon). 
\]
Therefore, 
\[ 
\liminf_{z\to 0} \frac{\log\det H(\bm v)}{\log |z|}
\] 
satisfies  
\[
\liminf_{z\to 0} \frac{\log\det H(\bm v)}{\log |z|}\geq -2\alpha r-\varepsilon. 
\] 
Since $\varepsilon >0$ is arbitrary, we obtain 
\begin{equation}\label{p-eq7.1}
\liminf_{z\to 0} \frac{\log\det H(\bm v)}{\log |z|}\geq -2\alpha r>-\infty. 
\end{equation}
\end{say}

\begin{lem}\label{p-lem7.2}
In the above definition, 
\[
\liminf_{z\to 0} \frac{\log\det H(\bm v)}{\log |z|}
\] 
is independent of the choice of the frame $\bm v=\{v_1, \ldots, v_r\}$ 
of ${}_\alpha E$. 
\end{lem} 
\begin{proof}[Proof of Lemma \ref{p-lem7.2}]
Let $\bm w:=\{w_1, \ldots, w_r\}$ be another frame of ${}_\alpha E$ around 
$0$. Then we can write 
\[
(w_1, \ldots, w_r)=(v_1, \ldots, v_r)A(z)
\] 
around $0$, where $A(z)$ is an invertible $r\times r$ matrix. 
Thus we have 
\[
\det H(\bm w)=|\det A(z)|^2 \det H(\bm v). 
\] 
Hence 
\[
\frac{\log \det H(\bm w)}{\log |z|} 
=\frac{2|\det A(z)|}{\log |z|} +
\frac{\log \det H(\bm v)}{\log |z|}. 
\] 
Since $\det A(0)\ne 0$, we obtain 
\[
\liminf_{z\to 0} \frac{\log\det H(\bm w)}{\log |z|}=
\liminf_{z\to 0} \frac{\log\det H(\bm v)}{\log |z|}. 
\] 
This is what we wanted. We finish the proof of Lemma \ref{p-lem7.2}. 
\end{proof}

We can prove the following lemma. 

\begin{lem}\label{p-lem7.3}
\[
-\infty<\liminf _{z\to 0} \frac{\log \det H(\bm v)}{\log |z|}<\infty. 
\]
\end{lem}
\begin{proof}[Proof of Lemma \ref{p-lem7.3}]
By \eqref{p-eq7.1}, we have already checked the left inequality. 
Hence it is sufficient to 
prove the right inequality. 
Since $(\det E, \det h)$ is an acceptable line bundle on $\Delta^*$ by 
Lemma \ref{p-lem2.2}, 
we can freely use Theorem \ref{p-thm4.4}. 

By the above observation, we have 
\[
v_1\wedge \cdots \wedge v_r\in {}_{\alpha r} \det E.  
\] 
By Corollary \ref{p-cor4.5}, 
there exists some real number $\beta<\alpha r$ 
such that 
$v_1\wedge \cdots \wedge v_r\not\in {}_\beta \det E$. 
Hence, by Theorem \ref{p-thm4.4} (iv), we can take $d$ and $C>0$ such that 
\[
|\det H(\bm v)|^{1/2} =|v_1\wedge \cdots \wedge v_r| \geq 
\frac{C}{|z|^d} 
\] 
holds around $0$. 
This implies 
\[
\frac{1}{2} \log \det H(\bm v) \geq  -d\cdot  \log |z| +\log C
\] 
Hence 
\[
\liminf _{z\to 0} \frac{\log \det H(\bm v)}{\log |z|} \leq -2 d<\infty. 
\] 
This is what we wanted. We finish the proof of Lemma \ref{p-lem7.3}. 
\end{proof}

Having completed the necessary preparations, we now define $\gamma({}_\alpha E)$.

\begin{defn}\label{p-def7.4}
We put 
\[
\gamma ({}_\alpha E):=-\frac{1}{2} \liminf_{z\to 0} 
\frac{\log\det H(h, \bm v)}{\log |z|} 
\in \mathbb R. 
\]
\end{defn}

The following theorem is the main theorem of this section. 

\begin{thm}[Determinant bundles, see Theorem \ref{p-thm1.10}]\label{p-thm7.5}
Let $(E, h)$ be an acceptable vector bundle on $\Delta^*$. 
Let $\alpha$ be any real number. 
Then 
\[
\det ({}_\alpha E)={}_{\gamma ({}_\alpha E)} \det E
\] 
holds. 
\end{thm}

We give a detailed proof of Theorem \ref{p-thm7.5}, 
which is essentially the same as the proof of Proposition 
\ref{p-prop4.1}. 

\begin{proof}[Proof of Theorem \ref{p-thm7.5}]
Let $\bm v:=\{v_1, \ldots, v_r\}$ be a frame of ${}_\alpha E$ on $\Delta$. 
We put 
\[
H(h, \bm v)(z):=\left(h(v_i, v_j)(z)\right)_{i, j}. 
\] 
We note that 
\[
\det ({}_\alpha E)=\mathcal O_\Delta v_1\wedge \cdots \wedge v_r
\] 
and 
\begin{equation}\label{p-eq7.2}
\det E=\mathcal O_{\Delta^*}v_1\wedge \cdots \wedge v_r\simeq 
\mathcal O_{\Delta^*}. 
\end{equation} 
We consider $L:=\det E$. Let $h_L$ be the induced metric on $L$. 
Note that $(L, h_L)$ is an acceptable line bundle on $\Delta^*$ 
by Lemma \ref{p-lem2.2} since $(E, h)$ is acceptable. 
By using the trivialization \eqref{p-eq7.2}, 
we argue as in the proof of Proposition \ref{p-prop4.1}. 
In this setting, 
\[
h_L=|\cdot |^2 e^{-2\varphi_\alpha}
\] 
with 
\begin{equation}\label{p-eq7.3}
\begin{split}
e^{-2\varphi_\alpha} &=h_L(v_1\wedge \cdots \wedge v_r, v_1\wedge \cdots \wedge v_r) 
\\&=\det H(\bm v). 
\end{split}
\end{equation} 
Thus 
\[
\varphi_\alpha=-\frac{1}{2} \log \det H(\bm v). 
\] 
Therefore, 
\begin{equation}\label{p-eq7.4}
\liminf_{z\to 0} \frac{-\varphi_\alpha(z)}{\log |z|} 
=\frac{1}{2} \liminf_{z\to 0} \frac{\log \det H(\bm v)}{\log |z|} 
=-\gamma ({}_\alpha E). 
\end{equation} 
Note that $(L, h_L)$ is an acceptable line bundle on $\Delta^*$. 
As in the proof of Proposition \ref{p-prop4.1}, we can write 
\begin{equation}\label{p-eq7.5}
2\varphi_\alpha +\chi (N)=2\psi_1+c_1\log |z|^2+2\re g_1(z)
\end{equation}
and 
\begin{equation}\label{p-eq7.6}
-2\varphi_\alpha +\chi (N)=2\psi_2+c_2\log |z|^2+2\re g_2(z), 
\end{equation} 
where $\psi_1$ and $\psi_2$ are subharmonic functions on $\Delta$ and 
$g_1$ and $g_2$ are holomorphic functions on $\Delta^*$. 
By \eqref{p-eq7.5}, we have 
\[
-\re g_1(z)=\psi_1+c_1\log |z|-\varphi_\alpha -\chi (N). 
\]
Hence we have 
\[
\frac{-\re g_1(z)}{\log |z|}=\frac{\psi_1(z)}{\log |z|}
+c_1+\frac{-\varphi_\alpha(z)}{\log |z|}
+\frac{-\chi(N)}{\log|z|}. 
\] 
Therefore, we obtain 
\[
\begin{split}
\liminf_{z\to 0}\frac{-\re g_1(z)}{\log |z|}&\geq 
\liminf_{z\to 0}\frac{\psi_1(z)}{\log |z|}
+c_1+\liminf_{z\to 0}\frac{-\varphi_\alpha(z)}{\log|z|}
+\liminf_{z\to 0}\frac{-\chi(N)}{\log|z|}
\\ &=\nu(\psi_1, 0)+c_1-\gamma({}_\alpha E) \\ 
&>-\infty. 
\end{split}
\] 
Here we used Lemma \ref{p-lem7.3}. 
Thus there exists some $C>0$ such that 
\[
\frac{-\re g_1(z)}{\log |z|}\geq -C
\] 
holds over some open neighborhood of $0$. 
This implies that 
\[
\re\left(-g_1(z)\right)\leq C\left(-\log |z|\right)
\] 
holds around $0$. 
By Lemma \ref{p-lem3.3}, we see that 
$g_1$ is holomorphic on $\Delta$. 
Hence we have 
\[\nu_1+c_1-\gamma ({}_\alpha E)\leq 0=\liminf_{z\to 0} \frac{-\re g_1(z)}{\log |z|}, 
\] 
where $\nu_1:=\nu(\psi_1, 0)$. 
Note that $e^{g_1(z)}$ is a nowhere vanishing holomorphic 
function on $\Delta$. 
By replacing $v_1$ with $e^{g_1(z)}v_1$,  
we may assume that
\begin{equation}\label{p-eq7.7}
2\varphi_\alpha + \chi(N) = 2\psi_1 + c_1 \log |z|^2
\end{equation}
and
\begin{equation}\label{p-eq7.8}
-2\varphi_\alpha + \chi(N) = 2\psi_2 + c_2 \log |z|^2 + 2\re g_2(z)
\end{equation}
hold, after replacing $g_2$ accordingly. 
By \eqref{p-eq7.7} and \eqref{p-eq7.8}, we have 
\[
2\chi (N)=2\psi_1+2\psi_2+(c_1+c_2)\log |z|^2+2\re g_2(z) 
\] 
holds. 
Note that $\chi(N)$, $\psi_1$, $\psi_2$, and $\log |z|^2$ are subharmonic 
functions on $\Delta$. 
We have 
\[
\frac{-\re g_2(z)}{\log |z|}=\frac{\psi_1(z)}{\log |z|}
+\frac{\psi_2(z)}{\log |z|}+c_1+c_2
+\frac{-\chi(N)}{\log|z|}. 
\] 
Therefore, we obtain 
\[
\begin{split}
\liminf_{z\to 0}\frac{-\re g_2(z)}{\log |z|}&\geq 
\liminf_{z\to 0}\frac{\psi_1(z)}{\log |z|}
+\liminf_{z\to 0}\frac{\psi_2(z)}{\log |z|}
+c_1+c_2
+\liminf_{z\to 0}\frac{-\chi(N)}{\log|z|}
\\ &=\nu(\psi_1, 0)+\nu(\psi_2, 0)+c_1+c_2. 
\end{split}
\] 
Thus there exists some $C>0$ such that 
\[
\frac{-\re g_2(z)}{\log |z|}\geq -C
\] 
holds over some open neighborhood of $0$. 
This implies that 
\[
\re\left(-g_2(z)\right)\leq C\left(-\log |z|\right)
\] 
holds around $0$. 
By Lemma \ref{p-lem3.3}, we see that 
$g_2$ is holomorphic on $\Delta$. 
In particular, $\re g_2(z)$ is harmonic on $\Delta$. 
By replacing $\psi_2$ with $\psi_2 +\re g_2(z)$, we can finally assume that 
\begin{equation}\label{p-eq7.9}
2\varphi_\alpha +\chi (N)=2\psi_1+c_1\log |z|^2
\end{equation}
and 
\begin{equation}\label{p-eq7.10}
-2\varphi_\alpha +\chi (N)=2\psi_2+c_2\log |z|^2
\end{equation} 
hold. 
By \eqref{p-eq7.10}, we obtain 
\[
-\gamma ({}_\alpha E)=\nu_2+c_2,  
\] 
where $\nu_2:=\nu(\psi_2, 0)$. 
By \eqref{p-eq7.9} and \eqref{p-eq7.10}, we have 
\[
\chi (N)=\psi_1+\psi_2+(c_1+c_2)\log |z|. 
\] 
Therefore, by Lemma \ref{p-lem3.1}, we obtain 
\[
0=\nu_1+\nu_2+c_1+c_2. 
\] 
This means that 
\[\nu _1+c_1=-(\nu_2+c_2)=\gamma ({}_\alpha E). 
\]
By \eqref{p-eq7.9}, 
we have 
\begin{equation}\label{p-eq7.11}
\liminf_{z\to 0} \frac{\varphi_\alpha(z)}{\log |z|}=\nu_1+c_1 =
\gamma ({}_\alpha E).   
\end{equation} 
Thus, by \eqref{p-eq7.4} and \eqref{p-eq7.11}, we get 
\[
\lim_{z\to 0} \frac{\varphi_\alpha(z)}{\log |z|}=
\gamma ({}_\alpha E), \quad \lim _{z\to 0} \frac{\psi_1(z)}{\log |z|}=\nu_1, 
\quad \text{and}\quad 
\lim _{z\to 0} \frac{\psi_2(z)}{\log |z|}=\nu_2.  
\] 
As in the proof of Proposition \ref{p-prop4.1}, we obtain 
\begin{equation}\label{p-eq7.12}
{}_\beta L={}_\beta (\det E)=\mathcal O_{\Delta} \cdot 
z^{-\lfloor \beta -\gamma ({}_\alpha E)\rfloor}v_1\wedge \cdots 
\wedge v_r. 
\end{equation} 
In particular, 
\[
\det ({}_\alpha E)={}_{\gamma ({}_\alpha E)}\det E. 
\] 
We finish the proof of Theorem \ref{p-thm7.5}. 
\end{proof}

As a byproduct of the proof of Theorem \ref{p-thm7.5}, 
we have the following useful result, that is, 
we can replace $\liminf$ with $\lim$ in the definition of 
$\gamma ({}_\alpha E)$. 

\begin{cor}\label{p-cor7.6}
In the same setting as in \ref{p-say7.1}, we have 
\[
\gamma({}_\alpha E)=-\frac{1}{2}\lim _{z\to 0} \frac{\log \det H(\bm v)}{\log {|z|}}. 
\]
\end{cor}
\begin{proof}[Proof of Corollary \ref{p-cor7.6}]
By the proof of Lemma \ref{p-lem7.2}, 
we may assume that $\bm v$ is a frame of 
${}_\alpha E$ on $\Delta$. 
In the proof of Theorem \ref{p-thm7.5}, 
we have 
\[
\gamma ({}_\alpha E)=\lim_{z\to 0} \frac{\varphi_\alpha(z)}{\log |z|}=-\frac{1}{2}
\lim _{z\to 0} \frac{\log \det H(\bm v)}{\log |z|}. 
\] 
This is what we wanted. 
\end{proof}

For later use, we explicitly state the following result,  
which is an immediate consequence of Theorem~\ref{p-thm7.5} and its proof.

\begin{cor}\label{p-cor7.7}
Let $(E, h)$ be an acceptable vector bundle over 
$\Delta^*$ with $\rank E = r$.  
Let $\bm v:=\{v_1, \ldots, v_r\}$ be a local frame 
of ${}_\alpha E$, defined over an open neighborhood of $0$, such that  
$v_i \in {}_{b_i} E$ for each $i$.  
Then we have 
\begin{equation}\label{p-eq7.13}
\gamma({}_\alpha E) \leq \sum_{i=1}^r b_i,
\end{equation} 
and 
\begin{equation}\label{p-eq7.14}
\gamma\left({}_{\gamma ({}_\alpha E)} \det E\right) = \gamma ({}_\alpha E).
\end{equation}
\end{cor}

\begin{proof}[Proof of Corollary~\ref{p-cor7.7}]
Since $v_i \in {}_{b_i} E$ for each $i$,  
we have 
\[
v_1 \wedge \cdots \wedge v_r \in {}_{\sum_{i=1}^r b_i} \det E.
\]  
On the other hand, from \eqref{p-eq7.12} in the proof of Theorem~\ref{p-thm7.5}, we obtain
\[
\begin{split}
{}_{\gamma ({}_\alpha E)} \det E &= \mathcal O_\Delta \cdot 
v_1 \wedge \cdots \wedge v_r, \quad \text{and} \\
{}_{\sum_{i=1}^r b_i} \det E &= \mathcal O_\Delta 
\cdot z^{-\lfloor \sum_{i=1}^r b_i - \gamma({}_\alpha E) \rfloor} \, v_1 \wedge \cdots \wedge v_r.
\end{split}
\] 
Therefore, it follows that
\[
\sum_{i=1}^r b_i \geq \gamma ({}_\alpha E)
\quad \text{and} \quad
{}_{\gamma({}_\alpha E)} \det E \subset {}_{\sum_{i=1}^r b_i} \det E.
\]
This yields \eqref{p-eq7.13}.

We note that 
${}_{\gamma ({}_\alpha E)} \det E=\det ({}_\alpha E)$ by Theorem 
\ref{p-thm7.5}. 
From \eqref{p-eq7.3} in the proof of Theorem~\ref{p-thm7.5}, we obtain
\[
\gamma \left({}_{\gamma ({}_\alpha E)} \det E\right) 
= \liminf_{z \to 0} \frac{\varphi_\alpha (z)}{\log |z|}
= -\frac{1}{2} \liminf_{z \to 0} \frac{\log \det H(\bm v)}{\log |z|}
= \gamma ({}_\alpha E).
\]
This proves \eqref{p-eq7.14}.

This completes the proof of Corollary~\ref{p-cor7.7}.
\end{proof}

From now, we discuss some 
basic properties of $\gamma ({}_\alpha E)$ and 
${}_\alpha E$. We note that ${}_\alpha E=E$ holds on $\Delta^*$ by definition. 

\begin{lem}\label{p-lem7.8}
For $\alpha \leq \beta$, we have the following properties. 
\begin{itemize}
\item[(i)] ${}_\alpha E\subset {}_\beta E$ holds. 
\item[(ii)] $\gamma ({}_\alpha E)\leq \gamma 
({}_\beta E)$, and $\gamma ({}_\beta E)-\gamma ({}_\alpha E)\in 
\mathbb Z_{\geq 0}$. 
\item[(iii)] ${}_\alpha E={}_\beta E$ if and only if 
$\gamma({}_\alpha E)=\gamma ({}_\beta E)$. 
\item[(iv)] ${}_{\alpha +1}E={}_\alpha E\otimes \mathcal O_{\Delta}([0])$. 
\item[(v)] $\gamma ({}_{\alpha +1}E)-\gamma ({}_\alpha E)=\rank E$. 
\end{itemize}
\end{lem}

\begin{proof}[Proof of Lemma \ref{p-lem7.8}]
It is obvious that (i) holds by definition. 
Let $\bm w:=\{w_1, \ldots, w_r\}$ be a frame of ${}_\alpha E$ on 
$\Delta$ and let $\bm v:=\{v_1, \ldots, v_r\}$ be a frame of ${}_\beta E$ on $\Delta$. 
Thus we can write 
\[
(w_1, \ldots, w_r)=(v_1, \ldots, v_r)A(z)
\] 
where $A(z)$ is an $r\times r$ matrix. 
By definition, $A(z)$ is invertible on $\Delta^*$. 
Hence $\det A(z)\ne 0$ for every $z\in \Delta^*$. 
Then we can write 
\[
\det A(z)=z^mf(z)
\] 
for some $m\in \mathbb Z_{\geq 0}$ such that $f(0)\ne 0$. 
In this setting, 
we obtain 
\[
\begin{split}
\det H(\bm w)(z) &=\det H(\bm v)(z)|\det A(z)|^2\\ 
&= \det H(\bm v)(z) |z|^{2m}|f(z)|^2.  
\end{split} 
\] 
Therefore, we have 
\[
\frac{\log \det H(\bm w)(z)}{\log |z|} 
=\frac{\log \det H(\bm v)(z)}{\log |z|}+2m 
+2\frac{\log|f(z)|}{\log |z|}. 
\] 
Thus we obtain 
\[
\begin{split}
\gamma ({}_\alpha E)&= 
-\frac{1}{2} \lim_{z\to 0} \frac{\log \det H(\bm w)(z)}{\log |z|}\\ 
&=-\frac{1}{2} \lim_{z\to 0} \frac{\log \det H(\bm v)(z)}{\log |z|}
-m-\lim _{z\to 0} \frac{\log |f(z)|}{\log |z|}\\ 
&=\gamma ({}_\beta E)-m. 
\end{split}
\]
This implies 
\[
\gamma ({}_\beta E)-\gamma ({}_\alpha E)=m\in \mathbb Z_{\geq 0}. 
\] 
Thus we have (ii). 
By the above argument, ${}_\alpha E={}_\beta E$ if and only if 
$\det A(0)\ne 0$, 
Moreover, $\det A(0)\ne 0$ if and only if $m=0$. 
Thus, ${}_\alpha E={}_\beta E$ if and only if 
$\gamma ({}_\alpha E)=\gamma ({}_\beta E)$. This is (iii). 
Since $\bm w=\{w_1, \ldots, w_r\}$ is a frame of ${}_\alpha E$, 
we can easily check that 
\[
\frac{\bm w}{z}=\left\{\frac{w_1}{z}, \ldots, \frac{w_r}{z}\right\}
\] 
is a frame of ${}_{\alpha +1}E$ on $\Delta^*$. 
Thus we can directly check that 
\[
{}_{\alpha +1} E={}_\alpha E\otimes \mathcal O_\Delta([0]) 
\] 
and 
\[
\gamma ({}_{\alpha +1}E)=\gamma ({}_\alpha E)+r.
\]  
Thus we obtain (iv) and (v). We finish the proof of Lemma \ref{p-lem7.8}. 
\end{proof}

\begin{lem}\label{p-lem7.9}
Let $(E, h)$ be an acceptable vector bundle on $\Delta^*$. 
Then, for every $\alpha \in \mathbb R$, 
\[
{}_\alpha E=\bigcap _{\beta>\alpha} {}_\beta E
\] 
holds. 
\end{lem}
\begin{proof}[Proof of Lemma \ref{p-lem7.9}]
Since ${}_\alpha E\subset {}_\beta E$ for $\beta>\alpha$ by Lemma \ref{p-lem7.8} (i), 
we have 
\[
{}_\alpha E\subset \bigcap _{\beta>\alpha} {}_\beta E. 
\] 
From now, we will prove the opposite inclusion. 
We take $v\in \bigcap _{\beta>\alpha} {}_\beta E$. 
Let $\varepsilon$ be any positive real number. 
We can take $\beta'$ such that 
$\alpha <\beta'<\alpha +\varepsilon$ and 
$\varepsilon'$ such that $0<\varepsilon'<\alpha +\varepsilon -\beta$. 
There exists some open neighborhood $U$ of $0$ such that 
$v\in \Gamma (U, {}_{\beta'} E)$. 
Then 
\[|v|_h|z|^{\beta'+\varepsilon'}<C
\] 
holds for some positive real number $C$. 
Hence we have 
\[
|v|_h|z|^{\alpha +\varepsilon}=|v|_h|z|^{\beta'+\varepsilon'}
|z|^{\alpha +\varepsilon -\beta'-\varepsilon'}<C
\] 
since $\alpha+\varepsilon-\beta'-\varepsilon'>0$ and $|z|<1$. 
This implies 
\[
|v|_h=O\left(\frac{1}{|z|^{\alpha +\varepsilon}}\right). 
\] 
Thus we obtain $v\in {}_\alpha E$. 
We finish the proof of Lemma \ref{p-lem7.9}. 
\end{proof}

\begin{lem}\label{p-lem7.10}
Let $(E, h)$ be an acceptable vector bundle on $\Delta^*$. 
Then, for every $\alpha \in \mathbb R$, 
there exists $\delta>0$ such that 
\[
{}_\alpha E={}_{\alpha +\varepsilon}E
\] 
holds for every $\varepsilon \in [0, \delta)$. 
\end{lem}
\begin{proof}[Proof of Lemma \ref{p-lem7.10}]
Note that $\gamma ({}_{\alpha +\varepsilon}E)-\gamma ({}_\alpha E)$ 
is a $\mathbb Z$-valued 
non-decreasing function for $\varepsilon \in \mathbb R$. 
As we already proved in Theorem \ref{p-lem7.8} (v), 
$\gamma ({}_{\alpha +1}E)-\gamma ({}_\alpha E)=\rank E$. 
Hence there exists $\delta>0$ such that 
$\gamma (_{\alpha +\varepsilon}E)$ is constant for every 
$\varepsilon \in (0, \delta)$. 
Since 
\[
\bigcap _{\varepsilon \in (0, \delta)}{}_{\alpha +\varepsilon}E=
\bigcap _{\beta>\alpha} {}_{\beta} E={}_\alpha E   
\] 
by Lemma \ref{p-lem7.9}, 
we have ${}_{\alpha +\varepsilon} E={}_\alpha E$ holds for 
$\varepsilon \in (0, \delta)$. 
We finish the proof of Lemma \ref{p-lem7.10}. 
\end{proof}

\begin{say}[Parabolic weights]\label{p-say7.11}
We set
\[
\Parr_\alpha(E, h) := \left\{ \lambda \in (\alpha - 1, \alpha] \,\middle|\,
{}_\lambda E / {}_{<\lambda}E \neq 0 \right\},
\]
where
\[
{}_{<\lambda} E := \bigcup_{\mu < \lambda} {}_\mu E \subset {}_\lambda E.
\]
Then we obtain
\[
\Parr_\alpha(E, h) = \{ \lambda_1, \ldots, \lambda_k \},
\]
with $\lambda_i \neq \lambda_j$ for $i \neq j$.

If there is no risk of confusion, 
we simply write $\Parr_\alpha(E)$ or $\Parr({}_\alpha E)$ instead of $\Parr_\alpha(E, h)$.

We set
\[
l_i := \dim_{\mathbb{C}}\left({}_{\lambda_i}E / {}_{<\lambda_i}E\right).
\]
Then we can verify, by Lemma~\ref{p-lem7.8} (v), that
\begin{equation}\label{p-eq7.15}
\sum_{i=1}^k l_i = r = \rank E.
\end{equation}

Thus, we define
\[
\Par_\alpha(E, h) := \{\underbrace{\lambda_1, \ldots, \lambda_1}_{l_1\text{ times}}, \ldots,
\underbrace{\lambda_k, \ldots, \lambda_k}_{l_k\text{ times}}\}.
\]
If there is no risk of confusion, 
we write $\Par_\alpha(E)$ or $\Par({}_\alpha E)$ for $\Par_\alpha(E, h)$.

Furthermore, if the multiplicity 
of $\lambda_i$ is not important in the 
context, we may also use $\Par_\alpha(E, h)$ to 
denote $\Parr_\alpha(E, h)$.
\end{say}

\begin{lem}\label{p-lem7.12}
For every $i$, we 
have 
\[\gamma ({}_{\lambda_i}E)-\gamma ({}_{\lambda_i -\varepsilon}E)=l_i
\] 
for $0<\varepsilon \ll 1$. 
Therefore, we have 
\[
\gamma ({}_\beta E)-\gamma ({}_\alpha E)=\dim _{\mathbb C} 
\left({}_\beta E/{}_\alpha E\right)
\] 
for every $\beta\geq \alpha$. 
\end{lem}

\begin{proof}[Proof of Lemma \ref{p-lem7.12}]
We fix a sufficiently small positive real number $\varepsilon$ such that 
\[\dim _{\mathbb C} ({}_{\lambda_i}E/{}_{\lambda_i-\varepsilon} E)=l_i.
\] 
Let $\bm v:=\{v_1, \ldots, v_r\}$ be a frame 
of ${}_{\lambda_i} E$ around $0$ and let 
$\bm w:=\{w_1, \ldots, w_r\}$ be a frame of 
${}_{\lambda_i-\varepsilon} E$ around $0$. 
Then we can write 
\[
(w_1, \ldots, w_r)=(v_1, \ldots, v_r)A(z)
\] 
for some $r\times r$ matrix $A(z)$ around $0$. 
Since $\dim _{\mathbb C} ({}_{\lambda_i}E/{}_{\lambda_i-\varepsilon} E)=l_i$, 
we obtain $\rank A(0)=r-l_i$. 
According to the theory of elementary divisors in the ring $\mathbb{C}\{z\}$, 
we can write 
\[
\det A(z)=z^{m_i}f(z)
\] 
such that $m_i\geq l_i$ and $f(z)$ is holomorphic with $f(0)\ne 0$. 
Hence we can prove that 
\[
\gamma ({}_{\lambda_i} E)-\gamma ({}_{\lambda_i-\varepsilon} E)
=m_i\geq l_i. 
\] 
Thus, by Lemma \ref{p-lem7.8} (v) and 
\eqref{p-eq7.15}, we obtain 
\[
r=\gamma ({}_a E)-\gamma ({}_{a-1}E)=\sum _{i=1}^k m_i \geq 
\sum _{i=1}^k l_i=r. 
\] 
This implies that $m_i=l_i$ for every $i$, that is, 
\[\gamma ({}_{\lambda_i}E)-\gamma ({}_{\lambda_i -\varepsilon}E)=l_i. 
\] 
This is what we wanted. We finish the proof of Lemma \ref{p-lem7.12}. 
\end{proof}

We will prove the following important formula in 
Section \ref{p-sec12}, which plays a crucial role for 
the study of ${}_{\alpha} E$. 
The proof of Theorem \ref{p-thm7.13} is much more difficult than the 
argument in this section. 

\begin{thm}[see Theorem \ref{p-thm1.9} and 
Theorem \ref{p-thm12.3} below]\label{p-thm7.13}
Let $E$ be an acceptable vector bundle on $\Delta^*$. 
Then the following equality 
\[
\gamma ({}_\alpha E)=\sum _{\lambda_i\in \Parr ({}_\alpha E)} 
\lambda_i\dim_\mathbb C \left({}_{\lambda_i} E/{}_{<\lambda_i} E\right)
\] 
holds. 
\end{thm}

Sections \ref{p-sec9} through \ref{p-sec12} will be devoted to the proof of Theorem 
\ref{p-thm7.13}. 

\begin{defn}\label{p-def7.14}
Let $\bm v=\{v_1, \ldots, v_r\}$ be a frame of ${}_\alpha E$, 
that is, 
\[
{}_\alpha E=\bigoplus _{i=1}^r\mathcal O_\Delta\cdot v_i. 
\] 
If there exists a decomposition 
\[
\bm v=\bigsqcup _{\alpha -1< \beta\leq \alpha}\bm v_{\beta}
\] 
such that $\bm v_\beta$ is a tuple of sections of ${}_\beta E$ 
and that $\bm v_\beta$ induces a basis of ${}_\beta E/{}_{<\beta} E$, then 
$\bm v$ is called a frame of ${}_\alpha E$ 
{\em{compatible with the parabolic filtration}}. 
Note that $\bm v_\beta=\emptyset$ if $\beta \not \in \Par _{\alpha} (E, h)$. 
\end{defn}

\begin{rem}\label{p-rem7.15}
Since ${}_\alpha E$ is a holomorphic vector bundle on $\Delta$, we can 
always take a trivialization (see, for example, \cite[30.4.~Theorem]{forster}). 
Therefore, 
there exists a frame $\{e_1, \ldots, e_r \}$ of ${}_\alpha E$ on $\Delta$, 
that is, 
\[
{}_\alpha E= \bigoplus _{i=1} ^r \mathcal O_{\Delta} \cdot e_i 
\] 
holds. Note that $\{e_1, \ldots, e_r\}$ gives a 
basis of the quotient vector space ${}_\alpha E /{}_{\alpha-1}E$. 
Thus, we can take $(a_{ij})\in \mathrm{GL}(r, \mathbb C)$ such that 
$\{v_1, \ldots, v_r\}$, where 
$v_j :=\sum _{i=1} ^r e_i a_{ij}$ for every $j$, 
gives a frame of ${}_\alpha E$ compatible with 
the parabolic filtration. 
\end{rem}

\begin{lem}\label{p-lem7.16}
Let $\bm v = \{v_1, \ldots, v_r\}$ be a local frame 
of ${}_\alpha E$ defined over some open neighborhood 
of $0$, such that 
$v_i \in {}_{\beta_i}E \setminus {}_{<\beta_i} E$ for every $i$. 
Then $\beta_i \in (\alpha -1, \alpha]$ for every $i$. 
\end{lem}

\begin{proof}[Proof of Lemma \ref{p-lem7.16}] 
Since $v_i \in {}_\alpha E$, it follows that 
$\beta_i \leq \alpha$ for all $i$.  
Moreover, since $\{v_1, \ldots, v_r\}$ forms 
a local frame of ${}_\alpha E$ near $0$, it induces a basis of the quotient vector space 
${}_\alpha E / {}_{\alpha - 1} E$.  
This implies that each $\beta_i$ lies in the interval $(\alpha - 1, \alpha]$. 
\end{proof}

\begin{lem}\label{p-lem7.17}
Let $\bm{v}=\{v_1, \ldots, v_r\}$ be a frame of ${}_\alpha E$ compatible with 
the parabolic filtration such that 
\[
v_i\in {}_{\beta_i}E\setminus {}_{<\beta_i}E 
\] 
for every $i$. 
In particular, 
\[
\Par _{\alpha}(E, h)=\{\beta_1, \ldots, \beta_r\}. 
\]
Let $\alpha'$ be any real number. 
Let $m_i$ be the smallest integer satisfying $\beta_i-m_i\leq 
\alpha'$ for every $i$. 
Then 
\[
\bm{v}':=\{z^{m_1}v_1, \ldots, z^{m_r}v_r\}
\] 
is a frame of ${}_{\alpha'} E$ compatible 
with the parabolic filtration. 
\end{lem}

\begin{proof}[Proof of Lemma \ref{p-lem7.17}]
We put $w_i:=z^{m_i}v_i$ for every $i$. 
We note that the map 
\begin{equation}\label{p-eq7.16}
z^m\times \colon {}_\lambda E\to {}_{\lambda-m} E
\end{equation} 
is an isomorphism for every $m\in \mathbb Z$ and 
every real number $\lambda$. By definition, 
we see that $w_i\in {}_{\alpha'} E$ for every $i$. 
By the isomorphism \eqref{p-eq7.16}, we can check 
that ${}_{\alpha'} E$ is spanned by $\bm v'$. 
Thus we have 
\[
{}_{\alpha'}E=\bigoplus _{i=1}^r \mathcal O_{\Delta}\cdot w_i, 
\] 
that is, $\bm v'$ is a frame of ${}_{\alpha'}E$. 
By \eqref{p-eq7.16} again, 
we can check that $\bm v'$ is a frame 
of ${}_{\alpha'} E$ compatible with the parabolic 
filtration. 
We finish the proof of Lemma \ref{p-lem7.17}
\end{proof}

We conclude this section with the following remark.

\begin{rem}\label{p-rem7.18}
The acceptability near the origin  
is preserved under the coordinate  
rescaling \( z \mapsto z/C \), where \( C \) is a positive constant.  
Note that the condition  
\( v \in {}_\alpha E \setminus {}_{<\alpha} E \) and  
the quantity  
\[
\gamma({}_\alpha E) = -\frac{1}{2} \lim_{z \to 0} \frac{\log \det H(\bm{v})}{\log |z|} 
\]  
are invariant under this rescaling.  
Therefore, such rescaling  
can be employed when we are concerned only with the behavior near the origin.
\end{rem}

\section{On filtered prolongation of acceptable bundles}\label{p-sec8}

In this short section, we recall the framework of filtered 
bundles as introduced by Mochizuki.  
His notation turns out to be particularly convenient in various contexts.

We have already verified the following properties of ${}_a E$.

\begin{say}[Filtered prolongation of acceptable bundles]\label{p-say8.1}
Let $(E, h)$ be an acceptable vector bundle on $\Delta^*$.  
We define
\[
\mathcal{P}^h_a E := {}_a E
\]
for every $a \in \mathbb{R}$, and set
\[
\mathcal{P}^h E := \bigcup_{a \in \mathbb{R}} \mathcal{P}^h_a E \subset j_* E,
\]
where $j \colon \Delta^* \hookrightarrow \Delta$ is the natural inclusion.  

Then, $\mathcal{P}^h E$ is a locally free 
$\mathcal{O}_\Delta(*[0])$-module of finite rank, where 
$\mathcal O_\Delta(\ast [0])$ is the sheaf of 
meromorphic functions on $\Delta$ with poles only at $0$.  
The following properties hold:
\begin{itemize}
  \item[(i)] For every $a \in \mathbb{R}$, $\mathcal{P}^h_a E$ 
  is a locally free $\mathcal{O}_\Delta$-submodule of $\mathcal{P}^h E$.
  
  \item[(ii)] $\mathcal{P}^h_a E(*[0]) = \mathcal{P}^h E$ for every 
  $a\in \mathbb R$.
  
  \item[(iii)] For any $a \leq b$, we have $\mathcal{P}^h_a E \subset \mathcal{P}^h_b E$.
  
  \item[(iv)] For any $a \in \mathbb{R}$ and $n \in \mathbb{Z}$, we have
  \[
  \mathcal{P}^h_{a+n} E = \mathcal{P}^h_a E(n[0]).
  \]
  
  \item[(v)] For any $a \in \mathbb{R}$, 
  there exists $\varepsilon > 0$ such that
  \[
  \mathcal{P}^h_{a+\varepsilon} E = \mathcal{P}^h_a E.
  \]
\end{itemize}
\end{say}

Therefore, it is natural to introduce the notion of filtered bundles as follows.

\begin{defn}[Filtered bundles]\label{p-def8.2}
We denote by $\mathcal{O}_\Delta$ the sheaf of 
holomorphic functions on $\Delta$, and by 
$\mathcal{O}_\Delta(*[0])$ the sheaf of meromorphic functions on $\Delta$ with poles only at $0$. 

Let $\mathcal{E}$ be a locally free $\mathcal{O}_\Delta(*[0])$-module.  
A {\em filtered bundle over} $\mathcal{E}$ is 
an increasing family of locally free $\mathcal{O}_\Delta$-modules 
$\mathcal{P}_a \mathcal{E} \subset \mathcal{E}$ 
indexed by $a \in \mathbb{R}$, satisfying the following conditions:
\begin{itemize}
\item[(1)] Each $\mathcal{P}_a \mathcal{E}$ is a lattice in $\mathcal{E}$, i.e.,
\[
\mathcal{P}_a \mathcal{E} \otimes_{\mathcal{O}_\Delta} 
\mathcal{O}_\Delta(*[0]) = \mathcal{E}.
\]
\item[(2)] For any $a \in \mathbb{R}$ and $n \in \mathbb{Z}$, we have
\[
\mathcal{P}_{a+n} \mathcal{E} = 
\mathcal{P}_a \mathcal{E} \otimes_{\mathcal{O}_\Delta} \mathcal{O}_\Delta(n  [0]).
\]
\item[(3)] For any $a \in \mathbb{R}$, 
there exists $\epsilon > 0$ such that
\[
\mathcal{P}_{a+\epsilon} \mathcal{E} 
= \mathcal{P}_a \mathcal{E}.
\]
\end{itemize}
In this case, we also say 
that $\mathcal{P}_\ast \mathcal{E}$ is a 
filtered bundle on $(\Delta, 0)$ for simplicity.

For any $a \in \mathbb{R}$, define
\[
\mathcal{P}_{<a} \mathcal{E} := \sum_{b < a} \mathcal{P}_b \mathcal{E},
\quad \text{and} \quad
\mathrm{Gr}^\mathcal{P}_a(\mathcal{E}) := 
\mathcal{P}_a \mathcal{E} / \mathcal{P}_{<a} \mathcal{E}.
\]
We may naturally regard $\mathrm{Gr}^\mathcal{P}_a(\mathcal{E})$ as a 
finite-dimensional $\mathbb{C}$-vector space.

A frame $\bm{v} = \{v_1, \dots, v_{\rank\mathcal{E}}\}$ of 
$\mathcal{P}_a \mathcal{E}$ is said to be 
{\em compatible} with the parabolic structure if there exists a decomposition
\[
\bm{v} = \bigsqcup_{a-1 < b \leq a} \bm{v}_b
\]
such that the following holds:
\begin{itemize}
\item For each $b$, $\bm{v}_b$ is a tuple of 
sections of $\mathcal{P}_b \mathcal{E}$, and 
it induces a basis of $\mathrm{Gr}^\mathcal{P}_b(\mathcal{E})$.
\end{itemize}

For any non-zero section $s$ of $\mathcal{E}$, the number
\[
\deg^\mathcal{P}(s) := \min\left\{ c \in \mathbb{R} \mid 
s \in \mathcal{P}_c \mathcal{E} \right\}
\]
is called the {\em parabolic degree} of $s$.  
If $s = 0$, we set $\deg^\mathcal{P}(s) := -\infty$.
\end{defn}

By Definition \ref{p-def8.2}, we can say that 
$\mathcal P^h_*E=\left(\mathcal P^h_a E\mid a\in \mathbb R\right)$ 
is a {\em{filtered bundle over}} $\mathcal P^hE$. 

\begin{rem}\label{p-rem8.3}
Definition \ref{p-def8.2} is 
essentially the same as 
\cite[2.11.1, Filtered Bundles on a Neighborhood of $0$ in $\mathbb{C}$]{mochizuki6}.  
It is a local definition. For the global setting, see \cite[2.11.3, Global Case]{mochizuki6}.  
In this paper, we are only concerned with the one-dimensional case.  
For the higher-dimensional case, we refer the reader to Section~2 of \cite{mochizuki5} (see also \cite[Section 4]{ffo}).
\end{rem}

In the following sections, we will 
use whichever of the notations ${}_a E$ and $\mathcal{P}^h_a E$ 
is more convenient in context. 
In particular, when discussing tensor products 
in Section \ref{p-sec16} and Hom bundles in Section \ref{p-sec17}, 
the notation $\mathcal{P}^h_a E$ appears to be more suitable.

\section{Some elementary inequalities}\label{p-sec9}

In this section, we present some elementary facts that will be used later.  
The arguments in this section are essentially due to Simpson~\cite{simpson3}.

We denote
\[
B(a, r) := \{ z \in \mathbb{C} \mid |z - a| < r \}, \quad
\overline{B}(a, r) := \{ z \in \mathbb{C} \mid |z - a| \leq r \},
\]
and let $\Area(\Omega)$ denote the area of a set $\Omega$.

\begin{defn}
\label{p-def9.1}
Fix a positive real number $r$.
We define
\[
B_r(w) :=
\inf\left\{
\frac{1}{r^3} \int_{\Omega} \log |w - z| \, d\lambda(z)
\ \middle| \
\Omega \subset B(0, r),\ \Area(\Omega) = r^3
\right\},
\]
for $w \in \mathbb C$, 
where $\Omega$ is an open subset of $\mathbb C$, and $d\lambda$ denotes the Lebesgue 
measure on $\mathbb C \simeq \mathbb R^2$.
\end{defn}

The following estimate is straightforward.

\begin{lem}
\label{p-lem9.2}
For any $w \in \mathbb C$ and any positive real number $r$, we have
\[
B_r(w)
\ge
\frac{3}{2} \log r - \frac{1}{2} \log \pi - \frac{1}{2}.
\]
\end{lem}

\begin{proof}[Proof of Lemma \ref{p-lem9.2}]
By definition, we have
\begin{align*}
B_r(w)
&\ge
\inf\left\{
\frac{1}{r^3} \int_{\Omega} \log |w - z| \, d\lambda(z)
\ \middle| \ 
\Area(\Omega) = r^3
\right\} \\
&=
\inf\left\{
\frac{1}{r^3} \int_{\Omega} \log |z| \, d\lambda(z)
\ \middle| \
\Area(\Omega) = r^3
\right\},
\end{align*}
where the second equality follows by translation invariance of Lebesgue measure.

It is easy to see 
that the minimum is attained when $\Omega = B(0,a)$ with $a = \pi^{-1/2} r^{3/2}$. Therefore,
\[
\begin{aligned}
B_r(w)
&\ge
\frac{1}{r^3} \int_{B(0, a)} \log |z| \, d\lambda(z) \\
&=
\frac{1}{r^3} \int_0^{2\pi} d\theta \int_0^a t \log t \, dt \\
&=
\frac{2\pi}{r^3} \left( \left[ \frac{1}{2} t^2 \log t \right]_0^a - \frac{1}{2} 
\int_0^a t \, dt \right) \\
&=
\frac{2\pi}{r^3} \left( \frac{1}{2} a^2 \log a - \frac{1}{4} a^2 \right) \\
&=
\frac{3}{2} \log r - \frac{1}{2} \log \pi - \frac{1}{2},
\end{aligned}
\]
as claimed. This completes the proof of Lemma~\ref{p-lem9.2}.
\end{proof}

\begin{lem}
\label{p-lem9.3}
Let $r > 0$. Then for every $w \in \overline{B}(0, 2)$, the following inequality holds:
\[
B_r(w)
\ge
\frac{3}{2} \log \left( \frac{|w|}{2} \right) - \frac{1}{2} \log \pi - \frac{1}{2}.
\]
\end{lem}

\begin{proof}[Proof of Lemma \ref{p-lem9.3}]
If $|w| \le 2r$, then the conclusion 
follows directly from Lemma~\ref{p-lem9.2} and the inequality $\log(|w|/2) \le \log r$.

If $|w| \ge 2r$, then for any $z \in B(0, r)$,
\[
|w - z| \ge |w| - |z| \ge |w| - r \ge \frac{|w|}{2}.
\]
Thus,
\[
\begin{aligned}
B_r(w)
&\ge
\inf\left\{
\frac{1}{r^3} \int_{\Omega} \log\left( \frac{|w|}{2} \right) \, d\lambda(z)
\ \middle| \
\Omega \subset B(0, r),\ \Area(\Omega) = r^3
\right\} \\
&=
\log\left( \frac{|w|}{2} \right).
\end{aligned}
\]

Combining both cases, we obtain the claimed 
inequality using the fact that 
$\log(|w|/2) \le 0$ for $w \in \overline{B}(0,2)$ and
\[
\frac{1}{2} \log \pi + \frac{1}{2} > 0.
\]
This completes the proof of Lemma~\ref{p-lem9.3}.
\end{proof}

\begin{lem}
\label{p-lem9.4}
Let $r \in \mathbb R$ with $0 < r < 1$. Then for all 
$z, w \in \mathbb C$ with $r \le |z| < 1$ and 
$|w| \le 2$, the following inequality holds:
\[
\log |w - z|
\le
\frac{2}{3} \cdot \frac{\log |z|}{\log r} \cdot B_r(w)
+ \frac{1}{3} \log \pi + \frac{1}{3} + 2 \log 2.
\]
\end{lem}

\begin{proof}[Proof of Lemma \ref{p-lem9.4}]
Since $r \le |z| < 1$, we have $\log r \le \log |z| < 0$, 
hence \[0 < \frac{\log |z|}{\log r} \le 1.\]

If $|w| \le |z|$, then
\[
\log |w - z| \le \log(2|z|) = \log |z| + \log 2 
= \frac{\log |z|}{\log r} \log r + \log 2.
\]
Applying Lemma~\ref{p-lem9.2}, we obtain
\begin{align*}
\log |w - z|
&\le
\frac{\log |z|}{\log r} \cdot \frac{2}{3} \left( B_r(w) 
+ \frac{1}{2} \log \pi + \frac{1}{2} \right) + \log 2 \\
& \leq  
\frac{2}{3} \cdot \frac{\log |z|}{\log r} \cdot B_r(w)
+ \frac{1}{3} \log \pi + \frac{1}{3} + \log 2.
\end{align*}

If $|w| \ge |z|$, then
\[
\log |w - z| \le \log(2|w|) = \log 
\left( \frac{|w|}{2} \right) + 2 \log 2,
\]
and since $\log(|w|/2) \le 0$ and $\log |z|/\log r \le 1$, we get
\[
\log |w - z| \le \frac{\log |z|}{\log r} 
\cdot \log \left( \frac{|w|}{2} \right) + 2 \log 2.
\]
Applying Lemma~\ref{p-lem9.3}, we obtain
\begin{align*}
\log |w - z|
&\le
\frac{\log |z|}{\log r} \cdot \frac{2}{3} 
\left( B_r(w) + \frac{1}{2} \log \pi + \frac{1}{2} \right) + 2 \log 2 \\
& \leq 
\frac{2}{3} \cdot \frac{\log |z|}{\log r} \cdot B_r(w)
+ \frac{1}{3} \log \pi + \frac{1}{3} + 2 \log 2.
\end{align*}
This completes the proof of Lemma~\ref{p-lem9.4}.
\end{proof}

\section{Simpson's key lemma}\label{q-sec10}

The main goal of this section is to establish the following key lemma (see 
Lemma \ref{q-lem10.1}), which is essentially due to Simpson.  
Note that our version is slightly different from 
the original statement (see \cite[Lemma 10.2]{simpson1}).  
However, our formulation of Lemma \ref{q-lem10.1} is sufficient 
for the proof of Theorem \ref{p-thm12.3}.

In this section, we frequently use the following notation:
\[
S(a, r) := \{ z \in \mathbb{C} \mid |z - a| = r \},
\]
and 
\[
B(a, r)^* := \{ z \in \mathbb{C} \mid 0 < |z - a| <r \} = B(a, r) \setminus \{a\}.
\]

\begin{lem}[{\cite[Lemma 10.2]{simpson1}}]\label{q-lem10.1}
Let $\delta$ be a positive real number with $\delta < 1$.  
Suppose that $h$ is a smooth Hermitian metric on the trivial holomorphic vector bundle  
$\mathcal{O}_{B(0,1+2\delta)^*}^{\oplus k}$  
over the punctured disk  
$B(0,1+2\delta)^* := B(0,1+2\delta) \setminus \{0\}$,  
and that $h$ has negative curvature.  
Assume further that the eigenvalues of $h$ are less than or 
equal to $1$, and that  
\[
|\det h| \leq C |z|
\]  
holds for some positive constant $C$.  
Then there exist a positive constant $C'$ and a constant section 
$e \in \mathbb{C}^k$ of $\mathcal O^{\oplus k}_{B(0, 1+2\delta)^*}$ 
such that  
\[
|e(z)|_h \leq C' |z|^{\frac{1}{3k}}
\]  
for all $z \in B(0,1)^*$.
\end{lem}

Before starting the proof of the lemma above,
we need to prove several preliminary results.

\begin{lem}\label{q-lem10.2} 
Let \( r_1 \) and \( \delta_1 \) be positive real numbers. 
Let \( u \) be a subharmonic function defined on \( B(0, r_1 + \delta_1) \) 
such that \( u \) is smooth outside the origin. 
Let \( f(z) \) be a smooth function on \( B(0, r_1 + \delta_1) \). 
Define
\[
\sigma(z) := \partial \overline{\partial} u 
= \frac{\partial^2 u}{\partial z \, \partial \overline{z}} \, dz \wedge d\overline{z} 
= \frac{1}{4} \Delta u \, dz \wedge d\overline{z}.
\]
Note that
\[
\Delta = 4 \frac{\partial^2 }{\partial z \, \partial \overline{z}}
\]
denotes the Laplacian with respect to \( z \), understood in the sense of distributions. 
Since \( u \) is subharmonic, \( \sqrt{-1} \partial \overline{\partial} u \) 
is a closed positive \((1,1)\)-current. Hence, \( \sqrt{-1} \sigma(z) \) defines 
a positive Radon measure {\em{(}}see, for example, 
\cite[(3.1.14) Lemma]{noguchi-ochiai}{\em{)}}. In this setting, we have
\begin{equation}\label{q-eq10.1} 
\begin{split}
&\int_{\overline{B}(0, r_1)} f(z) \, \sigma(z) 
- \frac{1}{4} \int_{\overline{B}(0, r_1)} (\Delta f) \, u(z) \, dz \wedge d\overline{z} 
\\&
= \int_{S(0, r_1)} \left( \frac{\partial u}{\partial \overline{z}} f(z) \, d\overline{z} 
+ \frac{\partial f}{\partial z} u(z) \, dz \right).
\end{split}
\end{equation}
\end{lem}

For the sake of completeness, we provide a detailed proof of Lemma~\ref{q-lem10.2}, 
although it is more or less standard.

\begin{proof}[Proof of Lemma~\ref{q-lem10.2}]
Take a smooth function \( \varphi(z) \) on \( \mathbb{C} \) such that 
\( \supp \varphi \subset B(0, r_1) \), and \( \varphi(z) = 1 \) on 
\( B(0, \tfrac{1}{2} r_1 + \delta_2) \) for some small constant 
\( 0 < \delta_2 \ll 1 \). Define \( g(z) := \varphi(z) f(z) \) and 
\( h(z) := f(z) - g(z) \). Then \( f(z) = g(z) + h(z) \), where 
\( \supp g \subset B(0, r_1) \) and \( h(z) = 0 \) on 
\( B(0, \tfrac{1}{2} r_1 + \delta_2) \).

Let \( \rho \) be a smooth function on \( \mathbb{C} \), supported in \( B(0,1) \), 
radial (i.e., \( \rho(z) \) depends only on \( |z| \)), non-negative, and normalized so that
\[
\int_{\mathbb{C}} \rho(z) \, d\lambda(z) = 1,
\]
where \( d\lambda(z) := \frac{\sqrt{-1}}{2} \, dz \wedge d\overline{z} \) 
denotes the Lebesgue measure on \( \mathbb{C} \). Define the family of smoothing kernels 
\[
\rho_\varepsilon(z) := \frac{1}{\varepsilon^2} \rho\left( \frac{z}{\varepsilon} \right).
\]

Set \( u_\varepsilon := u * \rho_\varepsilon \). 
Then \( u_\varepsilon \) is a smooth subharmonic 
function on a neighborhood of \( \overline{B}(0, r_1) \) 
for sufficiently small \( \varepsilon > 0 \).

By applying Stokes' theorem to \( g(z) \) and \( u_\varepsilon \), we obtain:
\begin{equation}\label{q-eq10.2} 
\begin{split}
&\int_{\overline{B}(0, r_1)} g(z) \, \partial \overline{\partial} u_{\varepsilon} 
- \frac{1}{4} \int_{\overline{B}(0, r_1)} 
(\Delta g) \, u_{\varepsilon}(z) \, dz \wedge d\overline{z} \\
&= \int_{S(0, r_1)} 
\left( \frac{\partial u_\varepsilon}{\partial \overline{z}} g(z) \, d\overline{z} 
+ \frac{\partial g}{\partial z} u_\varepsilon(z) \, dz \right).
\end{split}
\end{equation}

We note that we have 
\[
\frac{\partial u_\varepsilon}{\partial \overline{z}} = 
\left( \frac{\partial u}{\partial \overline{z}} \right) * \rho_\varepsilon, 
\] 
where \( \frac{\partial u}{\partial \overline{z}} \) is taken 
in the sense of distributions. Since \( u \) is smooth outside the origin, 
both \( u_\varepsilon(z) \) and \( \frac{\partial u_\varepsilon}{
\partial \overline{z}} \) converge uniformly to \( u(z) \) and 
\( \frac{\partial u}{\partial \overline{z}} \), 
respectively, on an open neighborhood of \(S(0, r_1)\) as \( \varepsilon \to +0\). 
It is well known that 
\( \partial \overline{\partial} u_\varepsilon \to \partial \overline{\partial} u \) 
in the sense of currents, and \( u_\varepsilon \to u \) in the sense of distributions.

Since \( g \) is smooth with compact support, we may 
let \( \varepsilon \to +0 \) in \eqref{q-eq10.2} to obtain:
\begin{equation}\label{q-eq10.3} 
\begin{split}
&\int_{\overline{B}(0, r_1)} g(z) \, \partial \overline{\partial} u 
- \frac{1}{4} \int_{\overline{B}(0, r_1)} (\Delta g) \, u(z) \, dz \wedge d\overline{z} \\
&= \int_{S(0, r_1)} \left( \frac{\partial u}{\partial \overline{z}} g(z) \, d\overline{z} 
+ \frac{\partial g}{\partial z} u(z) \, dz \right).
\end{split}
\end{equation}

Next, since \( h(z) = 0 \) and \( \Delta h = 0 \) 
on \( B(0, \tfrac{1}{2} r_1 + \delta_2) \), we can apply Stokes' theorem to get:
\begin{equation*}
\begin{split}
&\int_{\overline{B}(0, r_1)} h(z) \, \sigma(z) 
- \frac{1}{4} \int_{\overline{B}(0, r_1)} (\Delta h) \, u(z) \, dz \wedge d\overline{z} \\
&= \frac{1}{4} \int_{\overline{B}(0, r_1) \setminus B(0, \tfrac{1}{2} r_1)} 
\left( h(z) \Delta u - u(z) \Delta h \right) \, dz \wedge d\overline{z} \\
&= \int_{S(0, r_1)} 
\left( \frac{\partial u}{\partial \overline{z}} 
h(z) \, d\overline{z} 
+ \frac{\partial h}{\partial z} u(z) \, dz \right) 
-\int_{S(0, \tfrac{1}{2} r_1)} 
\left( \frac{\partial u}{\partial \overline{z}} h(z) \, d\overline{z} 
+ \frac{\partial h}{\partial z} u(z) \, dz \right).
\end{split}
\end{equation*}
Since \( h(z) = 0 \) and \( \partial h / \partial z = 0 \) 
on a neighborhood of \( S(0, \tfrac{1}{2} r_1) \), the second boundary integral vanishes. Hence,
\begin{equation}\label{q-eq10.4}
\begin{split}
&\int_{\overline{B}(0, r_1)} h(z) \, \sigma(z) 
- \frac{1}{4} \int_{\overline{B}(0, r_1)} (\Delta h) \, u(z) \, dz 
\wedge d\overline{z} \\
&= \int_{S(0, r_1)} \left( \frac{\partial u}{\partial \overline{z}} h(z) \, d\overline{z} 
+ \frac{\partial h}{\partial z} u(z) \, dz \right).
\end{split}
\end{equation}
By adding \eqref{q-eq10.3} and \eqref{q-eq10.4}, we obtain the desired equality \eqref{q-eq10.1}.

This completes the proof of Lemma~\ref{q-lem10.2}.
\end{proof}

\begin{lem}
\label{q-lem10.3}
Let $\delta > 0$ be a real number, and let $u$ be a subharmonic function
on $B(0,1+2\delta)$ that is smooth on $B(0,1+2\delta)^* := B(0,1+2\delta) \setminus \{0\}$.  
Then for every $a \in B(0,1)^*$, we have
\begin{equation}\label{q-eq10.5}
\begin{split}
u(a)
&= \frac{\sqrt{-1}}{\pi} \int_{B(0,1+\delta)} \log |z-a| \, \sigma(z) \\
&\quad + \frac{1}{\pi \sqrt{-1}} 
\int_{S(0,1+\delta)} 
\frac{\partial u}{\partial \overline{z}} \log|z-a|\, d\overline{z}
+ \frac{1}{2\pi \sqrt{-1}} \int_{S(0,1+\delta)} \frac{1}{z-a} u(z)\,dz.
\end{split}
\end{equation}
Here, $\sqrt{-1}\, \sigma(z) := \sqrt{-1}\, \partial \overline{\partial} u$, computed 
in the sense of currents, defines a positive Radon measure on $B(0,1+2\delta)$ 
because $u$ is subharmonic {\em{(}}see, 
for example, \cite[(3.1.14) Lemma]{noguchi-ochiai}{\em{)}}.
\end{lem}

\begin{proof}[Proof of Lemma \ref{q-lem10.3}]
Let $r_1$ and $r_2$ be small positive real numbers such that 
$\overline B(0, r_1)\cap 
\overline B(a, r_2)=\emptyset$ 
and $\overline B(a, r_2) \subset B(0, 1+\delta)$.  
Define
\[
\Omega_1 := \overline{B}(0, 1+\delta) \setminus 
\left( B(0, r_1) \cup B(a, r_2) \right).
\]
On an open neighborhood of $\Omega_1$, both $u$ and $\log |z - a|$ are smooth. 
Thus, we have
\[
d\left( \frac{\partial u}{\partial \overline{z}} \log|z-a|^2\, d\overline{z}
+ \frac{1}{z-a} u\, dz \right)
= 2\log|z-a|\, \partial \overline{\partial} u
\]
on an open neighborhood of $\Omega_1$. Applying Stokes' theorem yields
\begin{equation}\label{q-eq10.6}
\begin{split}
2\int_{\Omega_1} \log|z-a|\, \partial \overline{\partial} u
&= \int_{\partial \Omega_1}
\left( \frac{\partial u}{\partial \overline{z}} \log|z-a|^2\, d\overline{z}
+ \frac{1}{z-a} u\, dz \right) \\
&= \int_{S(0,1+\delta)}
\left( 2 \frac{\partial u}{\partial \overline{z}} \log|z-a|\, d\overline{z}
+ \frac{1}{z-a} u\, dz \right) \\
&\quad - \int_{S(0,r_1)}
\left( 2 \frac{\partial u}{\partial \overline{z}} \log|z-a|\, d\overline{z}
+ \frac{1}{z-a} u\, dz \right) \\
&\quad - \int_{S(a,r_2)}
\left( 2 \frac{\partial u}{\partial \overline{z}} \log|z-a|\, d\overline{z}
+ \frac{1}{z-a} u\, dz \right).
\end{split}
\end{equation}

As $r_2 \to +0$, elementary computations give
\[
\int_{S(a,r_2)} \frac{\partial u}{\partial \overline{z}} \log|z-a|\, d\overline{z} \to 0, 
\quad
\int_{S(a,r_2)} \frac{1}{z-a} u\, dz \to 2\pi \sqrt{-1} u(a).
\]
Taking the limit as $r_2 \to +0$ in \eqref{q-eq10.6}, we obtain
\begin{equation*}
\begin{split}
2 \int_{\Omega_2} \log|z-a|\, \partial \overline{\partial} u
&= \int_{S(0,1+\delta)}
\left( 2 \frac{\partial u}{\partial \overline{z}} \log|z-a|\, d\overline{z}
+ \frac{1}{z-a} u\, dz \right) \\
&\quad - \int_{S(0,r_1)}
\left( 2 \frac{\partial u}{\partial \overline{z}} \log|z-a|\, d\overline{z}
+ \frac{1}{z-a} u\, dz \right) \\
&\quad - 2\pi \sqrt{-1} u(a),
\end{split}
\end{equation*}
where
\[
\Omega_2 := \overline{B}(0,1+\delta) \setminus B(0, r_1).
\]

Solving for \( u(a) \), we obtain
\begin{equation}\label{q-eq10.7}
\begin{split}
u(a) &= \frac{\sqrt{-1}}{\pi} \int_{\Omega_2} \log |z-a|\, \partial \overline{\partial} u \\
&\quad + \frac{1}{2\pi \sqrt{-1}} \int_{S(0,1+\delta)}
\left( 2 \frac{\partial u}{\partial \overline{z}} \log|z-a|\, d\overline{z}
+ \frac{1}{z-a} u\, dz \right) \\
&\quad - \frac{1}{2\pi \sqrt{-1}} \int_{S(0,r_1)}
\left( 2 \frac{\partial u}{\partial \overline{z}} \log|z-a|\, d\overline{z}
+ \frac{1}{z-a} u\, dz \right).
\end{split}
\end{equation}

We put $f(z):=\log |z-a|^2$. Then $\Delta f=0$ on a neighborhood 
of $\overline B(0, r_1)$. 
Thus, by Lemma \ref{q-lem10.2}, we have
\begin{equation}\label{q-eq10.8}
2 \int_{\overline{B}(0, r_1)} \log |z-a|\sigma (z)
= \int_{S(0, r_1)} \left(
2 \frac{\partial u}{\partial \overline{z}} \log|z-a|\, d\overline{z}
+ \frac{1}{z-a} u\, dz \right).
\end{equation}
Combining \eqref{q-eq10.7} and \eqref{q-eq10.8}, we obtain 
the desired identity \eqref{q-eq10.5} 
since $\overline B(0, 1+\delta)\setminus \Omega_2=B(0, r_1)$.

This completes the proof of Lemma~\ref{q-lem10.3}.
\end{proof}

\begin{say}[Setting]\label{q-say10.4} 
We now proceed to prove Lemma \ref{q-lem10.1}.  
First, we clarify the setting of the lemma.  
Define the function  
\[
f(z, e) := \log |e(z)|_h,
\]  
where \( z \in B(0,1+2\delta)^* \) 
and \( e \in S^{2k-1} := \{v \in \mathbb{C}^k \mid |v| = 1\} \subset \mathbb{C}^k \). 
Then $f$ is a smooth function on 
$B(0, 1+2\delta)^* \times S^{2k-1}$, and satisfies 
$f(z, e) \leq 0$, since the eigenvalues 
of $h$ are less than or equal to $1$.

By Lemma \ref{p-lem2.7}, and since the curvature 
of $h$ is negative, it follows that 
$f(z, e)$ is a smooth subharmonic function on $B(0, 1+2\delta)^*$ 
for every $e \in S^{2k-1}$.
Therefore, $f(z, e)$ extends to a locally integrable subharmonic 
function on $B(0,1+2\delta)$ for every $e \in S^{2k-1}$ 
(see \cite[(3.3.25) Theorem]{noguchi-ochiai}).

We define
\[
\mu(z, e) := \Delta f(z, e),
\]
where
\[
\Delta = 4\frac{\partial^2}{\partial z \partial \overline{z}}
\]
is the Laplacian in the sense of distributions, taken 
with respect to the variable $z$.

Then, for each fixed $e \in S^{2k-1}$, the function $f(z, e)$ being 
subharmonic implies that $\mu(z, e)$ defines a positive Radon 
measure on $B(0,1+2\delta)$ (see, for example, \cite[(3.1.14) Lemma]{noguchi-ochiai}).
\end{say}

\begin{lem}\label{q-lem10.5}
In the setting of \ref{q-say10.4},  
there exists a positive constant \( C \) such 
that the following inequality holds:
\[
\left| 
f(z,e) - \frac{1}{2\pi} \int_{B(0,1+\delta)} \log |w - z| \, \mu(w,e) \, d\lambda(w)
\right| \leq C
\]
for all \( (z,e) \in B(0,1)^* \times S^{2k-1} \),  
where \( d\lambda(w) := \frac{\sqrt{-1}}{2} \, dw \wedge d\overline{w} \).
\end{lem}

\begin{proof}[Proof of Lemma \ref{q-lem10.5}]
Consider the function
\[
(z,e) \mapsto
\left|
\frac{1}{\pi \sqrt{-1}} \int_{S(0,1+\delta)}
\frac{\partial f}{\partial \overline{w}}(w, e) \log|w - z| \, d\overline{w}
+
\frac{1}{2\pi \sqrt{-1}} \int_{S(0,1+\delta)} \frac{f(w,e)}{w - z} \, dw
\right|.
\]
This function is continuous on the compact set $\overline{B}(0,1) \times S^{2k-1}$.  
Therefore, there exists a constant $C > 0$ such that
\[
\left|
\frac{1}{\pi \sqrt{-1}} \int_{S(0,1+\delta)}
\frac{\partial f}{\partial \overline{w}}(w, e) \log|w - z| \, d\overline{w}
+
\frac{1}{2\pi \sqrt{-1}} \int_{S(0,1+\delta)} \frac{f(w,e)}{w - z} \, dw
\right| \le C
\]
for all $(z,e) \in \overline{B}(0,1) \times S^{2k-1}$.

On the other hand, by Lemma \ref{q-lem10.3} and the identity
\[
4 \partial \overline{\partial} f(w,e) = \mu(w,e)  dw \wedge d\overline{w},
\]
we have
\begin{align*}
f(z,e) 
&- \frac{1}{2\pi} \int_{B(0,1+\delta)} \log |w - z| \, \mu(w,e) \, d\lambda(w) \\
&=
\frac{1}{\pi \sqrt{-1}} \int_{S(0,1+\delta)}
\frac{\partial f}{\partial \overline{w}}(w,e) \log|w - z| \, d\overline{w}
+
\frac{1}{2\pi \sqrt{-1}} \int_{S(0,1+\delta)} \frac{f(w,e)}{w - z} \, dw.
\end{align*}

This proves the desired estimate.
\end{proof}

\begin{lem}
\label{q-lem10.6}
In the setting of \ref{q-say10.4}, there exists a positive constant $C$
such that the inequality
\[
\int_{B(0,1+\delta)}\mu(z,e)d\lambda(z) \le C
\]
holds for every $e \in S^{2k-1}$.
\end{lem}
\begin{proof}[Proof of Lemma \ref{q-lem10.6}]
Fix a smooth function $\varphi(z)$ on $B(0,1+2\delta)$
with the following properties:
\begin{itemize}
\item
$0 \le \varphi(z) \le 1$ for all $z \in B(0,1+2\delta)$,
\item
$\varphi(z)=1$ for all $z \in B(0,1+\delta)$, and
\item
the support of $\varphi$ is compact and contained in $B(0,1+2\delta)$.
\end{itemize}
Since
\[
e \mapsto
\int_{B(0,1+2\delta)}\Delta\varphi(z) f(z,e) d\lambda(z)
\]
is a smooth function on $S^{2k-1}$,
there exists a positive constant $C$ such that
\[
\int_{B(0,1+2\delta)}\Delta\varphi(z) f(z,e) d\lambda(z) \le C
\]
for all $e \in S^{2k-1}$.

On the other hand,
by the definition of the Laplacian in the sense of distributions, 
we have
\[
\int_{B(0,1+2\delta)}\varphi(z) \mu(z,e)d\lambda(z)
=\int_{B(0,1+2\delta)}\Delta\varphi(z) f(z,e) d\lambda(z).
\]
Therefore, 
\begin{align*}
\int_{B(0,1+\delta)}\mu(z,e)d\lambda(z)
&=\int_{B(0,1+\delta)}\varphi(z)\mu(z,e)d\lambda(z) \\
&\le
\int_{B(0,1+2\delta)}\varphi(z)\mu(z,e)d\lambda(z) \\
&=\int_{B(0,1+2\delta)}\Delta\varphi(z) f(z,e) d\lambda(z) \\
&\le C
\end{align*}
for all $e \in S^{2k-1}$,  as claimed.
\end{proof}

Now, we begin the proof of Lemma \ref{q-lem10.1}.

\begin{proof}[Proof of Lemma \ref{q-lem10.1}]
If $C \leq 1$, then the inequality $|\det h| \leq |z|$ clearly holds.  
In the case $C > 1$, replacing $h$ 
by $C^{-1/k} h$ allows us to assume $C = 1$ without loss of generality.

Let $C_1$ and $C_2$ be the positive 
constants obtained in Lemmas \ref{q-lem10.5} and \ref{q-lem10.6}, respectively.  
That is, $C_1$ satisfies
\begin{equation}\label{q-eq10.9}
\left| 
f(z,e) - \frac{1}{2\pi} \int_{B(0,1+\delta)} \log |w-z| \, \mu(w,e) \, d\lambda(w)
\right|
\le C_1
\end{equation}
for all $(z,e) \in B(0,1)^* \times S^{2k-1}$,  
and $C_2$ satisfies
\begin{equation}\label{q-eq10.10}
\int_{B(0,1+\delta)} \mu(z,e) \, d\lambda(z) \le C_2
\end{equation}
for all $e \in S^{2k-1}$.

By Lemma \ref{q-lem10.7} below, there exists $0 < r_0 < 1$ such that for every 
$0 < r < r_0$ 
there exist $e \in S^{2k-1}$ and an open subset $\Omega_r$ satisfying
\[
\Omega_r \subset \left\{ z \in \overline{B}(0, r) \,\middle|\, 
f(z, e) \leq \frac{1}{2k} \log r + \log 2 \right\}
\] 
and 
$\Area(\Omega_r) = r^3$, where $\Area(\Omega_r)$ is the 
area of $\Omega_r$. 
Then, from \eqref{q-eq10.9}, we obtain
\[
\frac{1}{2\pi} \int_{B(0,1+\delta)} \log |w-z| \, \mu(w,e) \, d\lambda(w) - C_1
\le f(z,e)
\le \frac{1}{2k} \log r + \log 2
\]
for all $z \in \Omega_r$. 
By \eqref{q-eq10.9}, 
\[
\frac{1}{2\pi} \int _{B(0, 1+\delta)} \log |w-z|\mu (w, e) d\lambda (w)
\] 
is integrable over $\Omega_r$. 
Applying the averaging operator
\[
\frac{1}{r^3} \int_{\Omega_r} \bullet \, d\lambda(z),
\]
we obtain
\begin{equation}\label{q-eq10.11}
\begin{split}
\frac{1}{2k} \log r + \log 2
&\ge \frac{1}{r^3} \int_{\Omega_r} f(z,e) \, d\lambda(z) \\
&\ge \frac{1}{r^3} \int_{\Omega_r} 
\left( 
\frac{1}{2\pi} \int_{B(0,1+\delta)} \log |w - z| \, \mu(w,e) \, d\lambda(w) - C_1
\right) d\lambda(z) \\
&= \frac{1}{2\pi} \int_{B(0,1+\delta)} 
\left( \frac{1}{r^3} \int_{\Omega_r} \log |w - z| \, d\lambda(z) \right)
\mu(w,e) \, d\lambda(w) - C_1 \\
&\ge \frac{1}{2\pi} \int_{B(0,1+\delta)} B_r(w) \, \mu(w,e) \, d\lambda(w) - C_1,
\end{split}
\end{equation}
by the definition of $B_r(w)$ in Definition \ref{p-def9.1}.

On the other hand, Lemma \ref{p-lem9.4} implies 
that for all $z \in B(0,1) \setminus B(0,r)$,
\begin{equation}\label{q-eq10.12}
\begin{split}
\int_{B(0,1+\delta)}
&\log |w-z|\mu(w,e) d\lambda(w) \\
&\le
\int_{B(0,1+\delta)}
\Bigl(
\frac{2}{3}\cdot\frac{\log |z|}{\log r}\cdot 
B_r(w)+\frac{1}{3}\log \pi+\frac{1}{3}+2\log 2
\Bigr)
\mu(w,e) d\lambda(w) \\
&=
\frac{2\log|z|}{3\log r}
\int_{B(0,1+\delta)}
B_r(w)\mu(w,e) d\lambda(w) \\
&\qquad \qquad
+
\Bigl(\frac{1}{3}\log \pi+\frac{1}{3}+2\log 2\Bigr)
\int_{B(0,1+\delta)}\mu(w,e) d\lambda(w)
\end{split}
\end{equation}
for all $z \in B(0,1) \setminus B(0,r)$.

Combining these inequalities \eqref{q-eq10.11} and 
\eqref{q-eq10.12} with \eqref{q-eq10.10}, we obtain
\begin{align*}
\int_{B(0,1+\delta)}
&\log |w-z|\mu(w,e) d\lambda(w) \\
&\le
\frac{4\pi\log|z|}{3\log r}
\Bigl(\frac{1}{2k}\log r+\log 2+C_1\Bigr)
+
\Bigl(\frac{1}{3}\log \pi+\frac{1}{3}+2\log 2\Bigr)
C_2 \\
&\le
\frac{2\pi}{3k}\log |z|
+\frac{4\pi}{3}
\Bigl(\log 2+C_1\Bigr)
+
\Bigl(\frac{1}{3}\log \pi+\frac{1}{3}+2\log 2\Bigr)
C_2
\end{align*}

Using \eqref{q-eq10.9} again, we deduce
\[
f(z,e) \le \frac{1}{3k} \log |z|
+ \frac{2}{3} \log 2 + \frac{5}{3} C_1
+ \frac{1}{2\pi}
\left( \frac{1}{3} \log \pi + \frac{1}{3} + 2 \log 2 \right) C_2
\]
for all $z \in B(0,1) \setminus B(0,r)$. Setting
\[
C' := 
\exp \left(
\frac{2}{3} \log 2 + \frac{5}{3} C_1
+ \frac{1}{2\pi} 
\left( \frac{1}{3} \log \pi + \frac{1}{3} + 2 \log 2 \right) C_2
\right),
\]
we obtain the inequality
\[
f(z,e) \le \frac{1}{3k} \log |z| + \log C',
\] 
or equivalently,  
\[ 
|e(z)|_h \le C' |z|^{\frac{1}{3k}}
\]
for all $z \in B(0,1) \setminus B(0,r)$.

Let $r_i \to 0$ be any decreasing sequence with $0 < r_i < r_0$.  
Then by the above argument, we can find $e_i \in S^{2k-1}$ such that
\begin{equation}\label{q-eq10.13}
|e_i(z)|_h \le C' |z|^{\frac{1}{3k}}
\end{equation}
for all $z \in B(0,1) \setminus B(0,r_i)$.
Since $S^{2k-1}$ is compact, we may, after passing to a subsequence, assume that
\[
\lim_{i \to \infty} e_i = e \in S^{2k-1}.
\]
Then from \eqref{q-eq10.13}, it follows that
\[
|e(z)|_h \le C' |z|^{\frac{1}{3k}}
\]
holds for all $z \in B(0,1)^*$.  
This completes the proof of Lemma \ref{q-lem10.1}.
\end{proof}

The following lemma is used in the proof of Lemma \ref{q-lem10.1} above.

\begin{lem}\label{q-lem10.7} 
In the setting of Lemma \ref{q-lem10.1}, assume that 
$|\det h| \leq |z|$. 
Then there exists a constant $0 < r_0 < 1$ such that for every $0 < r < r_0$, there exists a vector $e \in S^{2k-1}$ for which 
\[
\left\{ z \in \overline{B}(0, r) \,\middle|\, f(z, e) < \frac{1}{2k} \log (2^{2k} r) \right\}
\]
contains an open subset $\Omega_r$ with $\Area(\Omega_r) = r^3$.
\end{lem}

\begin{proof}[Proof of Lemma \ref{q-lem10.7}]
By assumption, namely $|\det h| \leq |z|$ on $\overline{B}(0, r)$, we can choose a vector $e^\dag \in S^{2k-1}$ such that
\[
f(z, e^\dag) = \log |e^\dag(z)|_h \leq \frac{1}{k} \log r < \frac{1}{2k} \log r < 0.
\]
Let $v$ be any vector with Euclidean norm $|v|_{\mathrm{Euclid}} \leq r^{1/(2k)}$. Then
\[
|e^\dag + v|_h \leq |e^\dag|_h + |v|_{\mathrm{Euclid}} < 2 r^{1/(2k)},
\]
since all eigenvalues of $h$ are less than or equal to $1$.

Without loss of generality, we may assume that $0 < r < r'_0$ 
for some sufficiently small constant $r'_0$.

Then, for each $z \in \overline{B}(0, r)$, the volume of the set of $e \in S^{2k-1}$ for which the above bound holds is at least
\[
\alpha \left(r^{1/(2k)}\right)^{2k-1} = \alpha r^{1 - (1/(2k))},
\]
for some positive constant $\alpha$.

This implies that the volume of the subset of $S^{2k-1} \times \overline{B}(0, r)$ where the bound holds is at least
\[
\alpha r^{1 - (1/(2k))} \cdot \pi r^2 = \alpha \pi r^{-1/(2k)} r^3.
\]

Suppose, for contradiction, that for every $e \in S^{2k-1}$, 
the area of the region in $\overline{B}(0, r)$ where the 
bound holds is less than $r^3$. Then the total volume 
in $S^{2k-1} \times \overline{B}(0, r)$ would be less than
\[
\sigma_{2k-1} \cdot r^3,
\]
where $\sigma_{2k-1}$ denotes the volume of the unit sphere $S^{2k-1}$ in $\mathbb{C}^k$.

However, if $r$ is sufficiently small, we have
\[
\alpha \pi r^{-1/(2k)} r^3 > \sigma_{2k-1} \cdot r^3,
\]
which is a contradiction.

Therefore, there exists a sufficiently small constant $0 < r_0 < 1$ 
such that for every $0 < r < r_0$, there exists at 
least one $e' \in S^{2k-1}$ such that the set
\[
\left\{ z \in \overline{B}(0, r) \,\middle|\, f(z, e') < \frac{1}{2k} \log (2^{2k} r) \right\}
\]
contains an open subset $\Omega_r$ with $\Area(\Omega_r) = r^3$.
\end{proof}

\section{On cyclic covers}\label{p-sec11}

In what follows, we briefly discuss cyclic 
covers, which will be used in later arguments. 
Let us recall the following elementary fact for the reader's convenience. 

\begin{lem}\label{p-lem11.1} 
Let $m$ be any positive integer with $m \geq 2$, and let 
$\epsilon$ be a complex number such that $\epsilon^m = 1$ and $\epsilon \ne 1$. 
Then 
\[
\sum_{i=0}^{m-1} \epsilon^i = 0.
\]
\end{lem}

\begin{proof}[Proof of Lemma \ref{p-lem11.1}]
Since 
\[
1 - \epsilon^m = (1 - \epsilon)(1 + \epsilon + \cdots + \epsilon^{m-1}) = 0,
\]
and $1 - \epsilon \ne 0$, it follows that 
\[
\sum_{i=0}^{m-1} \epsilon^i = 0.
\]
\end{proof}

The following lemma is the main result of this section.

\begin{lem}[{cf.~\cite[Lemma 10.3]{simpson1}}]\label{p-lem11.2} 
Let $(E, h)$ be an acceptable vector bundle on $X = \Delta^*$ with $\rank E = r$. 
Let $\pi \colon W := \Delta^* \to X$ be 
the $m$-fold cyclic cover of $\Delta^*$ 
given by $\pi(w) = z^m$, where $z$ is the coordinate on $X$ and $w$ is the coordinate on $W$. 
Let 
\[
\{v_1, \ldots, v_r\}
\] 
be a frame of ${}^\diamond\! E = {}_0E$ 
compatible with the parabolic filtration, 
such that $v_i \in {}_{b_i}E \setminus {}_{< b_i}E$ for each $i$. 
Let $\alpha$ be any real number, and let $m_{\alpha,i}$ be the smallest integer such that 
\[
m b_i - m_{\alpha,i} \leq \alpha
\] 
for each $i$. 
Then 
\[
\{w^{m_{\alpha,1}} \pi^*v_1, \ldots, w^{m_{\alpha,r}} \pi^*v_r\}
\] 
is a frame of ${}_\alpha (\pi^*E)$ compatible with the parabolic filtration. 
\end{lem}

\begin{proof}[Proof of Lemma \ref{p-lem11.2}] 
By direct calculation, $\pi^* \omega_P$ is the Poincar\'e metric on $W$.  
Therefore, it is straightforward to verify that $\pi^* E$ is an acceptable vector bundle on $W$.  
By definition, we can readily see that $w^{m_{\alpha,i}} \pi^* v_i$ 
is a section of ${}_\alpha(\pi^* E)$ for each $i$.

\setcounter{step}{0}
\begin{step}\label{p-step11.2-1}
Let $G = \mathbb{Z}/m\mathbb{Z} = \langle g \rangle$ be the Galois group of $\pi \colon W \to X$. 
Then $G$ acts naturally on $\pi^*E$, and this action preserves the metric. 
Let $U$ be any open subset of $X$. 
Then we have 
\[
H^0(\pi^{-1}(U), \pi^*E) = H^0(U, \pi_* \pi^*E).
\] 
We also have the decomposition
\begin{equation}\label{p-eq11.1}
\pi_* \pi^* E = \bigoplus_{j=0}^{m-1} w^j E,
\end{equation} 
i.e., the $\mathcal{O}_W$-module 
$\pi^*E$ decomposes into a direct sum 
of $\mathcal{O}_X$-modules as in \eqref{p-eq11.1}, under the action of $G$. 
The action of $G$ on the right-hand side 
is given by $g^* w = \zeta w$, where $\zeta$ is an $m$th root of unity. 
\end{step}

\begin{step}\label{p-step11.2-2}
In this step, we prove that 
\[
\{w^{m_{\alpha,1}} \pi^*v_1, \ldots, w^{m_{\alpha,r}} \pi^*v_r\}
\] 
is a frame of ${}_\alpha(\pi^*E)$. 

Take any $u \in H^0(\pi^{-1}(U), \pi^*E)$. 
By \eqref{p-eq11.1}, we can write 
\[
u = \sum_{j=0}^{m-1} w^j u_j,
\] 
where $u_j \in H^0(U, E)$ for each $j$. 
Assume that 
\[
|u|_{\pi^*h} \leq \frac{C}{|w|^\lambda}
\] 
holds for some $C > 0$ and $\lambda \in \mathbb{R}$. 
Then, by considering 
\[
\sum_{l=0}^{m-1} \zeta^{-lj} (g^l)^* u
\] 
for each $j$, we obtain the same estimate:
\[
|w^j u_j|_{\pi^*h} \leq \frac{C}{|w|^\lambda}
\] 
for every $j$. 
Here we used Lemma \ref{p-lem11.1} 
and the fact that the $G$-action preserves the metric. 

This implies that ${}_\alpha(\pi^*E)$ is generated by 
\[
\{w^{m_{\alpha,1}} \pi^*v_1, \ldots, w^{m_{\alpha,r}} \pi^*v_r\}
\] 
for every $\alpha$. 
Hence this set forms a frame of ${}_\alpha(\pi^*E)$, as desired.
\end{step}

\begin{step}\label{p-step11.2-3}
In this final step, we verify that the frame 
\[
\{w^{m_{\alpha,1}} \pi^*v_1, \ldots, w^{m_{\alpha,r}} \pi^*v_r\}
\] 
is compatible with the parabolic filtration. 

Assume that 
\[
\beta := m b_1 - m_{\alpha,1} = \cdots = m b_l - m_{\alpha,l}
\] 
for some $l \geq 1$. Under this assumption, it suffices to show that 
\[
\{w^{m_{\alpha,1}} \pi^*v_1, \ldots, w^{m_{\alpha,l}} \pi^*v_l\}
\] 
is linearly independent in the quotient 
space ${}_\beta (\pi^*E) / {}_{<\beta}(\pi^*E)$. 

Suppose that 
\begin{equation}\label{p-eq11.2}
a_1 w^{m_{\alpha,1}} \pi^*v_1 + \cdots + 
a_l w^{m_{\alpha,l}} \pi^*v_l \in {}_{<\beta}(\pi^*E),
\end{equation} 
for some $a_i \in \mathbb{C}$. 

Note that if $m b_i - m_{\alpha,i} = m b_j 
- m_{\alpha,j}$ and $b_i \ne b_j$, then $m_{\alpha,i} \ne 
m_{\alpha,j}$ and $|m_{\alpha,i} - m_{\alpha,j}| < m$. 
Therefore, using the 
decomposition \eqref{p-eq11.1} as 
in Step \ref{p-step11.2-2}, we may assume that $b_1 = \cdots = b_l$. 
In this case, \eqref{p-eq11.2} implies 
\[
a_1 v_1 + \cdots + a_l v_l \in {}_{< b_1} E.
\] 
Since $\{v_1, \ldots, v_r\}$ is 
compatible with the parabolic filtration, this implies that $a_1 = \cdots = a_l = 0$. 

Thus, 
\[
\{w^{m_{\alpha,1}} \pi^*v_1, 
\ldots, w^{m_{\alpha,r}} \pi^*v_r\}
\] 
is compatible with the parabolic filtration, as claimed.
\end{step}
This completes the proof of Lemma \ref{p-lem11.2}. 
\end{proof}

The converse of the above lemma also holds, as shown below.

\begin{lem}\label{p-lem11.3}
Let $(E, h)$ be an acceptable vector 
bundle on $X = \Delta^*$ with $\rank E = r$. 
Let $\pi \colon W := \Delta^* \to X$ be 
the $m$-fold cyclic cover of $\Delta^*$ 
given by $\pi(w) = z^m$, where $z$ is the coordinate 
on $X$ and $w$ is the coordinate on $W$. 
Let 
\[
\{v_1, \ldots, v_r\}
\] 
be a frame of ${}^\diamond\! E = {}_0E$ 
such that $v_i \in {}_{b_i}E \setminus {}_{< b_i}E$ for each $i$. 
Let $\alpha$ be any real number, 
and let $m_{\alpha, i}$ be the smallest integer such that 
\[
m b_i - m_{\alpha,i} \leq \alpha
\] 
for each $i$. 
If  
\[
\{w^{m_{\alpha,1}} \pi^*v_1, \ldots, w^{m_{\alpha,r}} \pi^*v_r\}
\] 
is a frame of ${}_\alpha(\pi^*E)$ 
compatible with the parabolic filtration for some $\alpha$, then 
\[
\{v_1, \ldots, v_r\}
\] 
is a frame of ${}^\diamond\!E = {}_0E$ 
compatible with the parabolic filtration. 
\end{lem}

\begin{proof}[Proof of Lemma \ref{p-lem11.3}]
Assume that $\beta := b_1 = \cdots = b_l$ for some $l \geq 1$. 
Under this assumption, it suffices to show 
that $\{v_1, \ldots, v_l\}$ is linearly 
independent in the quotient space ${}_\beta E / {}_{<\beta}E$. 

Suppose that 
\[
a_1 v_1 + \cdots + a_l v_l \in {}_{<\beta} E
\]
for some $a_1, \ldots, a_l \in \mathbb{C}$. 
Let $n := m_{\alpha,1} = \cdots = 
m_{\alpha,l}$. 
Then we have
\[
a_1 w^n \pi^*v_1 + \cdots + a_l w^n \pi^*v_l 
\in {}_{< m\beta - n} (\pi^*E).
\] 
But by assumption, the set 
\[
\{w^n \pi^*v_1, \ldots, w^n \pi^*v_l\}
\] 
is part of a frame of ${}_\alpha(\pi^*E)$ that is 
compatible with the parabolic filtration. 
This implies that the above linear combination 
lies in a lower filtration step only if all coefficients vanish, i.e., 
\[
a_1 = \cdots = a_l = 0.
\] 
Therefore, $\{v_1, \ldots, v_r\}$ is compatible with 
the parabolic filtration, as claimed.
\end{proof}

\section{On determinant bundles}\label{p-sec12}

The main purpose of this section is to establish 
Theorem \ref{p-thm12.3} and Corollary \ref{p-cor12.4}, 
which will play crucial roles in the subsequent sections.  
We begin with an elementary lemma from Diophantine approximation.

\begin{defn}\label{p-def12.1}
Let $\alpha = (\alpha_1, \dots, \alpha_l) \in \mathbb{R}^l$ be a vector.  
We define
\[
R_m(\alpha) := \left( m\alpha_1 - n_1, \dots, m\alpha_l - n_l \right),
\]
where each $n_i$ is the integer that minimizes $|m\alpha_i - n_i|$.  
We also define
\[
\delta_m(\alpha) := \max_i |m\alpha_i - n_i|.
\]
\end{defn}

\begin{lem}\label{p-lem12.2}
Let $\alpha = (\alpha_1, \dots, \alpha_l) 
\in \mathbb{R}^l$ be a vector such that $\alpha_i \notin \mathbb{Q}$ for every $i$.  
Then, for any real number $q > 1$, 
there exists a positive integer $m \leq q$ such that
\[
\delta_m(\alpha) \leq q^{-1/l}.
\]
\end{lem}

\begin{proof}[Proof of Lemma \ref{p-lem12.2}]
This follows easily from \cite[Chapter I, Theorem VI]{cas}.  
It is essentially a consequence of Minkowski's theorem.  
We omit the details.
\end{proof}

The following theorem is the main result of this section.

\begin{thm}\label{p-thm12.3}
Let $\{v_1, \dots, v_r\}$ be a frame of ${}_aE$ 
around the origin, compatible with the parabolic 
filtration, such that $v_i \in {}_{b_i}E \setminus {}_{< b_i}E$ 
for every $i$.  
Then we have the following equality:
\[
\gamma({}_a E) = \sum_{i=1}^r b_i.
\]
\end{thm}

By combining Theorem \ref{p-thm7.5} with Theorem \ref{p-thm12.3}, 
we obtain the following important result on 
determinant bundles, which will also play a crucial role in the subsequent sections.

\begin{cor}\label{p-cor12.4}
We use the same notation as in Theorem~\ref{p-thm12.3}.  
Then we have
\[
\det({}_aE) = {}_{\sum_{i=1}^r b_i} \left( \det E \right).
\]
\end{cor}

The following two remarks are 
straightforward, but we include them for completeness.

\begin{rem}\label{p-rem12.5}
Let $(E, h)$ be an acceptable vector bundle 
on $\Delta^*$. 
Let $\{v_1\ldots, v_r\}$ be a frame of ${}_a E$. 
We consider 
\[
(E^\dag, h^\dag):=(E, h\cdot |z|^{2c}), 
\] 
where $c$ is a real number. 
Then $(E^\dag, h^\dag)$ is also an acceptable vector bundle 
on $\Delta^*$ since $\partial \overline \partial \log |z|^{2c}=0$ on 
$\Delta^*$. 
It is easy to see that 
\[ 
{}_{a-c} E^\dag ={}_a E
\] 
holds and that $\{v_1, \ldots, v_r\}$ is a frame of ${}_{a-c} E^\dag$. 
By definition, we have 
\[
\gamma ({}_{a-c} E^\dag)=\gamma ({}_a E)-rc. 
\] 

We further assume that 
$\{v_1, \ldots, v_r\}$ is compatible with the parabolic filtration 
such that $v_i\in {}_{b_i} E\setminus {}_{<b_i}E$ for every $i$. 
Then it is obvious that $\{v_1, \ldots, v_r\}$ is a frame 
of ${}_{a-c} E^\dag$ compatible with the parabolic filtration such that 
$v_i \in {}_{b_i-c} E^\dag \setminus {}_{<b_i-c} E^\dag$ for every $i$. 
We note that 
\[
\gamma ({}_a E)-\sum _{i=1}^r b_i =\gamma ({}_{a-c} E^\dag) -\sum 
_{i=1}^r (b_i-c) 
\] 
holds. 
Hence, in the proof of Theorem \ref{p-thm12.3}, we can freely replace 
$h$ with $h\cdot |z|^{2c}$ for any real number $c$. 
\end{rem}

\begin{rem}\label{p-rem12.6}
Let $(E, h)$ be an acceptable vector bundle on $\Delta^*$, and let 
$\{v_1, \ldots, v_r\}$ be a frame of ${}_a E$. 
Consider the pair
\[
(\widetilde E, \widetilde h) := \left(E, h e^{-\chi(-N)}\right),
\]
where $N$ is a real number. 
It is straightforward to verify that 
$(\widetilde E, \widetilde h)$ is also an acceptable 
vector bundle on $\Delta^*$, and that 
${}_\alpha \widetilde E = {}_\alpha E$ for every 
$\alpha \in \mathbb{R}$. 
Note that $\{v_1, \ldots, v_r\}$ is a frame of 
${}_a \widetilde E$, and that 
$v_i \in {}_{b_i}\widetilde E \setminus 
{}_{<b_i}\widetilde E$ if and only if 
$v_i \in {}_{b_i}E \setminus 
{}_{<b_i}E$. 
By definition, we have 
$\gamma ({}_a \widetilde E) = \gamma ({}_a E)$. 
It is also clear that 
$\{v_1, \ldots, v_r\}$ is a frame of ${}_a \widetilde E$ compatible 
with the parabolic filtration if and only if it is so for 
${}_a E$. 
Therefore, in the proof of Theorem \ref{p-thm12.3}, we can freely replace 
$h$ with 
$he^{-\chi(-N)}$ for any real number $N$. 
\end{rem}

Let us now prove Theorem~\ref{p-thm12.3}.

\begin{proof}[Proof of Theorem~\ref{p-thm12.3}] 
Although the inequality
\[
\gamma({}_a E) \leq \sum_{i=1}^r b_i
\]
was already established in Corollary~\ref{p-cor7.7}, 
we provide an alternative proof in Step~\ref{p-step12.3-1}.  
It should be noted that the assumption that the frame 
$\{v_1, \ldots, v_r\}$ is compatible with the parabolic 
filtration is not required in this step.  
The discussion in Step~\ref{p-step12.3-1} will be needed 
in Step~\ref{p-step12.3-2}.  
In Step~\ref{p-step12.3-2}, we will establish the 
reverse inequality, where Lemma~\ref{q-lem10.1} will play a crucial role.

\setcounter{step}{0}
\begin{step}\label{p-step12.3-1}
In this step, we prove that
\[
\gamma({}_a E) \leq \sum_{i=1}^r b_i.
\]
As noted above, the assumption 
that $\{v_1, \ldots, v_r\}$ is compatible 
with the parabolic filtration is not needed here.

By Lemma~\ref{p-lem5.4}, we can write
\[
|v_i|_h = \frac{v^\dag_i(z)}{|z|^{b_i}} \left(-\log |z|\right)^{M_i}
\]
around the origin, where each $M_i \in \mathbb{Z}_{>0}$ and $v^\dag_i(z)$ is bounded for all $i$.

Since $v_i \in {}_{b_i}E \setminus {}_{< b_i} E$, we know that
\[
\frac{v^\dag_i(z)}{|z|^c}
\]
is unbounded for any $c > 0$.

Let $\bm{v} := \{v_1, \ldots, v_r\}$, and consider $\det H(\bm{v})(z)$ as in \ref{p-say7.1}.  
Then we can write
\[
\det H(\bm{v})(z) = |z|^{-2\sum b_i} \left(-\log |z|\right)^{2r \sum M_i} u(z),
\]
where $u(z)$ is bounded.

Therefore, by Corollary~\ref{p-cor7.6}, we have
\[
\begin{split}
\gamma({}_a E) &= -\frac{1}{2} \liminf_{z \to 0} \frac{\log \det H(\bm{v})(z)}{\log |z|} \\
&= -\frac{1}{2} \lim_{z \to 0} \frac{\log \det H(\bm{v})(z)}{\log |z|} \\
&= \sum_{i=1}^r b_i - \frac{1}{2} \lim_{z \to 0} \frac{\log u(z)}{\log |z|} \\
&\leq \sum_{i=1}^r b_i,
\end{split}
\]
since
\[
\lim_{z \to 0} \frac{\log u(z)}{\log |z|} \geq 0.
\]
This completes the proof of the inequality.
\end{step}

\begin{step}\label{p-step12.3-2}
In this step, we prove the reverse inequality:
\[
\gamma({}_a E) \geq \sum_{i=1}^r b_i.
\]
We emphasize that the 
assumption that $\{v_1, \ldots, v_r\}$ is 
compatible with the parabolic filtration is essential in 
this step.

Suppose, to the contrary, that
\[
\gamma({}_a E) < \sum_{i=1}^r b_i.
\]
This implies that
\begin{equation}\label{p-eq12.1}
\lim_{z \to 0} \frac{\log u(z)}{\log |z|} > 0
\end{equation}
in Step~\ref{p-step12.3-1}.

By replacing $h$ with $h \cdot |z|^{2 \max_i \{b_i\} 
+ \varepsilon}$ for some small $\varepsilon > 0$, 
we may assume that $a = 0$ and $b_i \in (-1, 0)$ 
for all $i$ (see Remark~\ref{p-rem12.5}).  
Next, by replacing $h$ with $h \cdot e^{-\chi(-N)}$ for some 
sufficiently large $N \gg 0$, we may further assume that the curvature of $h$ is 
negative (see Remark~\ref{p-rem12.6}).

Note that all entries of the matrix $h$ are bounded, since each $b_i \in (-1, 0)$. 
Rescaling the coordinate via 
\( z \mapsto z/C \) for some 
constant \( C > 0 \) does not affect the 
values of \( \gamma({}_a E) \) and \( b_i \) (see Remark \ref{p-rem7.18}). 
Therefore, by choosing an appropriate rescaling, 
we may assume that the frame \( \{v_1, \ldots, v_r\} \) 
is defined on the unit disk \( B(0, 1) \). 
Applying the rescaling \( z \mapsto z/2 \) once more, 
we may further assume that the pair \( (E, h) \) is 
defined and trivialized on \( B(0, 2) \). 
Then, by further replacing $h$ with $(1/C)h$ for some 
sufficiently large constant $C \gg 0$, we may assume 
that all eigenvalues of $h$ are $\leq 1$ on some open 
neighborhood of the closed disk 
$\overline{B}(0, 1 + \delta)$ for sufficiently small $\delta > 0$.

By Lemma \ref{p-lem12.2}, we can choose a 
sufficiently large positive integer $m$ such that 
$m b_i \in \mathbb{Z}$ for all 
rational $b_i$, and
\[
  \delta_m(b) < \frac{1}{6 \, \operatorname{rank} E}.
  \]
Note that $m$ can be taken arbitrarily large.

Set $b'_i := m b_i - n_i$ for each $i$, 
where $n_i$ is the integer minimizing $|m b_i - n_i|$ (see Definition~\ref{p-def12.1}).

Now consider the $m$-fold cyclic cover
\[
\pi \colon B(0, 1 + \delta') \to B(0, 1 + \delta), \quad \pi(w) = z^m
\]
for some $\delta' > 0$.  
Define
\[
(E', h') := \left( \pi^* E, \, \pi^* h \cdot |w|^{2 \max_i \{b'_i\}} \right).
\]

Then, since $\gamma({}_a E) < \sum_{i=1}^r b_i$ (i.e., \eqref{p-eq12.1}) 
and $m$ is sufficiently large, there exists a constant $C' > 0$ such that
\[
|\det h'| \leq C' |w|
\]
on some open neighborhood of the closed disk 
$\overline{B}(0, 1 + \delta'')$ for sufficiently small $\delta'' > 0$.

By Lemma~\ref{p-lem11.2}, the set
\[
\{ w^{n_1} \pi^* v_1, \ldots, w^{n_r} \pi^* v_r \}
\]
forms a frame of ${}_0 E'$ that is compatible with the parabolic filtration.

Applying Lemma~\ref{q-lem10.1}, we obtain a 
constant section $e'$ of $E'$ 
and a constant $C'' > 0$ such that
\[
|e'(w)|_{h'} \leq C'' |w|^{\frac{1}{3 \, \operatorname{rank} E}}
\]
holds near the origin.

By construction, we have
\[
\Par({}_0 E') \subset \left( -\frac{1}{3 \, \operatorname{rank} E}, 0 \right],
\]
which contradicts the existence of such a section $e'$ with the above estimate.

Therefore, our assumption must be false, and we conclude that
\[
\gamma({}_a E) \geq \sum_{i=1}^r b_i.
\]
\end{step}

By Steps~\ref{p-step12.3-1} and~\ref{p-step12.3-2}, we obtain the desired equality:
\[
\gamma({}_a E) = \sum_{i=1}^r b_i.
\]
This completes the proof of Theorem~\ref{p-thm12.3}.
\end{proof}

\begin{prop}\label{p-prop12.7} 
Let $\{w_1, \ldots, w_r\}$ be a local frame of ${}_aE$ defined over some 
open 
neighborhood of $0$, such that 
$w_i \in {}_{c_i}E$ for every $i$. Assume that 
\[
\sum_{i=1}^r c_i \leq \gamma({}_a E).
\] 
Then the following assertions hold:
\begin{itemize}
\item[(i)] $w_i \in {}_{c_i}E \setminus {}_{<c_i}E$ for every $i$;
\item[(ii)] $\gamma ({}_a E)=\sum _{i=1}^r c_i$; 
\item[(iii)] $c_i \in (a-1, a]$ for every $i$;
\item[(iv)] $\{w_1, \ldots, w_r\}$ is 
a local frame of ${}_aE$ compatible with the parabolic filtration.
\end{itemize}
\end{prop}

\begin{proof}[Proof of Proposition \ref{p-prop12.7}]
For each $i$, take a real number $c_i'$ such that 
$w_i \in {}_{c_i'}E \setminus {}_{<c_i'}E$. 
By definition, we have $c_i' \leq c_i$ for all $i$. 
By Corollary \ref{p-cor7.7}, it follows that
\[
\gamma({}_a E) \leq \sum_{i=1}^r c_i'.
\]
Therefore,
\[
\gamma({}_a E) \leq \sum_{i=1}^r c_i' \leq \sum_{i=1}^r c_i \leq \gamma({}_a E).
\]
Thus, all inequalities must be equalities, and 
we conclude that $c_i' = c_i$ for all $i$. 
In particular, $w_i \in {}_{c_i}E \setminus {}_{<c_i}E$ for all $i$ 
and $\gamma({}_a E)=\sum _{i=1}^r c_i$, which proves (i) and 
(ii).

Next, since $\{w_1, \ldots, w_r\}$ forms 
a local frame of ${}_a E$ near $0$, it follows from (i) and 
Lemma~\ref{p-lem7.16} that $c_i \in (a - 1, a]$ for all $i$. 
Hence, (iii) follows.

Finally, consider the quotient vector space ${}_aE / {}_{a-1}E$. 
Suppose that $\{w_1, \ldots, w_r\}$ is not compatible with the parabolic filtration. 
Then there exists $(a_{ij}) \in \mathrm{GL}(r, \mathbb{C})$ such that 
the new frame $\{w'_1, \ldots, w'_r\}$, defined by
\[
w'_j = \sum_{i=1}^r w_i a_{ij},
\]
satisfies $w'_j \in {}_{c_j'}E \setminus {}_{<c_j'}E$ for each $j$, and
\[
\sum_{j=1}^r c_j' < \sum_{i=1}^r c_i.
\]
However, by Corollary \ref{p-cor7.7} again, we have 
\[
\gamma({}_a E) \leq \sum_{j=1}^r c_j',
\]
which contradicts the assumption $\sum_{i=1}^r c_i \leq \gamma({}_a E)$. 
Hence, $\{w_1, \ldots, w_r\}$ must be compatible with the parabolic filtration. 
This proves (iv), and completes the proof of Proposition \ref{p-prop12.7}.
\end{proof}

By Corollary~\ref{p-cor7.7}, Theorem~\ref{p-thm12.3}, and  
Proposition~\ref{p-prop12.7}, we obtain the following useful statement,  
which will be used in subsequent sections.

\begin{cor}\label{p-cor12.8} 
Let $\{u_1, \ldots, u_r\}$ be a local frame of ${}_a E$ 
defined over some open neighborhood of $0$, such that  
$u_i \in {}_{d_i} E$ for every $i$.  
Then the inequality  
\[
\gamma({}_a E) \leq \sum_{i=1}^r d_i
\]
holds.  
Equality holds if and only if the following three conditions are satisfied:
\begin{itemize}
\item[(i)] $\{u_1, \ldots, u_r\}$ is a local 
frame of ${}_a E$ compatible with the parabolic filtration;
\item[(ii)] $u_i \in {}_{d_i} E \setminus {}_{<d_i} E$ for every $i$;
\item[(iii)] $\Par({}_a E) = \{d_1, \ldots, d_r\}$.
\end{itemize}
\end{cor}

We will use this corollary when 
showing that a given frame is compatible with the parabolic filtration.

\section{On dual bundles}\label{p-sec13}

In this section, we investigate the prolongation of dual vector bundles. 
We begin by reformulating Corollary \ref{p-cor1.7} (see also Corollary \ref{p-cor4.3}).

\begin{lem}[Duality for line bundles, 
see Lemma \ref{p-lem1.11}]\label{p-lem13.1}
Let $(L, h)$ be an acceptable line bundle on $\Delta^*$.  
Let $\alpha \in \mathbb{R}$ be any real number.  
Then we have ${}_\alpha L = {}_{\gamma({}_\alpha L)} L$, 
and $\Par_\alpha (L, h) = \{ \gamma({}_\alpha L) \}$.   

Moreover, if $0 < \varepsilon \ll 1$, then  
\[
\gamma({}_{-\alpha + 1 - \varepsilon}(L^\vee)) = -\gamma({}_\alpha L).
\]  

In particular, the following equality holds:
\[
\left({}_{\gamma({}_\alpha L)} L\right)^\vee = {}_{-\gamma({}_\alpha L)}(L^\vee).
\]
\end{lem}

\begin{proof}[Proof of Lemma~\ref{p-lem13.1}]
We use the same notation as in the 
proof of Proposition \ref{p-prop4.1} (see also Theorem \ref{p-thm4.4}).  

Since ${}_\alpha L = \mathcal{O}_\Delta \cdot z^{-\lfloor \alpha - \gamma \rfloor}$, 
we may take the following trivialization: 
\[
({}_\alpha L, h) \simeq 
\left( \mathcal{O}_\Delta, 
|\cdot|^2 e^{-2\varphi_\alpha} \right),
\] 
where $\varphi_\alpha := \varphi + \lfloor \alpha - \gamma \rfloor \log |z|$. 
Note that 
\[
\lim_{z \to 0} \frac{\varphi(z)}{\log |z|} = \gamma. 
\]
Therefore, we obtain
\[
\gamma({}_\alpha L) 
= \lim_{z \to 0} \frac{\varphi_\alpha(z)}{\log |z|} 
= \lfloor \alpha - \gamma \rfloor + \gamma.
\] 
This immediately implies that ${}_\alpha L = {}_{\gamma({}_\alpha L)} L$, 
and that $\Par_\alpha (L, h) = \{ \gamma({}_\alpha L) \}$.

Similarly, we compute
\[
\gamma \left({}_{-\alpha + 1 - \varepsilon}(L^\vee)\right) 
= \lfloor -\alpha + 1 - \varepsilon + \gamma \rfloor - \gamma.
\] 
For details, see the proof of Corollary \ref{p-cor4.3}.   

If $0 < \varepsilon \ll 1$, then
\begin{equation*}
\begin{split}
\gamma({}_{-\alpha + 1 - \varepsilon}(L^\vee)) 
&= \lfloor -\alpha + 1 - \varepsilon + \gamma \rfloor - \gamma \\
&= -\lceil \alpha - 1 + \varepsilon - \gamma \rceil - \gamma \\
&= -\lfloor \alpha - \gamma \rfloor - \gamma \\
&= -\gamma({}_\alpha L).
\end{split}
\end{equation*}
Hence, we have
\[
\left({}_{\gamma({}_\alpha L)} L\right)^\vee 
= \left({}_\alpha L\right)^\vee 
= {}_{-\alpha + 1 - \varepsilon}(L^\vee) 
= {}_{-\gamma({}_\alpha L)}(L^\vee),
\] 
by Corollary \ref{p-cor1.7} (see also Corollary \ref{p-cor4.3}). 

This completes the proof of Lemma \ref{p-lem13.1}.
\end{proof}

The main result of this section is the following theorem.  
One of the main ingredients in the proof of 
Theorem \ref{p-thm13.2} is Theorem \ref{p-thm12.3}.

\begin{thm}[Dual bundles, see Theorem \ref{p-thm1.12}]\label{p-thm13.2}
Let $(E, h)$ be an acceptable vector bundle on $\Delta^*$, 
and let $a$ be any real number. Then,
\[
\left({}_a E\right)^\vee = {}_{-a + 1 - \varepsilon}\left(E^\vee\right)
\]
holds for any sufficiently small $\varepsilon > 0$. 

Moreover, let $\{v_1, \ldots, v_r\}$ be a local frame of ${}_a E$ near the origin, 
compatible with the parabolic filtration, such that 
$v_i \in {}_{b_i} E \setminus {}_{<b_i} E$ for each $i$. 
For each $i$, define
\begin{equation}\label{p-eq13.1}
v_i^\vee := (-1)^{i-1} \, v_1 \wedge 
\cdots \wedge v_{i-1} \wedge v_{i+1} \wedge 
\cdots \wedge v_r \otimes (v_1 \wedge \cdots \wedge v_r)^{\otimes -1}.
\end{equation}
Then $\{v_1^\vee, \ldots, v_r^\vee\}$ forms a local frame of ${}_{-a + 1 - \varepsilon}(E^\vee)$ 
near the origin, compatible with the parabolic filtration, such that 
\[
v_i^\vee \in {}_{-b_i}(E^\vee) \setminus {}_{<-b_i}(E^\vee)
\]
for each $i$. In particular, we have 
\[
\Par _a(E, h)=\{b_1, \ldots, b_r\} \quad 
\text{ and }\quad \Par _{-a+1-\varepsilon} (E^\vee, h^\vee)=\{ -b_1, \ldots, -b_r\}. 
\]
\end{thm}

\begin{proof}[Proof of Theorem~\ref{p-thm13.2}] 
By Corollary~\ref{p-cor7.7}, we have 
\[
\gamma \left({}_{\gamma ({}_a E)} \det E\right) = \gamma ({}_a E).
\] 
Therefore, by Lemma~\ref{p-lem13.1}, we obtain 
\begin{equation}\label{p-eq13.2} 
\left({}_{\gamma ({}_a E)} \det E\right)^\vee = {}_{-\gamma ({}_a E)} (\det E)^\vee.
\end{equation}
As a result, we have
\begin{equation}\label{p-eq13.3}
\begin{aligned}
\left({}_a E\right)^\vee 
&= \bigwedge^{r-1} \left({}_a E\right) \otimes 
\left(\det({}_a E)\right)^\vee \\
&= \bigwedge^{r-1} \left({}_a E\right) \otimes 
\left({}_{\gamma({}_a E)} \det E\right)^\vee \\
&= \bigwedge^{r-1} \left({}_a E\right) 
\otimes {}_{-\gamma({}_a E)} \left(\det E\right)^\vee,
\end{aligned}
\end{equation}
where we have used Theorem~\ref{p-thm7.5} and \eqref{p-eq13.2}.

By \eqref{p-eq13.1} and \eqref{p-eq13.3}, it follows that 
\[
v^\vee _i \in {}_{-b_i} (E^\vee)
\] 
since $\gamma ({}_a E)=\sum _{i=1}^r b_i$ by Theorem \ref{p-thm12.3}. 

Since $\bm v:=\{v_1, \ldots, v_r\}$ is a local frame of ${}_a E$, 
we have $b_i\in (a-1, a]$ for every $i$ by Lemma \ref{p-lem7.16}. 
Therefore, $-b_i\leq -a+1-\varepsilon$ for every $i$ when $0<\varepsilon \ll 1$. 
Thus, $v^\vee _i\in {}_{-a +1-\varepsilon} (E^\vee)$ for 
every $i$. 

By definition, $\bm v':=\{v_1^\vee, \ldots, v_r^\vee\}$ is 
the dual frame of $\{v_1, \ldots, v_r\}$. Thus, it gives a local 
frame of $({}_a E)^\vee$ near the origin. 
This implies that 
\begin{equation}\label{p-eq13.4}
({}_a E)^\vee \subset {}_{-a+1-\varepsilon} (E^\vee). 
\end{equation}

\begin{claim}
We have the inclusion
\[
{}_{-a + 1 - \varepsilon}(E^\vee) \subset \left({}_a E\right)^\vee
\]
for any sufficiently small $\varepsilon > 0$.
\end{claim}

\begin{proof}[Proof of Claim]
Let $f \in {}_{-a + 1 - \varepsilon}(E^\vee)$. Then locally near $0$, we can write
\[
f = \sum_{i=1}^r f_i(z) v_i^\vee(z),
\]
where each $f_i$ is holomorphic outside $0$. Since
\[
|f_i(z)| = |v_i(f)| \leq |v_i(z)|_h 
\cdot |f(z)|_{h^\vee} = O\left( \frac{1}{|z|^{1 - \delta}} \right)
\]
for small $\delta > 0$, we conclude that each $f_i(z)$ is holomorphic near $0$.

Hence, $f$ extends holomorphically, and we obtain the desired inclusion.
This completes the proof of Claim.
\end{proof}

Therefore,
by combining \eqref{p-eq13.4} with Claim, we obtain the equality:
\[
\left({}_a E\right)^\vee = {}_{-a + 1 - \varepsilon}(E^\vee).
\]

As shown above, 
$\bm{v}^\vee = \{v_1^\vee, \ldots, v_r^\vee\}$ is a local frame of 
${}_{-a + 1 - \varepsilon}(E^\vee)$ near the origin. 
We have already proved that $v_i^\vee \in {}_{-b_i}(E^\vee)$ for every $i$. 
By definition, it is easy to verify that
\[
H(h^\vee, \bm{v}^\vee) = H(h, \bm{v})^{-1}.
\]
Therefore,
\[
\begin{split}
\gamma \left( {}_{-a + 1 - \varepsilon}(E^\vee) \right) 
&= -\frac{1}{2} \lim_{z \to 0} \frac{\log \det H(\bm{v}^\vee)}{\log |z|} \\
&= -\frac{1}{2} \lim_{z \to 0} \frac{-\log \det H(\bm{v})}{\log |z|} \\
&= -\gamma({}_a E) = \sum_{i=1}^r (-b_i),
\end{split}
\]
by Theorem~\ref{p-thm12.3}. 
Finally, by Corollary~\ref{p-cor12.8}, 
the frame $\bm{v}^\vee$ satisfies all the desired properties. 

This completes the proof of Theorem \ref{p-thm13.2}. 
\end{proof}

For each $i$, we set
\[
v'_i := v_i \cdot |z|^{b_i}.
\]
We denote $\bm{v}' := \{v'_1, \ldots, v'_r\}$. 
As in \ref{p-say7.1}, we define
\[
H(h, \bm{v}') := \left(h(v'_i, v'_j)\right)_{i, j}.
\]

Now we are ready to prove the following theorem. 
Although this property does not play a role in the present work, 
it is of independent interest.

\begin{thm}[Weak norm estimate, see Theorem \ref{p-thm1.13}]\label{p-thm13.3}
Let $\{v_1, \ldots, v_r\}$ be a local frame of ${}_a E$ around 
the origin, compatible with the parabolic filtration, such that 
\[
v_i \in {}_{b_i} E \setminus {}_{<b_i} E \quad \text{for every } i.
\]
Then there exist positive constants $C$ and $M$ such that
\[
C^{-1}(-\log |z|)^{-M} I_r \leq H(h, \bm{v}')(z) \leq C(-\log |z|)^M I_r
\]
holds in a neighborhood of the origin, where $I_r$ is 
the identity matrix of size $r$.

This means that both
\[
C(-\log |z|)^M I_r - H(h, \bm{v}')(z)
\quad \text{and} \quad
H(h, \bm{v}')(z) - C^{-1}(-\log |z|)^{-M} I_r
\]
are positive semidefinite around the origin.
\end{thm}

\begin{proof}[Proof of Theorem \ref{p-thm13.3}]
For each $i$, set
\[
(v_i^\vee)' := v_i^\vee \cdot |z|^{-b_i}.
\]
We denote $(\bm{v}^\vee)' := \{(v_1^\vee)', \ldots, (v_r^\vee)'\}$. 
By Lemma \ref{p-lem5.4}, there exist positive constants $C'$, $M'$ such that
\[
H(h, \bm{v}')(z) \leq C'(-\log |z|)^{M'} I_r
\]
holds near the origin. 
Similarly, applying 
Lemma \ref{p-lem5.4} to the dual bundle, 
we obtain positive constants $C''$ and $M''$ such that
\[
H(h^\vee, (\bm{v}^\vee)')(z) \leq C''(-\log |z|)^{M''} I_r
\]
holds near the origin. 
By definition, it is easy to verify that
\[
H(h, \bm{v}')(z) = \left(H(h^\vee, (\bm{v}^\vee)')(z)\right)^{-1}.
\]
Combining these inequalities, we conclude that 
there exist positive constants $C$ and $M$ such that
\[
C^{-1}(-\log |z|)^{-M} I_r \leq H(h, \bm{v}')(z) \leq C(-\log |z|)^M I_r
\]
holds around the origin. This completes the proof of Theorem \ref{p-thm13.3}.
\end{proof}

As a direct consequence of Theorem \ref{p-thm13.3}, we obtain 
the following useful estimate.

\begin{cor}\label{p-cor13.4} 
Let $\{v_1, \ldots, v_r\}$ be a local frame of ${}_a E$ around 
the origin, compatible with the parabolic filtration, such that 
\[
v_i \in {}_{b_i} E \setminus {}_{<b_i} E \quad \text{for every } i.
\]
Then there exist positive constants $C_0$ and $M_0$ such that
\[
\frac{C_0^{-1}}{|z|^{b_i}} \left(-\log |z|\right)^{-M_0} 
\leq |v_i|_h \leq \frac{C_0}{|z|^{b_i}} \left(-\log |z|\right)^{M_0}
\]
holds for every $i$ in a neighborhood of the origin.
\end{cor}

\begin{proof}[Proof of Corollary \ref{p-cor13.4}]
This follows directly from Theorem \ref{p-thm13.3}.
\end{proof}

\section{Examples of filtered bundles}\label{p-sec14}

Before discussing the prolongation of tensor products 
and Hom bundles of acceptable vector bundles on 
$\Delta^*$, we set up the framework of filtered bundles. 
We use the same notation as in Section \ref{p-sec8}. 
Let us begin with a simple example, which we 
will use again in Section \ref{p-sec15}. 

\begin{ex}\label{p-ex14.1}
Note that $\mathcal{O}_\Delta(*[0])$ is itself 
a locally free $\mathcal{O}_\Delta(*[0])$-module of rank one. 
Let $\mathcal{P}^{(c)}_\ast\left(\mathcal{O}_\Delta(*[0])\right)$ 
denote the filtered bundle over $\mathcal{O}_\Delta(*[0])$ defined by
\[
\mathcal{P}^{(c)}_a \left(\mathcal{O}_\Delta(*[0])\right) 
= \mathcal{O}_\Delta(\lfloor a - c \rfloor [0]).
\] 
\end{ex}

\begin{rem}\label{p-rem14.2}
Let $(\mathcal{O}_{\Delta^*}, h_c)$ be a flat line bundle on $\Delta^*$, where
\[
h_c := \frac{|\cdot|^2}{|z|^{2c}} = |\cdot|^2 \cdot e^{-2c \log |z|}.
\]
Then we can verify that
\[
\mathcal{P}^{h_c}_a \mathcal{O}_{\Delta^*} 
= \mathcal{O}_\Delta\left(\lfloor a - c \rfloor [0]\right)
\]
holds for every $a \in \mathbb{R}$. 
Hence, the filtered 
bundle $\mathcal{P}^{(c)}_\ast\left(\mathcal{O}_\Delta(\ast[0])\right)$ 
in Example~\ref{p-ex14.1} can be realized as 
the filtered prolongation of the acceptable line bundle 
$(\mathcal{O}_{\Delta^*}, h_c)$ over $\Delta^*$. 
In particular, we can view 
$\mathcal{P}^{(0)}_\ast\left(\mathcal{O}_{\Delta}(\ast [0])\right)$ 
as the filtered prolongation of the trivial Hermitian line bundle 
$(\mathcal{O}_{\Delta^*}, |\cdot|^2)$. 
Note that 
$\mathcal{P}^{(0)}_\ast\left(\mathcal{O}_{\Delta}(\ast [0])\right)$ 
will be used in Proposition~\ref{p-prop15.1}.
\end{rem}

Let $\mathcal{P}_\ast E_1$ and $\mathcal{P}_\ast E_2$ be 
filtered bundles of rank $r_1$ and $r_2$ on $(\Delta, 0)$, respectively. Then 
\[
E_1 \otimes E_2 \quad \text{and} \quad 
\Hom_{\mathcal{O}_\Delta(*[0])}(E_1, E_2)
\]
are locally free $\mathcal{O}_\Delta(*[0])$-modules of rank $r_1 r_2$.

Let $a \in \mathbb{R}$. We define
\begin{align*}
\mathcal{P}_a(E_1 \otimes E_2) & 
:= \sum_{b + c \leq a} \mathcal{P}_b E_1 \otimes \mathcal{P}_c E_2, \\
\mathcal{P}_a \Hom(E_1, E_2) & 
:= \left\{ f \in \Hom_{\mathcal{O}_\Delta(*[0])}(E_1, E_2) \mid 
f(\mathcal{P}_k E_1) \subset \mathcal{P}_{a+k} E_2 \text{ for all } k 
\in \mathbb{R} \right\}.
\end{align*}

Suppose $\mathcal{P}ar(\mathcal{P}_0 E_1) 
= \{b_1, \dots, b_{r_1}\}$ and $\mathcal{P}ar(\mathcal{P}_0 E_2) 
= \{c_1, \dots, c_{r_2}\}$. Let $\{v_i\}$ and $\{w_j\}$ be frames 
of $\mathcal{P}_0 E_1$ and $\mathcal{P}_0 E_2$, 
respectively, which are compatible with the filtrations.

By the definition of filtered bundles, we have
\[
E_1 \otimes E_2 = \sum_{i, j} \mathcal{O}_\Delta(*[0]) \cdot v_i \otimes w_j,
\]
\[
\Hom(E_1, E_2) = \sum_{i, j} \mathcal{O}_\Delta(*[0]) \cdot v_i^\vee \otimes w_j.
\]
We used condition (1) in the definition of a filtered 
bundle (see Definition \ref{p-def8.2}).

\begin{prop}\label{p-prop14.3}
The family $\mathcal{P}_a(E_1 \otimes E_2)\,(a \in \mathbb{R})$ 
defines a filtered bundle structure over $E_1 \otimes E_2$. 
We denote this filtered bundle by 
$\mathcal{P}_\ast E_1 \otimes \mathcal{P}_\ast E_2$.
\end{prop}

\begin{proof}[Proof of Proposition \ref{p-prop14.3}]
Fix $a \in \mathbb{R}$. By definition, we have 
$v_i \otimes w_j \in \mathcal{P}_{b_i} E_1 \otimes \mathcal{P}_{c_j} E_2$. Set
\[
n_{ij,a} := \max\{ n \in \mathbb{Z} \mid n + b_i + c_j \leq a \}.
\]
Then
\[
z^{-n_{ij,a}} v_i \otimes w_j \in 
\mathcal{P}_{b_i + n_{ij,a}} E_1 \otimes 
\mathcal{P}_{c_j} E_2 \subset \mathcal{P}_a(E_1 \otimes E_2).
\]
Hence, 
\[
\sum_{i,j} \mathcal{O}_\Delta \cdot z^{-n_{ij,a}} 
v_i \otimes w_j \subset \mathcal{P}_a(E_1 \otimes E_2).
\]

Let $b, c \in \mathbb{R}$. Set
\[
n_{i,b} := \max\{ n \in \mathbb{Z} 
\mid n + b_i \leq b \}, \quad m_{j,c} := \max\{ n \in \mathbb{Z} \mid n + c_j \leq c \}.
\]
Then $\{ z^{-n_{i,b}} v_i \}$ and $\{ z^{-m_{j,c}} w_j \}$ 
are frames of $\mathcal{P}_b E_1$ and $\mathcal{P}_c E_2$, respectively. Therefore,
\[
\mathcal{P}_a(E_1 \otimes E_2) = 
\sum_{b + c \leq a} \mathcal{O}_\Delta \cdot z^{-n_{i,b}} v_i \otimes z^{-m_{j,c}} w_j.
\]
By the maximality of $n_{ij,a}$, we obtain
\[
\mathcal{P}_a(E_1 \otimes E_2) = 
\sum_{i,j} \mathcal{O}_\Delta \cdot z^{-n_{ij,a}} v_i \otimes w_j.
\]

It is clear that condition (1) in the 
definition of filtered bundles is satisfied (see Definition \ref{p-def8.2}). 
Condition (2) in Definition \ref{p-def8.2} follows from the 
identity $n_{ij,a+n} = n_{ij,a} + n$. 
Choose $\epsilon_{ij,a} > 0$ small enough 
such that $n_{ij,a + \epsilon_{ij,a}} = n_{ij,a}$, 
and set $\epsilon := \min_{i,j} \epsilon_{ij,a}$. Then,
\[
\mathcal{P}_{a+\epsilon}(E_1 \otimes E_2) = 
\sum_{i,j} \mathcal{O}_\Delta \cdot z^{-n_{ij,a}} 
v_i \otimes w_j = \mathcal{P}_a(E_1 \otimes E_2).
\]
Therefore, $\mathcal{P}_\ast(E_1 \otimes E_2)$ 
defines a filtered bundle over $E_1 \otimes E_2$, as desired.
\end{proof}

\begin{prop}\label{p-prop14.4}
The increasing family of $\mathcal{O}_\Delta$-modules $\mathcal{P}_a 
\Hom(E_1, E_2)\,(a \in \mathbb{R})$ defines a filtered bundle structure over 
$\Hom(E_1, E_2)$. This filtered bundle is denoted by 
$\Hom(\mathcal{P}_\ast E_1, \mathcal{P}_\ast E_2)$.
\end{prop}

\begin{proof}[Proof of Proposition \ref{p-prop14.4}]
Let $a \in \mathbb{R}$ and 
$f \in \mathcal{P}_a \Hom(E_1, E_2)$. 
By definition, 
\[f(\mathcal{P}_{b_i} E_1) 
\subset \mathcal{P}_{a + b_i} E_2
\] 
holds. 
Conversely, if an $\mathcal{O}_\Delta(*[0])$-module 
morphism $f \colon E_1 \to E_2$ satisfies 
\[
f(\mathcal{P}_{b_i} E_1) 
\subset \mathcal{P}_{a + b_i} E_2
\] 
for all $i$, 
then $f \in \mathcal{P}_a \Hom(E_1, E_2)$. 

Hence,
\[
\mathcal{P}_a \Hom(E_1, E_2) = 
\left\{ f \in \Hom_{\mathcal{O}_\Delta(*[0])}(E_1, E_2) \,\middle|\, 
f(\mathcal{P}_{b_i} E_1) \subset \mathcal{P}_{a + b_i} E_2 \text{ for all } i\right\}.
\]
Define 
\[
m_{ij,a} := \max\{ m \in \mathbb{Z} \mid m + c_j \leq b_i + a \}.
\]
Then, by the above discussion,
\[
\mathcal{P}_a \Hom(E_1, E_2) = 
\sum_{i,j} \mathcal{O}_\Delta \cdot v_i^\vee \otimes z^{-m_{ij,a}} w_j.
\]

It is clear that condition (1) in Definition \ref{p-def8.2} 
is satisfied. 
Condition (2) in Definition \ref{p-def8.2} follows from $m_{ij,a+n} = 
m_{ij,a} + n$. Choose $\epsilon_i > 0$ such that 
$\mathcal{P}_{a + b_i + \epsilon_i} E_2 = \mathcal{P}_{a + b_i} E_2$, 
and set $\epsilon := \min \epsilon_i$. Then,
\[
\mathcal{P}_{a+\epsilon} \Hom(E_1, E_2) = \mathcal{P}_a \Hom(E_1, E_2).
\]
Therefore, $\mathcal{P}_\ast \Hom(E_1, E_2)$ 
defines a filtered bundle over $\Hom(E_1, E_2)$.
\end{proof}

\section{Dual bundles revisited}\label{p-sec15}

In this section, we study prolongations of dual bundles 
within the framework of filtered bundles.

\begin{prop}\label{p-prop15.1}
Let \( (E, h) \) be an acceptable 
vector bundle of rank \( r \) on \( \Delta^* \). Then
\[
\mathcal{P}^{h^\vee}_\ast E^\vee = 
\Hom(\mathcal{P}^h_\ast E, \mathcal{P}^{(0)}_{\ast} \mathcal{O}_\Delta(*[0])).
\]
\end{prop}

Note that 
\[
\Hom(\mathcal{P}^h_\ast E, \mathcal{P}^{(0)}_{\ast} \mathcal{O}_\Delta(*[0]))
\]
in Proposition \ref{p-prop15.1} 
is a filtered bundle, as described in Proposition \ref{p-prop14.4}, 
since both \( \mathcal{P}^h_\ast E \) and 
\( \mathcal{P}^{(0)}_\ast \mathcal{O}_\Delta(*[0]) \) 
are filtered bundles (see Sections \ref{p-sec4} and \ref{p-sec8}).

\begin{proof}[Proof of Proposition \ref{p-prop15.1}]
Let \( k \in \mathbb{R} \) be arbitrary. Take any 
\( f \in \mathcal{P}_k \Hom(E, \mathcal{O}_\Delta(*[0])) \). 
By definition, for any \( a \in \mathbb{R} \), we have
\[
f(\mathcal{P}^h_a E) \subset \mathcal{P}^{(0)}_{a + k} \mathcal{O}_\Delta(*[0]).
\]
Take \( a = -k + 1 - \epsilon \) with any \( 0 < \epsilon \ll 1 \). Then
\[
f(\mathcal{P}^h_{-k + 1 - \epsilon} E) \subset 
\mathcal{P}^{(0)}_{1 - \epsilon} \mathcal{O}_\Delta(*[0]) = \mathcal{O}_\Delta.
\]
This implies
\[
f \in (\mathcal{P}^h_{-k + 1 - \epsilon} E)^\vee 
= \mathcal{P}^{h^\vee}_{k + \epsilon - \delta} E^\vee
\]
for any sufficiently small \( \delta > 0 \), and hence
\[
f \in \mathcal{P}^{h^\vee}_{k + \varepsilon'} E^\vee
\] 
holds for any $0<\varepsilon '\ll 1$. 
By Lemma \ref{p-lem7.9}, we conclude
\[
f \in \mathcal{P}^{h^\vee}_k E^\vee.
\]
Thus,
\[
\mathcal{P}_k \Hom(E, \mathcal{O}_\Delta(*[0])) 
\subset \mathcal{P}^{h^\vee}_k E^\vee.
\]

We now prove the opposite inclusion. Fix \( k \in \mathbb{R} \).  
It suffices to show that for 
any \( f \in \mathcal{P}^{h^\vee}_k E^\vee \) and any \( a \in \mathbb{R} \), we have
\[
f(\mathcal{P}^h_a E) \subset \mathcal{P}^{(0)}_{a + k} \mathcal{O}_\Delta(*[0]).
\]

Suppose \( \mathcal{P}ar(\mathcal{P}^h_0 E) = \{b_1, \dots, b_r\} \),  
and let \( \{v_i\}_{i=1}^r \) be a 
frame of \( \mathcal{P}^h_0 E \) compatible with the parabolic filtration.  
Let \( \{v_i^\vee\}_{i=1}^r \) be the dual frame of \( \{v_i\}_{i=1}^r \) 
as in Theorem \ref{p-thm13.2}.  
Define
\[
n_i := \max \{ n \in \mathbb{Z} \mid n - b_i \leq k \}.
\]
Then \( \{ z^{-n_i} v_i^\vee \}_{i=1}^r \) is 
a frame of \( \mathcal{P}^{h^\vee}_k E^\vee \) 
by Lemma \ref{p-lem7.17} and Theorem \ref{p-thm13.2}.

Fix \( a \in \mathbb{R} \), and define
\[
m_i := \max \{ m \in \mathbb{Z} \mid m + b_i \leq a \}.
\]
By Lemma \ref{p-lem7.17} again, the 
set \( \{ z^{-m_i} v_i \}_{i=1}^r \) forms a frame of \( \mathcal{P}^h_a E \).

Therefore, in order to prove 
that \( f(\mathcal{P}^h_a E) \subset \mathcal{P}^{(0)}_{a+k} 
\mathcal{O}_\Delta(*[0]) \) for 
any \( f \in \mathcal{P}^{h^\vee}_k E^\vee \), it suffices to check
\[
z^{-n_i} v_i^\vee (z^{-m_i} v_i) 
\in \mathcal{P}^{(0)}_{a + k} 
\mathcal{O}_\Delta(*[0]) = \mathcal{O}_\Delta(\lfloor a + k \rfloor [0]).
\]
This follows from the inequality
\[
n_i + m_i = (n_i - b_i) + (m_i + b_i) \leq k + a,
\]
which implies
\[
n_i + m_i \leq \lfloor a + k \rfloor.
\]
Hence, for all \( f \in \mathcal{P}^{h^\vee}_k E^\vee \), we have
\[
f(\mathcal{P}^h_a E) \subset \mathcal{P}^{(0)}_{a + k} \mathcal{O}_\Delta(*[0]).
\]
Since \( a \in \mathbb{R} \) is arbitrary, we obtain the inclusion
\[
\mathcal{P}^{h^\vee}_k E^\vee \subset \mathcal{P}_k 
\Hom(E, \mathcal{O}_\Delta(*[0])).
\]

Therefore,
\[
\mathcal{P}^{h^\vee}_k E^\vee = \mathcal{P}_k 
\Hom(E, \mathcal{O}_\Delta(*[0])).
\]
Since this equality holds 
for every \( k \in \mathbb{R} \), the proof of Proposition \ref{p-prop15.1} is complete.
\end{proof}

\section{On tensor products}\label{p-sec16}

In this section, we discuss the prolongations of tensor products of acceptable bundles in detail. 
Throughout this section, we use the notation 
\[\lceil a \rceil := \min \{ n \in \mathbb{Z} \mid n \geq a \}
\] for $a \in \mathbb{R}$. 

\begin{prop}\label{p-prop16.1}
Let $(E_1, h_1)$ and $(E_2, h_2)$ be acceptable vector 
bundles of rank $r_1$ and $r_2$, respectively.  
Suppose that  
\[
\mathcal{P}ar(\mathcal{P}^{h_1}_{0}E_1) 
= \{b_1, \dots, b_{r_1}\}, \quad 
\mathcal{P}ar(\mathcal{P}^{h_2}_{0}E_2) = \{c_1, \dots, c_{r_2}\}.
\]  
Let $\{v_1, \ldots, v_{r_1}\}$ 
and $\{w_1, \ldots, w_{r_2}\}$ 
be frames of $\mathcal{P}^{h_1}_{0}E_1$ and $\mathcal{P}^{h_2}_{0}E_2$, respectively, such that  
\[
v_i \in \mathcal{P}^{h_1}_{b_i}E_1 
\setminus \mathcal{P}^{h_1}_{<b_i}E_1, \quad w_j 
\in \mathcal{P}^{h_2}_{c_j}E_2 \setminus \mathcal{P}^{h_2}_{<c_j}E_2,
\]  
and are compatible with the parabolic filtrations. Then:

\begin{itemize}
  \item[(i)] The set $\{z^{\lceil b_i + c_j \rceil} 
  v_i \otimes w_j\}_{1 \leq i \leq r_1,\ 1 \leq j \leq r_2}$ 
  forms a frame of $\mathcal{P}^{h_1 \otimes h_2}_{0}(E_1 \otimes E_2)$.
  
  \item[(ii)] The set of parabolic weights 
  of $\mathcal{P}^{h_1 \otimes h_2}_{0}(E_1 \otimes E_2)$ is given by  
  \[
  \mathcal{P}ar\left( \mathcal{P}^{h_1 \otimes h_2}_{0}(E_1 \otimes E_2) \right) 
  = \left\{ b_i + c_j - \lceil b_i + c_j \rceil 
  \right\}_{1 \leq i \leq r_1,\ 1 \leq j \leq r_2}.
  \]  
  In particular, the set $\{z^{\lceil b_i + c_j 
  \rceil} v_i \otimes w_j\}_{1 \leq i \leq r_1,\ 1 \leq j \leq r_2}$ 
  is compatible with the parabolic filtration.
\end{itemize}
\end{prop}

\begin{proof}[Proof of Proposition \ref{p-prop16.1}]
Since \( |v_i \otimes w_j|_{h_1 \otimes h_2} = |v_i|_{h_1} \cdot |w_j|_{h_2} \), we have 
\[
v_i \otimes w_j \in \mathcal{P}^{h_1 \otimes h_2}_{b_i + c_j}(E_1 \otimes E_2),
\]
and hence
\[
z^{\lceil b_i + c_j \rceil} v_i \otimes w_j \in 
\mathcal{P}^{h_1 \otimes h_2}_{b_i + c_j - \lceil b_i + c_j \rceil}(E_1 \otimes E_2) 
\subset \mathcal{P}^{h_1 \otimes h_2}_0(E_1 \otimes E_2).
\]

\setcounter{step}{0}
\begin{step}\label{p-step16.1-1} 
In this step, we prove statement (i).

To prove (i), it suffices to show that for every 
\( f \in \mathcal{P}^{h_1 \otimes h_2}_0(E_1 \otimes E_2) \), 
there exist holomorphic functions \( f_{ij} \in \mathcal{O}_\Delta \) such that
\[
f = \sum_{i,j} f_{ij} z^{\lceil b_i + c_j \rceil} v_i \otimes w_j.
\]

Since \( \{v_i\} \) and \( \{w_j\} \) are frames of 
\( E_1 \) and \( E_2 \) on \(\Delta^*\), respectively, 
the set \( \{ z^{\lceil b_i + c_j \rceil} 
v_i \otimes w_j \}_{i,j} \) forms a frame of \( E_1 \otimes E_2 \) 
on \(\Delta^*\).  
Hence, for 
any \( f \in \mathcal{P}^{h_1 \otimes h_2}_0(E_1 \otimes E_2) \), we can write
\[
f = \sum_{i,j} f_{ij} z^{\lceil b_i + c_j \rceil} v_i \otimes w_j,
\]
where each \( f_{ij} \) is holomorphic outside the origin.  
Therefore, it remains to show that \( f_{ij} \) is holomorphic at the origin.

Let \( \{v_i^\vee\} \) and \( \{w_j^\vee\} \) 
denote the dual frames of \( \{v_i\} \) and \( \{w_j\} \), respectively.  
Recall that 
\[
v_i^\vee \in \mathcal{P}^{h_1^\vee}_{-b_i} E_1^\vee, \quad 
w_j^\vee \in \mathcal{P}^{h_2^\vee}_{-c_j} E_2^\vee.
\]
Therefore,
\[
v_i^\vee \otimes w_j^\vee \in 
\mathcal{P}^{h_1^\vee \otimes h_2^\vee}_{-b_i - c_j}(E_1^\vee \otimes E_2^\vee).
\]
By Proposition~\ref{p-prop15.1}, we have
\[
v_i^\vee \otimes w_j^\vee \in 
\mathcal{P}_{-b_i - c_j} \Hom(E_1 \otimes E_2, \mathcal{O}_\Delta(*[0])).
\]

Since \( f \in \mathcal{P}^{h_1 \otimes h_2}_0(E_1 \otimes E_2) \), it follows that
\begin{align*}
(v_i^\vee \otimes w_j^\vee)(f) 
&= (v_i^\vee \otimes w_j^\vee)\left( \sum_{i,j} f_{ij} 
z^{\lceil b_i + c_j \rceil} v_i \otimes w_j \right) \\
&= z^{\lceil b_i + c_j \rceil} f_{ij} 
\in \mathcal{P}^{(0)}_{-b_i - c_j} \mathcal{O}_\Delta(*[0]).
\end{align*}

Since \( 0 \leq \lceil b_i + c_j \rceil - (b_i + c_j) < 1 \), we have
\[
f_{ij} \in \mathcal{P}^{(0)}_{\lceil b_i + c_j \rceil - (b_i + c_j)} \mathcal{O}_\Delta(*[0]) 
= \mathcal{O}_\Delta(\lfloor \lceil b_i + c_j \rceil - b_i - c_j \rfloor [0]) 
= \mathcal{O}_\Delta.
\]

This completes the proof of (i).
\end{step}

\begin{step}\label{p-step16.1-2}
In this step, we prove statement (ii).

As before, we define
\[
\bm{v} := \{ v_1, \ldots, v_{r_1} \}, \quad 
\bm{w} := \{ w_1, \ldots, w_{r_2} \}, \quad 
\bm{v} \otimes \bm{w} := \{ v_i \otimes w_j \}_{i,j}.
\]
We further define 
\[
\left(\bm v \otimes \bm w\right)^\sharp 
:= \left\{ z^{\lceil b_i + c_j \rceil} 
v_i \otimes w_j \right\}_{1 \leq i \leq r_1, \, 1 \leq j \leq r_2}.
\]
As shown in Step~\ref{p-step16.1-1}, 
the set $\left(\bm v \otimes \bm w\right)^\sharp$ forms a frame of 
$\mathcal P^{h_1 \otimes h_2}_0(E_1 \otimes E_2)$. 

We consider the Hermitian matrix
\[
H(h_1 \otimes h_2, \bm{v} \otimes \bm{w}) 
:= \left( h_1(v_i, v_j) \cdot h_2(w_k, w_l) \right),
\]
whose $\left((i-1)r_2 + k, \, (j-1)r_2 + l\right)$-th entry is given by 
$h_1(v_i, v_j) \cdot h_2(w_k, w_l)$. Similarly, define
\[
H\left(h_1 \otimes h_2, \left(\bm{v} \otimes \bm{w}\right)^\sharp\right) 
:= \left( z^{\lceil b_i + c_k \rceil} \cdot \overline{z}^{\lceil b_j + c_l \rceil} 
h_1(v_i, v_j) \cdot h_2(w_k, w_l) \right).
\]
Then we have
\[
\begin{split}
\det H&\left(h_1 \otimes h_2, \left(\bm{v} \otimes \bm{w}\right)^\sharp\right)\\
&= \left( \prod_{i, k} z^{\lceil b_i + c_k \rceil} \right)
\left( \prod_{j, l} \overline{z}^{\lceil b_j + c_l \rceil} \right)
\cdot (\det H(h_1, \bm{v}))^{r_2} \cdot (\det H(h_2, \bm{w}))^{r_1} \\
&= |z|^{2 \sum_{i, j} \lceil b_i + c_j \rceil} 
\cdot (\det H(h_1, \bm{v}))^{r_2} \cdot (\det H(h_2, \bm{w}))^{r_1}.
\end{split}
\] 
Therefore,
\begin{equation}
\begin{split}\label{p-eq16.1}
\gamma\left( \mathcal{P}^{h_1 \otimes h_2}_0(E_1 \otimes E_2) \right) 
&= -\frac{1}{2} \liminf_{z \to 0} 
\frac{ \log \det H(h_1 \otimes h_2, \left(\bm{v} \otimes \bm{w}\right)^\sharp) }
{ \log |z| } \\
&= -\frac{1}{2} \lim_{z \to 0} 
\frac{ \log \det H(h_1 \otimes h_2, \left(\bm{v} \otimes \bm{w}\right)^\sharp) }
{ \log |z| } \\
&= -\sum_{i, j} \lceil b_i + c_j \rceil 
+ r_2 \cdot \gamma\left( \mathcal{P}^{h_1}_0(E_1) \right) 
+ r_1 \cdot \gamma\left( \mathcal{P}^{h_2}_0(E_2) \right) \\
&= -\sum_{i, j} \lceil b_i + c_j \rceil 
+ r_2 \cdot \sum_i b_i 
+ r_1 \cdot \sum_j c_j
\end{split}
\end{equation}
by Corollary \ref{p-cor7.6} and Theorem \ref{p-thm12.3}.  

On the other hand, we have
\begin{equation}\label{p-eq16.2}
z^{\lceil b_i + c_j \rceil} v_i \otimes w_j \in 
\mathcal{P}^{h_1 \otimes h_2}_{b_i + c_j - \lceil b_i + c_j \rceil}(E_1 \otimes E_2),
\end{equation}
and
\begin{equation}\label{p-eq16.3}
\sum_{i,j} \left( b_i + c_j - \lceil b_i + c_j \rceil \right)
= -\sum_{i,j} \lceil b_i + c_j \rceil + r_2 \sum_i b_i + r_1 \sum_j c_j.
\end{equation}

By Corollary~\ref{p-cor12.8}, together with \eqref{p-eq16.1}, 
\eqref{p-eq16.2}, and \eqref{p-eq16.3}, 
we conclude that $\left(\bm v \otimes \bm w\right)^\sharp$ is a frame of 
$\mathcal{P}^{h_1 \otimes h_2}_0(E_1 \otimes E_2)$ compatible with the parabolic filtration, 
\[
\mathcal{P}ar\left( \mathcal{P}^{h_1 \otimes h_2}_0(E_1 \otimes E_2) \right) 
= \left\{ b_i + c_j - \lceil b_i + c_j \rceil \right\}_{1 \leq i \leq r_1,\ 1 \leq j \leq r_2},
\]
and
\[
z^{\lceil b_i + c_j \rceil} v_i \otimes w_j \in 
\mathcal{P}^{h_1 \otimes h_2}_{b_i + c_j - \lceil b_i + c_j \rceil}(E_1 \otimes E_2)
\setminus 
\mathcal{P}^{h_1 \otimes h_2}_{< b_i + c_j - \lceil b_i + c_j \rceil}(E_1 \otimes E_2).
\]

Thus, statement (ii) is proved.
\end{step}

We now complete the proof of Theorem~\ref{p-prop16.1}.
\end{proof}

We are now ready 
to describe the behavior of the prolongation of the tensor product of acceptable bundles.

\begin{thm}[Tensor products, see Theorem \ref{p-thm1.14}]\label{p-thm16.2}
Let $(E_1, h_1)$ and $(E_2, h_2)$ be 
acceptable vector bundles of rank $r_1$ and $r_2$, respectively. 
Then the parabolic filtration on $E_1\otimes E_2$ induced 
by $h_1\otimes h_2$ coincides with the tensor product filtration:
\[
\mathcal{P}^{h_1 \otimes h_2}_\ast(E_1 \otimes E_2)
= \mathcal{P}^{h_1}_\ast(E_1) \otimes \mathcal{P}^{h_2}_\ast(E_2).
\]
Equivalently, for every \( a \in \mathbb{R} \),
\[
\mathcal P^{h_1 \otimes h_2}_a(E_1 \otimes E_2) 
= \sum_{a_1 + a_2 \leq a} \mathcal P^{h_1}_{a_1}(E_1) \otimes 
\mathcal P^{h_2}_{a_2}(E_2).
\]
\end{thm}

\begin{proof}[Proof of Theorem \ref{p-thm16.2}]
Let \( k \in \mathbb{R} \) be arbitrary. Then the following inclusion  
\[
\sum_{a + b \leq k} \mathcal{P}^{h_1}_a(E_1) \otimes \mathcal{P}^{h_2}_b(E_2)
\subset \mathcal{P}^{h_1 \otimes h_2}_k(E_1 \otimes E_2)
\]  
holds obviously by definition. Hence, it 
suffices to prove the opposite inclusion.

Suppose that  
\[
\mathcal{P}ar(\mathcal{P}^{h_1}_0 E_1) = \{b_1, \ldots, b_{r_1}\}, \quad 
\mathcal{P}ar(\mathcal{P}^{h_2}_0 E_2) = \{c_1, \ldots, c_{r_2}\}.
\]  
Let \( \{v_1, \ldots, v_{r_1}\} \) and \( \{w_1, \ldots, w_{r_2}\} \) be frames of 
\( \mathcal{P}^{h_1}_0 E_1 \) and \( \mathcal{P}^{h_2}_0 E_2 \), respectively, 
compatible with the corresponding parabolic filtrations.

By Proposition \ref{p-prop16.1}, we have  
\[
\mathcal{P}ar\left( \mathcal{P}^{h_1 \otimes h_2}_0(E_1 \otimes E_2) \right)
= \{ b_i + c_j - \lceil b_i + c_j 
\rceil \}_{1 \leq i \leq r_1,\ 1 \leq j \leq r_2},
\]  
and the set  
\[
\{ z^{\lceil b_i + c_j \rceil} v_i \otimes w_j \}_{1 \leq i \leq r_1,\ 1 \leq j \leq r_2}
\]  
forms a frame of \( \mathcal{P}^{h_1 \otimes h_2}_0(E_1 \otimes E_2) \) 
compatible with the parabolic filtration.

For each \( (i,j) \in \{1, \ldots, r_1\} \times \{1, \ldots, r_2\} \), define  
\[
n_{ij} := \max \left\{ n \in \mathbb{Z} 
\mid n + b_i + c_j - \lceil b_i + c_j \rceil \leq k \right\}.
\]  
Then the set  
\[
\{ z^{-n_{ij} + \lceil b_i + c_j \rceil} v_i \otimes w_j \}_{i,j}
\]  
is a frame of \( \mathcal{P}^{h_1 \otimes h_2}_k(E_1 \otimes E_2) \).  
Since  
\[
z^{-n_{ij} + \lceil b_i + c_j \rceil} v_i \in 
\mathcal{P}^{h_1}_{n_{ij} - \lceil b_i + c_j \rceil + b_i}(E_1), 
\quad w_j \in \mathcal{P}^{h_2}_{c_j}(E_2),
\]  
and \( n_{ij} + b_i + c_j - \lceil b_i + c_j \rceil \leq k \), it follows that  
\[
z^{-n_{ij} + \lceil b_i + c_j \rceil} v_i \otimes w_j \in 
\mathcal{P}^{h_1}_{n_{ij} - \lceil b_i + c_j \rceil + b_i}(E_1) \otimes 
\mathcal{P}^{h_2}_{c_j}(E_2) 
\subset \sum_{a + b \leq k} \mathcal{P}^{h_1}_a(E_1) \otimes \mathcal{P}^{h_2}_b(E_2).
\]

Therefore, we obtain the inclusion  
\[
\mathcal{P}^{h_1 \otimes h_2}_k(E_1 \otimes E_2) 
\subset \sum_{a + b \leq k} \mathcal{P}^{h_1}_a(E_1) \otimes \mathcal{P}^{h_2}_b(E_2),
\]  
and hence the desired equality  
\[
\mathcal{P}^{h_1 \otimes h_2}_k(E_1 \otimes E_2) 
= \sum_{a + b \leq k} \mathcal{P}^{h_1}_a(E_1) \otimes \mathcal{P}^{h_2}_b(E_2)
\]  
holds for every \( k \in \mathbb{R} \). This completes the proof of Theorem \ref{p-thm16.2}.
\end{proof}

\section{On Hom bundles}\label{p-sec17}

In this final section, we prove that the parabolic filtration on 
$\Hom (E_1, E_2)$ induced by 
$h^\vee _1\otimes h_2$ coincides with the filtration on the filtered bundle 
$\Hom(\mathcal{P}^{h_1}_\ast E_1,  \mathcal{P}^{h_2}_\ast E_2)$. 

\begin{prop}\label{p-prop17.1}
Let \( (E_1, h_1) \) and \( (E_2, h_2) \) be 
acceptable vector bundles of 
rank \( r_1 \) and \( r_2 \), respectively, defined on \( \Delta^* \).  
Then the parabolic filtration 
on the Hom bundle \(\Hom(E_1, E_2) \) induced by the 
metric \( h_1^\vee \otimes h_2 \) coincides with the filtration on the filtered bundle 
$\Hom(\mathcal{P}^{h_1}_\ast E_1, \mathcal{P}^{h_2}_\ast E_2)$: 
\[
\mathcal{P}^{h_1^\vee \otimes h_2}_\ast \Hom(E_1, E_2) = 
\Hom(\mathcal{P}^{h_1}_\ast E_1,\, \mathcal{P}^{h_2}_\ast E_2).
\]
\end{prop}

\begin{proof}[Proof of Proposition \ref{p-prop17.1}]
As usual, we denote the filtered bundle 
\(\Hom(\mathcal{P}^{h_1}_\ast E_1, \mathcal{P}^{h_2}_\ast E_2)\) 
by \( \mathcal{P}_\ast \Hom(E_1, E_2) \) (see Proposition \ref{p-prop14.3}). 
By Theorem \ref{p-thm16.2}, for any \( k \in \mathbb{R} \), we have
\[
\mathcal{P}^{h_1^\vee \otimes h_2}_k \Hom(E_1, E_2)
= \mathcal{P}^{h_1^\vee \otimes h_2}_k (E_1^\vee \otimes E_2)
= \sum_{b + c \leq k} \mathcal{P}^{h_1^\vee}_b E_1^\vee \otimes \mathcal{P}^{h_2}_c E_2.
\]

Let \( f \otimes u \in \mathcal{P}^{h_1^\vee}_b E_1^\vee 
\otimes \mathcal{P}^{h_2}_c E_2 \) with \( b + c \leq k \), and 
let \( x \in \mathcal{P}^{h_1}_a E_1 \).  
Then, by Proposition \ref{p-prop15.1}, we have  
\[
f(x) \in \mathcal{P}^{(0)}_{b + a} \mathcal{O}_\Delta(*[0]).
\]  
Therefore,
\[
(f \otimes u)(x) = f(x) \cdot u \in 
\mathcal{P}^{h_2}_{b + a + c} E_2 \subset \mathcal{P}^{h_2}_{a + k} E_2.
\]
This implies that, for every \( a \in \mathbb{R} \), we have  
\[
(f \otimes u)(\mathcal{P}^{h_1}_a E_1) \subset \mathcal{P}^{h_2}_{a + k} E_2,
\]
and hence \( f \otimes u \in \mathcal{P}_k \Hom(E_1, E_2) \).  
Thus, we obtain the inclusion
\[
\mathcal{P}^{h_1^\vee \otimes h_2}_k \Hom(E_1, E_2)
\subset \mathcal{P}_k \Hom(E_1, E_2).
\]

We now prove the opposite inclusion. 
Let  
\[
\mathcal{P}ar(\mathcal{P}^{h_1}_0 E_1) = \{b_1, \ldots, b_{r_1}\}, \quad
\mathcal{P}ar(\mathcal{P}^{h_2}_0 E_2) = \{c_1, \ldots, c_{r_2}\}.
\]  
Let \( \{v_1, \ldots, v_{r_1}\} \) and \( \{w_1, \ldots, w_{r_2}\} \) be 
frames of \( \mathcal{P}^{h_1}_0 E_1 \) and 
\( \mathcal{P}^{h_2}_0 E_2 \), respectively, compatible with the parabolic filtrations, such that  
\[
v_i \in \mathcal{P}^{h_1}_{b_i} E_1 \setminus \mathcal{P}^{h_1}_{< b_i} E_1,\quad 
w_j \in \mathcal{P}^{h_2}_{c_j} E_2 \setminus \mathcal{P}^{h_2}_{< c_j} E_2.
\] 
By Theorem \ref{p-thm13.2}, we have  
\[
\mathcal{P}^{h_1^\vee}_{1 - \epsilon} E_1^\vee \simeq (\mathcal{P}^{h_1}_0 E_1)^\vee
\] 
for sufficiently small $\epsilon > 0$.
Moreover, by Theorem \ref{p-thm13.2},  
\[
\mathcal{P}ar(\mathcal{P}^{h_1^\vee}_{1 - \epsilon} E_1^\vee) = \{-b_1, \ldots, -b_{r_1}\},
\]  
and the dual frame \( \{v_1^\vee, \ldots, v_{r_1}^\vee\} \) is 
compatible with this parabolic filtration. 
Fix an arbitrary \( k \in \mathbb{R} \), and let \( f \in \mathcal{P}_k \Hom(E_1, E_2) \).  
Then, for any \( a \in \mathbb{R} \), we have  
\[
f(\mathcal{P}^{h_1}_a E_1) \subset \mathcal{P}^{h_2}_{a + k} E_2.
\]  
In particular,
\[
f(\mathcal{P}^{h_1}_{b_i} E_1) \subset 
\mathcal{P}^{h_2}_{b_i + k} E_2 
\] 
for all $i$. 
Define  
\[
n_{ij} := \max \{ n \in \mathbb{Z} \mid n + c_j \leq b_i + k \}.
\]  
Then, by Lemma \ref{p-lem7.17}, the set 
\( \{z^{-n_{ij}} w_j\}_{j = 1}^{r_2} \) 
forms a frame of \( \mathcal{P}^{h_2}_{b_i + k} E_2 \).  
Since \( f(\mathcal{P}^{h_1}_{b_i} E_1) \subset \mathcal{P}^{h_2}_{b_i + k} E_2 \), 
there exist holomorphic functions \( f_{ij} \in \mathcal{O}_\Delta \) such that
\[
f = \sum_{i,j} f_{ij} \cdot v_i^\vee \otimes z^{-n_{ij}} w_j.
\] 
Since \( v_i^\vee \in \mathcal{P}^{h_1^\vee}_{-b_i} E_1^\vee \) 
and \( z^{-n_{ij}} w_j \in \mathcal{P}^{h_2}_{b_i + k} E_2 \), it follows that  
\[
v_i^\vee \otimes z^{-n_{ij}} w_j 
\in \mathcal{P}^{h_1^\vee}_{-b_i} 
E_1^\vee \otimes \mathcal{P}^{h_2}_{b_i + k} E_2 
\subset \sum_{b + c \leq k} \mathcal{P}^{h_1^\vee}_b E_1^\vee \otimes \mathcal{P}^{h_2}_c E_2 
= \mathcal{P}^{h_1^\vee \otimes h_2}_k \Hom(E_1, E_2).
\]
Therefore, \( f \in \mathcal{P}^{h_1^\vee \otimes h_2}_k \Hom(E_1, E_2) \), and hence
\[
\mathcal{P}_k \Hom(E_1, E_2) \subset \mathcal{P}^{h_1^\vee \otimes h_2}_k \Hom(E_1, E_2).
\]

Combining both inclusions, we conclude that
\[
\mathcal{P}_k \Hom(E_1, E_2) = \mathcal{P}^{h_1^\vee \otimes h_2}_k \Hom(E_1, E_2)
\quad \text{for all } k \in \mathbb{R}.
\]
This completes the proof of Proposition \ref{p-prop17.1}.
\end{proof}

%%%%%%%%%%%%%%%%%%%%

\end{document}